\UseRawInputEncoding
\documentclass{amsart}

\usepackage{amssymb,amsmath,amsthm}
\usepackage{amsaddr}
\usepackage{amsfonts}
\usepackage{authblk}
\usepackage{graphicx}
\usepackage{epstopdf}
\usepackage{caption}
\usepackage{subcaption}
\usepackage{hyperref}
\usepackage{color}
\usepackage{marginnote}
\usepackage{float}
\usepackage{graphicx}
\usepackage{comment}
\usepackage{tikz-cd}
\usetikzlibrary{arrows}
\usetikzlibrary{intersections,positioning}
\tikzset{>=latex}
\usepackage{lipsum}%
\usepackage{a4wide}

\usepackage[letterpaper,top=2cm,bottom=2cm,left=3cm,right=3cm,marginparwidth=1.75cm]{geometry}

\allowdisplaybreaks

\newcounter{newcounter}[section]
\numberwithin{equation}{section}
\numberwithin{newcounter}{section}
\numberwithin{figure}{section}
\numberwithin{footnote}{section}

\newtheorem{thm}[newcounter]{Theorem}
\newtheorem{defi}[newcounter]{Definition}
\newtheorem{prop}[newcounter]{Proposition}
\newtheorem{lem}[newcounter]{Lemma}
\newtheorem{cor}[newcounter]{Corollary}
\newtheorem{rem}[newcounter]{Remark}

\newcommand{\fe}{\mathfrak{e}}
\newcommand{\fn}{\mathfrak{n}}

\newcommand{\R}{\mathbb{R}}
\newcommand{\C}{\mathbb{C}}
\newcommand{\Z}{\mathbb{Z}}
\newcommand{\D}{\mathbb{D}}
\newcommand{\N}{\mathbb{N}}
\newcommand{\Q}{\mathbb{Q}}

\title[ECH constraints and twist dynamics in the spatial isosceles three-body problem]{ECH constraints and twist dynamics in the spatial isosceles three-body problem}
\date{\today}

\begin{document}
\maketitle

\begin{center}
\normalsize
Xijun Hu\textsuperscript{1}, Lei Liu\textsuperscript{1}, Yuwei Ou\textsuperscript{1},
Zhiwen Qiao\textsuperscript{1} and Pedro A. S. Salom\~ao\textsuperscript{2} \par \bigskip

\textsuperscript{1} School of Mathematics, Shandong University  \par
\textsuperscript{2} Shenzhen International Center for Mathematics, SUSTech  \par

\end{center}

\begin{abstract}
We study dynamical constraints arising from Embedded Contact Homology (ECH) in the spatial isosceles three-body problem. For energies below the critical
level, the dynamics on the energy surface is identified with a Reeb flow on the tight three-sphere.
We obtain quantitative estimates for the Euler orbit, including monotonicity of its transverse rotation number and a strict inequality comparing its action with the contact volume. Combined with the ECH classification of Reeb flows on the tight three-sphere with two simple periodic orbits, these estimates rule out the two-orbit scenario, thus forcing every compact energy surface below the critical level to have infinitely many periodic orbits.
The result admits a dynamical interpretation via disk-like global surfaces of section bounded by the Euler orbit. In this setting, the rotation number and the contact volume define a non-trivial twist interval which encodes the relative winding of periodic orbits.
For energies above the critical level, where the energy surface is non-compact, we prove the existence of infinitely many periodic orbits and infinitely many parabolic trajectories via twist estimates near infinity.
\end{abstract}

\tableofcontents

\section{Introduction and main results}

The spatial isosceles three-body problem is an important model in celestial mechanics.
It describes the motion of three bodies under Newton¡¯s law of gravitation,
where two bodies have equal masses and move symmetrically around a fixed axis,
and the third body moves along this axis.
The system is far from integrable, and its non-trivial dynamical structure exhibits mechanisms that also arise in more complex celestial mechanics models.

Let $\alpha>0$ be the mass ratio between the first two bodies and the third body, and let $\varpi>0$ be the angular momentum of the three-body system. Fixing the mass center at the origin, the Hamiltonian of the spatial isosceles three-body problem in reduced form is given by
\begin{equation}\label{equ: Ham1}
H(p_r,p_z,r,z)=\frac{p_r^2+p_z^2}{2}+V(r,z),\quad V(r,z):=\frac{\varpi^2}{2r^2}-\frac{1}{r}-\frac{4}{\alpha\sqrt{r^2+(1+2\alpha)z^2}},
\end{equation}
where $(r,z)\in\mathbb{R}_+\times \mathbb{R}$ are suitable coordinates determined by the first two bodies, and $(p_r,p_z)\in \mathbb{R}^2$ are the corresponding momenta, see \cite{HLOSY2023}. Hamilton's equations for $H$ are then given by
\begin{equation}\label{equ: Ham system}
\left\{\begin{aligned}
 \dot r = p_r, \qquad \dot p_r & = \frac{\varpi^2}{r^3} -\frac{1}{r^2} -\frac{4r}{\alpha\sqrt{(r^2+(1+2\alpha)z^2)^3}}, \\
\dot z  = p_z, \qquad \dot p_z &   =  -\frac{4(1+2\alpha)z}{\alpha\sqrt{(r^2+(1+2\alpha)z^2)^3}}.
\end{aligned}
\right.
\end{equation}

The dynamics of the spatial isosceles three-body problem on the energy surface $H^{-1}(h), h<0,$ is very rich. The restricted case,  corresponding to  $\alpha=+\infty$, was studied by Sitnikov \cite{Sitnikov1960}, Alekseev \cite{Ale69}, and Moser \cite{Ms73}. Alekseev \cite{Alekseev1972} and Moeckel \cite{Moeckel1984} later considered large values of $\alpha$ and small values of $\varpi$, finding a diversity of complex motions via symbolic dynamics.  Shibayama \cite{Shibayama2009} studied periodic orbits using variational methods and established the KAM stability of the Euler orbit for every $\mathfrak e:=\sqrt{1+2h\varpi^2\beta^2}\in(0,1)$ sufficiently small. Here, $\beta:=1/(1+4/\alpha)\in(0,1)$. A quantitative refinement was recently given in \cite{C-RLS2025}, where non-resonant Hopf links were proved to exist for $\mathfrak e>0$ sufficiently small. Hu, Ou, and Tang \cite{HOT2023} studied the stability of the Euler orbit for every $\mathfrak e$ sufficiently close to $1$ using a blow-up argument. Hu, Qiao, and Yu \cite{HQY25} later considered $\beta>0$ sufficiently small to study the problem as a perturbation of the spatial Kepler problem. In the general setting, abstract results based on pseudo-holomorphic curves were applied in \cite{HLOSY2023} to obtain global surfaces of section, Hopf links and  periodic orbits.

We aim to study the dynamics of the mechanical Hamiltonian \eqref{equ: Ham1} assuming $\alpha,\varpi>0$ and energy $h<0$, or equivalently assuming $(\beta,\mathfrak{e})\in (0,1)\times (0,1)$. Notice that $\varpi^2h$ is an essential parameter of \eqref{equ: Ham1}. Hence, we may fix $h=-1$ and study the dynamics on the energy surface $\mathfrak M:=H^{-1}(-1)\subset \R^4$ without loss of generality. From Moeckel \cite{Moeckel1984}, see also \cite[Proposition 3.1]{HLOSY2023}, the energy surface $\mathfrak M$ is diffeomorphic to the three-sphere for every $(\beta,\mathfrak e)$ satisfying $0<\beta^2+\mathfrak e^2<1$, and diffeomorphic to $S^2\times \mathbb R$ for every $(\beta,\mathfrak e)\in (0,1)\times (0,1)$ satisfying $\beta^2+\mathfrak e^2>1$. The case $\beta^2+\mathfrak e^2=1$ is called critical, and $\mathfrak M$ is unbounded in the $z$-direction, diffeomorphic to $S^2 \times \R$, and approaches the limiting points $(p_r,p_z,r,z)= (0,0,\varpi^2,\pm\infty)$ at infinity.

A classical periodic orbit of \eqref{equ: Ham1} is the Euler orbit given by
$$
\zeta_e = \mathfrak{M} \cap \{z=p_z=0\},
$$
corresponding to the aligned motion of the three-bodies.
Notice that $\zeta_e$ exists for every $(\beta,\mathfrak{e}) \in (0,1) \times (0,1)$ and is a brake-orbit, i.e., it touches the circle-like boundary of the Hill region $V^{-1}((-\infty,-1])$ twice along its prime period, precisely at those two points in $z=0$ where $V(r,0)=-1$. It is proved in \cite{HLOSY2023} that the (transverse) rotation number of the Euler orbit satisfies $\rho_e\in (2, +\infty)$, and that $\rho_e$ plays a decisive role in the existence of other periodic orbits and their symmetries in the Hill region. In suitable coordinates,  $\rho_e$ is determined by integrating the following Ince-type Hill equation
\begin{equation}\label{Hilleq}
\ddot x(\theta) = -\left( 1+ \frac{7\beta}{1+\mathfrak{e} \cos \theta} \right) x(\theta), \quad \theta \in \R / 2\pi \Z.
\end{equation}
Studying this equation and determining the properties of $\rho_e$ as a function of the parameters $(\beta,\mathfrak{e})$ is a challenging problem and still an object of extensive research, see \cite{HLOSY2023,HOT2023,HLS14,MaWi66}. The monotonicity of $\rho_e$ with respect to $\beta$ is straightforward to prove, see \cite{HLOSY2023, HOT2023}. However, the monotonicity of $\rho_e$ with respect to $\mathfrak{e}$ is a subtle question. Our first result is the following theorem proving the monotonicity of $\rho_e$ with respect to $\mathfrak e$.

\begin{thm}\label{thm: main1}
 For every $(\beta,\mathfrak{e})\in (0,1) \times [0,1)$ denote the rotation number of the Euler orbit $\zeta_e$ by $\rho_{\beta,\mathfrak e}:=\rho_e$. Then for every $\beta\in (0,1)$ the function $\mathfrak{e} \mapsto \rho_{\beta,\mathfrak{e}}, \mathfrak{e} \in [0,1),$ is non-decreasing. Moreover, $$\rho_{\beta,\mathfrak e}>\rho_{\beta,0}=1+\sqrt{1+7 \beta}, \qquad \forall (\beta,\mathfrak{e})\in (0,1)\times (0,1).$$
\end{thm}

The proof of Theorem \ref{thm: main1} is based on Fourier analysis and on Morse-type estimates
for the operator associated with the linear differential equation \eqref{Hilleq}.
This result plays a central role in the proof of Theorem \ref{thm: main2} below and its applications.

In \cite[Theorem 2.3]{HLOSY2023}, the authors prove that for every $(\beta,\mathfrak{e}) \in (0,1) \times (0,1)$ satisfying $0<\beta^2+\mathfrak e^2<1$, the Euler orbit $\zeta_e\subset \mathfrak{M}\equiv S^3$ has rotation number $\rho_e>2$ and is the binding of an open book decomposition whose pages are disk-like global surfaces of section. In particular, $\mathfrak{M}$ carries either two or infinitely many periodic orbits.
It is also proved in \cite{HLOSY2023} that for a dense set of parameters $(\beta,\mathfrak{e})$, $\mathfrak{M}$ contains infinitely many periodic orbits. The question of whether $\mathfrak{M}$ always admits infinitely many periodic orbits was left open in \cite{HLOSY2023}.

Since $\mathfrak M\equiv S^3$ is a regular compact energy surface of the mechanical Hamiltonian \eqref{equ: Ham1}, there exists a smooth contact form $\lambda$ on $\mathfrak M$ so that its Reeb vector field $R$ is parallel to the Hamiltonian vector field. This contact form satisfies $d\lambda = \tilde \omega_0|_{\mathfrak M}$, where $\tilde \omega_0=dp_r\wedge dr + dp_z \wedge dz$ is the standard symplectic form on $\R^4$. Moreover, the contact structure $\xi=\ker \lambda$ is tight, and thus the Hamiltonian flow on $\mathfrak{M}\equiv S^3$ is equivalent to a Reeb flow on the tight three-sphere with contact form also denoted $\lambda$.

In \cite{C-GHHL2023}, Cristofaro-Gardiner, Hryniewicz, Hutchings, and Liu used tools from ECH to obtain the following beautiful result on Reeb dynamics: if a Reeb flow on the tight three-sphere $(S^3,\lambda)$ admits precisely two geometrically distinct simple periodic orbits, say $\gamma_1$ and $\gamma_2$, then the following algebraic relations between their rotation numbers and Reeb periods are verified
\begin{equation}\label{equ: relation in C-GHHL2023}
\mathrm{vol}(S^3,\lambda)=\frac{T_1^2}{\rho_1-1}=\frac{T_2^2}{\rho_2-1}=T_1T_2\quad \text{and}\quad \rho_1,\rho_2\in (1,+\infty) \setminus \mathbb Q.
\end{equation}
Here, $\mathrm{vol}(S^3,\lambda) :=\int_{S^3}\lambda\wedge d\lambda>0$ is the contact volume of $S^3$ with respect to the contact form $\lambda$.
It is also known in this case that $\gamma_1$ and $\gamma_2$ bind open book decompositions whose pages are disk-like global surfaces of section, and the first return map is a pseudo-rotation, i.e., it contains precisely one fixed point and no other periodic point. The two or infinitely many dychotomy was futher discussed in \cite{C-GHHL2024, C-GHP}.

It is one of the goals of this paper to show that ${\rm vol}(\mathfrak{M}) \neq T_{\beta,\mathfrak{e}}^2/(\rho_{\beta,\mathfrak{e}}-1)$ and thus conclude once and for all that $\mathfrak{M}$ admits infinitely many periodic orbits for every $(\beta,\mathfrak{e})\in (0,1) \times (0,1)$ satisfying $0<\beta^2+\mathfrak{e}^2<1$. Here, $T_{\beta,\mathfrak{e}}=T_e$  is the Reeb period of $\zeta_e$, which coincides with its symplectic action, and $\rho_{\beta,\mathfrak{e}}=\rho_e$ is the rotation number of $\zeta_e\subset \mathfrak{M}=\mathfrak{M}_{\beta,\mathfrak{e}}$. Notice that both $T_{\beta,\mathfrak{e}}$ and ${\rm vol}(\mathfrak{M})$ do not depend on the choice of $\lambda$.

Our second result combines Theorem \ref{thm: main1} and an estimate on ${\rm vol}(\mathfrak{M})$ to show that conditions in \eqref{equ: relation in C-GHHL2023} fail and thus $\mathfrak{M}$ cannot admit precisely two periodic orbits.

\begin{thm} \label{thm: main2}
For every $(\beta,\mathfrak{e})\in (0,1)\times (0,1)$ satisfying $0<\beta^2+\mathfrak{e}^2<1$, we have
\begin{equation}\label{ineq_vol}
{\rm vol}(\mathfrak{M}) > \frac{T_{\beta,\mathfrak e}^2}{\sqrt{1+7\beta}} > \frac{T_{\beta,\mathfrak{e}}^2}{\rho_{\beta,\mathfrak{e}}-1}, \qquad T_{\beta,\mathfrak e} = 2\pi \left(\frac{1-\sqrt{1-\mathfrak e^2}}{\sqrt{2} \beta} \right).
\end{equation}
\end{thm}

\begin{cor}\label{cor_infinitely_many}
    For every $(\beta,\mathfrak{e})\in (0,1)\times (0,1)$ satisfying $0<\beta^2+\mathfrak{e}^2<1$,  $\mathfrak M$ admits infinitely many periodic orbits.
\end{cor}

Corollary \ref{cor_infinitely_many} can be seen as a first step towards proving the following question posed by Marco \cite[Problem 6]{ACS12}:  Is the topological entropy of the Hamiltonian flow on $\mathfrak{M}$ positive?

Another important corollary of Theorem \ref{thm: main2} refers to the systolic ratio of the energy surface $\mathfrak M$. Let $T_{\rm min}(\mathfrak M)$ be the action of the periodic orbit of $\mathfrak M$ with the smallest action. This minimum is always realized by some periodic orbit, called systole, due to the Arzel\'a-Ascoli Theorem, and does not depend on the choice of $\lambda$. The systolic ratio of $\mathfrak M$ is defined as
$$
\rho_{\rm sys}(\mathfrak M):= \frac{(T_{\rm min}(\mathfrak M))^2}{{\rm vol}(\mathfrak M)}.
$$
It is proved in \cite{ABHS2018} that for contact forms sufficiently $C^3$-close to the standard contact form on the tight three-sphere, the systolic ratio is bounded from above by $1$. It is an open question whether the systolic ratio is bounded from above in the space of dynamically convex contact forms on the tight three-sphere \cite{ABHS2018b}.

\begin{cor}\label{cor_systolic}
    For every $(\beta,\mathfrak{e})\in (0,1)\times (0,1)$ satisfying $\beta^2+\mathfrak{e}^2<1$, the systolic ratio of $\mathfrak M$ satisfies $\rho_{\rm sys}(\mathfrak M) < \sqrt{1+7\beta} < 2\sqrt{2}.$
\end{cor}

Another application of Theorem \ref{thm: main2} concerns the $\mathrm{ECH}$ spectral invariants.
The estimate on the contact volume obtained in Theorem \ref{thm: main2}
enters directly into the Weyl law for the $\mathrm{ECH}$ spectrum and
provides quantitative information on the growth of the spectral invariants.

We briefly recall the necessary notation.
Let $f_n:\mathfrak M\to \R_{>0}$, $n\in \N$, be a sequence such that
$f_n\to 1$ in the $C^0$-topology and $\lambda_n=f_n\lambda$ is nondegenerate.
Let $\mathrm{ECC}_*(\mathfrak M,\lambda_n)$ denote the $\Z_2$-vector space
freely generated by $\mathrm{ECH}$-generators, endowed with a $\Z$-grading.
An $\mathrm{ECH}$-generator is a finite formal sum
$\alpha=\sum_i m_i\zeta_i$, where $m_i\in\N$, $\zeta_i\in\mathcal P_{\lambda_n}$,
and $m_i=1$ whenever $\zeta_i$ is hyperbolic.

For a generic $\lambda_n$-compatible almost complex structure $J_n$,
one defines a differential on $\mathrm{ECC}_*(\mathfrak M,\lambda_n)$,
obtaining a chain complex whose homology is
$\mathrm{ECH}_*(\mathfrak M,\xi)$, which depends only on
$\xi=\ker\lambda$ by Taubes~\cite{Taubes2010}. In particular, since $(\mathfrak M,\xi)$ is contactomorphic to the tight three-sphere, we have
$$
\mathrm{ECH}_*(\mathfrak M,\xi)=
\begin{cases}
\mathbb Z_2, & *=0,2,4,\dots,\\
0, & \text{otherwise}.
\end{cases}
$$
The empty set defines a nontrivial class
$[\emptyset]\in \mathrm{ECH}_0(\mathfrak M,\xi)$,
and there is a $U$-map
$U:\mathrm{ECH}_*(\mathfrak M,\xi)\to\mathrm{ECH}_{*-2}(\mathfrak M,\xi)$.

Let $\mathcal A(\alpha)=\sum_i m_i T(\zeta_i)$ denote the action of
an $\mathrm{ECH}$-generator $\alpha$.
The filtered complex $\mathrm{ECC}^a_*(\mathfrak M,\lambda_n)$
generated by orbit sets of action less than $a>0$
induces the filtered homology $\mathrm{ECH}^a_*(\mathfrak M,\lambda_n)$.
The $k$-th $\mathrm{ECH}$ spectral invariant is defined by
$$
c_k(\mathfrak M,\lambda)
:=\lim_{n\to+\infty} c_k(\mathfrak M,\lambda_n),
$$
where
$$
c_k(\mathfrak M,\lambda_n)
:=\inf\left\{a>0:\sigma\in \mathrm{ECH}^a_*(\mathfrak M,\lambda_n),\
U^k\sigma=[\emptyset]\right\}.
$$
Each $c_k(\mathfrak M,\lambda)$ is realized as the action of a finite sum $\alpha_k=m_{k,0}\zeta_e+\sum_i m_{k,i}\zeta_{k,i},$ where $(m_{k,i},\zeta_{k,i})\in \N\times (\mathcal P_\lambda\setminus \{\zeta_e\})$.
The Weyl law for the $\mathrm{ECH}$ spectrum \cite{C-GHR} reads as
\begin{equation}\label{eq:weyl}
\lim_{k\to\infty}\frac{c_k(\mathfrak M,\lambda)}{\sqrt{2k}}={\rm vol}(\mathfrak M).
\end{equation}

Combining \eqref{eq:weyl} with the explicit lower bound for the contact volume obtained in
Theorem~\ref{thm: main2}, we obtain the following corollary.

\begin{cor}\label{coro: Weyl_law} Given $(\beta,\fe)\in (0,1)\times (0,1)$ satisfying $\beta^2 + \fe^2<1$, and $\varepsilon>0$, there exists  $k_0=k_0(\varepsilon,\beta,\mathfrak e)\in\mathbb{N}$
such that  for all $k\ge k_0$,
\begin{equation}
c_k(\mathfrak M,\lambda)\geq ({\rm vol}(\mathfrak M)-\varepsilon)\sqrt{2k}
 >\left(\frac{T_{\beta,\fe}^2}{\sqrt{1+7\beta}} -\varepsilon\right)\sqrt{2k}.
\end{equation}
In particular, the rate of growth of $c_k(\mathfrak M,\lambda)$ is uniformly bounded from below on compact subsets
of the subcritical parameter region $\{(\beta,\fe)\in(0,1)\times (0,1): \beta^2+\fe^2<1\}$.
\end{cor}

To relate the ECH spectral invariants with the linking behavior of periodic
orbits around the binding $\zeta_e$, we introduce an auxiliary quantity associated
with the orbit sets realizing the spectral invariants. We define
$$
F(\alpha_k)
:= m_{k,0}(\rho_{\beta,\fe} - 1)
   + \sum_i m_{k,i}  {\rm lk}(\zeta_{k,i}, \zeta_e),
$$
where, as above, the orbit set $\alpha_k=m_{k,0}\zeta_e + \sum_i m_{k,i}\zeta_{k,i}$ realizes $c_k(\mathfrak M, \lambda)$ for every $k$. The quantity $\mathcal F(\alpha_k)$ combines the multiplicity of the Euler orbit with the
linking numbers of the remaining components relative to $\zeta_e$. Its relevance
in our setting lies in the fact that $\mathcal F(\alpha_k)$ is non-decreasing in $k$, see \cite[Lemma 5.1]{Hutchings2016}  providing monotonic control on how linking accumulates among orbit sets realizing the ECH spectrum.

We now return to the interval determined by Theorem \ref{thm: main2} and investigate
its dynamical meaning. We use tools from ergodic theory
to understand how this interval relates to the dynamics of the return map  and in the organization of periodic orbits.
The inequality \eqref{ineq_vol} determines the non-trivial compact interval
\begin{equation}\label{interval}
I_{\beta,\mathfrak{e}}:= \left[\frac{1}{\rho_{\beta,\mathfrak e}-1}, \frac{{\rm vol}(\mathfrak M)}{T_{\beta,\mathfrak{e}}^2}\right]\subset \R.
\end{equation}

Notice that $I_{\beta,\fe}$ is a geometrically interesting interval since it is invariant under rescaling of the contact form $\lambda \mapsto c\lambda, c\neq 0$. The right endpoint arises from the contact volume, which controls the asymptotic growth of the ECH spectrum via the Weyl law. The left endpoint is determined by the transverse rotation number of the binding orbit $\zeta_e$ and represents the local
linearized dynamics. In this sense, $I_{\beta,\fe}$ measures the discrepancy between the rotation number of the binding and the contact volume constraint, thus its size quantifies the deviation from the two¨Corbit configuration. As we will see, this interval provides a dynamical mechanism linking ECH spectral constraints with the open book structure organizing the flow.

Now we discuss the relation between $I_{\beta,\mathfrak{e}}$ and the periodic orbits of $\mathfrak{M}$. More precisely, we investigate if  $I_{\beta,\mathfrak{e}}$ corresponds to a twist interval explaining the periodic orbits in Theorem \ref{thm: main2}. In order to do that, we need to recall the open book decomposition $(B,\Phi)$ of $\mathfrak M$ considered in \cite{HLOSY2023}, whose binding $B=\zeta_e$ is the Euler orbit and the pages $\Sigma_s:=\Phi^{-1}(s), s\in \R / \Z$, are disk-like global surface of section. Moreover, $\Sigma_s\subset \mathfrak M \setminus \zeta_e$ is an embedded open disk bounded by $\zeta_e$ whose projection to the $(p_z,z)$-plane is a line-segment through the origin with argument $2\pi s$, that is,
$$
\Sigma_s:=\mathfrak{M} \cap \{(p_z,z)\neq (0,0), \arg(p_z+iz) = 2\pi s\},\quad s\in \R/\Z.
$$
Each $\Sigma_s$ is in one-to-one correspondence with its projection to the $(p_r,r)$-plane, given by
\begin{equation}\label{equ: Upsilon}
\dot \Upsilon:=\left\{(p_r,r)\in \mathbb R\times \mathbb R_+:\ \frac{p_r^2}{2}+\frac{\varpi^2}{2r^2}-\frac{1}{\beta r}< -1\right\}.
\end{equation}
Notice that $\dot \Upsilon$ does not depend on $s$ and the boundary of $\dot \Upsilon$ is precisely the simple closed curve $\partial \Upsilon$ determined by the projection of $\zeta_e$ to the $(p_r,r)$-plane. Also, the symplectic form restricted to $\Sigma_s$ becomes $dp_r \wedge dr$ on $\Upsilon$, and thus the Reeb period $T_e>0$ of $\zeta_e$ coincides with the area of $\dot \Upsilon$.

Given $0\leq s_1 \leq s_2 \leq 1$, there exists a smooth function $\tau_{s_1,s_2}:\Sigma_{s_1}: \to [0,+\infty)$ and a smooth map $\psi_{s_1,s_2}:\Sigma_{s_1} \to \Sigma_{s_2}$, both depending smoothly on $(s_1,s_2)$,  given by the first hitting time and the first hitting map from $\Sigma_{s_1}$ to $\Sigma_{s_2}$ along the flow on $\mathfrak M$, and satisfying $\psi_{s,s} = {\rm Id}$ for every $s\in [0,1]$ and $\psi_{0,1}$ is the first return map associated with $\Sigma_0$. These maps naturally extend for every $s_1\leq s_2 \in \R$.  In particular, the maps $\psi_{0,s}: \Sigma_0 \to \Sigma_s, s\in \R,$ induce a smooth family of area-preserving diffeomorphisms
$$
\psi_s: (\dot \Upsilon, dp_r \wedge dr) \to (\dot \Upsilon, dp_r \wedge dr), \qquad s\in \R,
$$
representing the flow on $\mathfrak M$ and satisfying $\psi_0 = {\rm Id}$. Such maps continuously extend to $\Upsilon := \dot \Upsilon \cup \partial \Upsilon$ due to the linearized dynamics near $\zeta_e$. Here, we shall prove that the family $\psi_s:\Upsilon \to \Upsilon,s\in \R,$ is a smooth family of symplectic diffeomorphisms preserving the area form $dp_r \wedge dr$.

Given $x_1\neq x_2 \in \Upsilon$ we define the mean relative winding number of $(x_1,x_2)$ as
\begin{equation}\label{equ: winding number}
w_\infty(x_1,x_2):=\lim_{s\rightarrow +\infty}\frac{\theta(s)-\theta(0)}{s},
\end{equation}
where $2\pi\theta(s)$ denotes a continuous argument of $\psi_s(x_1) - \psi_s(x_2)$ for every $s$.  By Birkhoff's ergodic theorem, $w_\infty$ is an integrable function almost everywhere defined on $(\Upsilon \times \Upsilon) \setminus \{\rm diagonal\}$. If either $x_1$ or $x_2$ lies in $\partial \Upsilon$, then $w_\infty(x_1,x_2)=(\rho_{\beta,\fe}-1)^{-1}$ is the left endpoint of $I_{\beta,\mathfrak e}$.

Given a periodic orbit $\zeta\subset \mathfrak M\setminus\zeta_e$, we denote by $\mathcal O(\zeta)\subset \dot \Upsilon$ the finite set of points in $\dot \Upsilon$ given by the projection of  $\zeta\cap \Sigma_0$ to the $(p_r,r)$-plane. If $\zeta_1, \zeta_2 \subset \mathfrak M \setminus \zeta_e$ are periodic orbits (not necessarily geometrically distinct) and $x_1\in \mathcal{O}(\zeta_1)\neq  x_2\in \mathcal{O}(\zeta_2)$, then $w_\infty(x_1,x_2) \in \Q.$

The next theorem shows that any rational number in the interior of $I_{\beta,\mathfrak e}$ can be realized as the mean relative winding number of points associated with periodic orbits in $\mathfrak M \setminus \zeta_e$.

\begin{thm} \label{thm: main4}
Let $I_{\beta,\mathfrak e}=[(\rho_{\beta,\mathfrak e}-1)^{-1}, T_{\beta,\mathfrak e}^{-2}{\rm vol}(\mathfrak M)]$ be the non-trivial interval determined by Theorem \ref{thm: main2}. Given a rational number $\mathfrak r$ in the interior of $I_{\beta,\mathfrak e}$, there exist geometrically distinct periodic orbits $\zeta_1,\zeta_2\subset \mathfrak M \setminus \zeta_e$, and points $x_1\in \mathcal{O}(\zeta_1) \neq x_2\in \mathcal{O}(\zeta_2)$ in $\dot \Upsilon$, such that $w_\infty(x_1,x_2)= \mathfrak r$.
\end{thm}

The proof of Theorem \ref{thm: main4} relies on Hutchings' mean action theorem \cite{Hutchings2016} and its generalization by Pirnapasov \cite{Pirnapasov2021}. In order to apply such results, we first show the crucial fact that $\psi_{s}: \Upsilon \to \Upsilon, s\in \R,$ is a smooth family of symplectic diffeomorphisms preserving the area-form $dp_r \wedge dr$. Applying the results in \cite{Hutchings2016, Pirnapasov2021}, we obtain a simple periodic orbit $\zeta_1 \subset \mathfrak M \setminus \zeta_e$ whose mean action satisfies $T(\zeta_1)/ k_1> ({\rm vol}(\mathfrak M)-\epsilon)/ T_{\beta,\mathfrak e}$, for any given $\epsilon>0$ small. Here, $T(\zeta_1)$ is the action (Reeb period), and $k_1$ is the linking number ${\rm lk}(\zeta_1,\zeta_e)$ which coincides with the prime period of any $x_1\in \mathcal{O}(\zeta_1)$ under $\psi_1$. Then we generalize a result by Bechara \cite{Bechara2023} implying that there exists $\hat x_2\in \dot \Upsilon$, not necessarily periodic, such that $w_\infty(x_1,\hat x_2)> ({\rm vol}(\mathfrak M)-\epsilon)/T_{\beta,\mathfrak e}^2>(\rho_{\beta,\fe}-1)^{-1}$ for some $x_1\in \mathcal{O}(\zeta_1)$. Due to the rotation number $(\rho_{\beta,\fe}-1)^{-1}$ along $\partial \Upsilon$, we apply Franks' generalization of the Poincar\'e-Birkhoff Theorem \cite{Franks03} to the map $\psi_1^{k_1}$ to find a periodic orbit $\zeta_2$ so that $w_\infty(x_1,x_2)\in ((\rho_{\beta,\mathfrak e}-1)^{-1},T_{\beta,\fe}^{-2}({\rm vol}(\mathfrak M)-\epsilon))$ realizing the given rational number in $I_{\beta,\fe}$ for some $x_2 \in \mathcal{O}(\zeta_2)$.  Theorem \ref{thm: main4} then follows by taking $\epsilon \to 0^+$.

\begin{rem}\label{rem: pairing and w_infty} In \cite{BHS2021}, the authors define
\begin{equation}\label{equ: pairing}
\rho(\zeta_1,\zeta_2):=\frac{\mathrm{lk}(\zeta_1,\zeta_2)\mathrm{vol}(\mathfrak M)}{T(\zeta_1)T(\zeta_2)},
\end{equation}
where $\zeta_1,\zeta_2 \subset \mathfrak{M}$ are geometrically distinct periodic orbits.
They ask if $\inf_{\zeta_1 \subset \mathfrak{M} \setminus \zeta_2} \rho(\zeta_1,\zeta_2) \leq 1$ for any given periodic orbit $\zeta_2$.  This fact has been proved for a wide class of Reeb flows on the tight three-sphere, including toric domains (see Question 5 and Remark 1.4 in \cite{BHS2021}).
The proof of Theorem~\ref{thm: main4} implies that $\inf_{\zeta \subset \mathfrak{M}\setminus \zeta_e} \rho (\zeta, \zeta_e) \leq 1.$

Let $\zeta_1,\zeta_2 \subset \mathfrak M\setminus \zeta_e$ be geometrically distinct periodic orbits. Notice that the identity $w_\infty(x_1,x_2)= w_\infty(x_1',x_2')$ does not always hold for every $x_1,x_1'\in \mathcal{O}(\zeta_1)$ and $x_2,x_2'\in \mathcal{O}(\zeta_2)$. In fact, it is possible to show that the following relation holds
\begin{equation}\label{equ: winding and linking}
\begin{aligned}
\sum_{l_1=0}^{k_1-1}\sum_{l_2=0}^{k_2-1}w_\infty(\psi^{l_1}(x_1),\psi^{l_2}(x_2))=\mathrm{lk}(\zeta_1,\zeta_2),
\end{aligned}
\end{equation}
where $k_1,k_2>0$ are the prime periods of $x_1,x_2$ under $\psi=\psi_1$, respectively. Moreover, if $\mathrm{gcd}(k_1,k_2)=1$, then
$w_\infty(x_1,x_2)=\mathrm{lk}(\zeta_1,\zeta_2)/(k_1k_2)$ is independent of $x_1\in \mathcal{O}(\zeta_1),x_2\in \mathcal{O}(\zeta_2)$. See the Appendix \ref{sec: wind and link} for a discussion.
\end{rem}

Finally, we consider parameters $(\beta,\mathfrak{e}) \in (0,1) \times (0,1)$ satisfying $\beta^2 + \mathfrak{e}^2>1$. In this case, the energy surface $\mathfrak{M}$ is unbounded in the $z$-direction and is diffeomorphic to $S^2\times \mathbb R$. The Euler orbit $\zeta_e$ is no longer the boundary of a global surface of section due to the existence of trajectories that escape to infinity.

Recall that the rotation number of the Euler orbit satisfies $\rho_e>2$. Consider the family $\Sigma_s=\mathfrak M\cap \{\mathrm{arg}(p_z+iz)=2\pi s\}, s\in \mathbb R/\mathbb Z$. These surfaces determine a singular foliation of $\mathfrak M$ with binding $\zeta_e$,  consisting of two families of disks $\Sigma_s,s\in (-1/4,1/4)$, corresponding to $p_z>0$, and $\Sigma_s,s\in (1/4,3/4)$, corresponding to $p_z<0$. These families are separated by the cylinders $\Sigma_{1/4}$ and $\Sigma_{-1/4}$, corresponding to $p_z=0$, which are bounded by the Euler orbit $\zeta_e$ and the orbits $\zeta_{\pm\infty}$ at $z=\pm \infty$, respectively, see Figure \ref{fig: foliations}.

Since $\ddot z = -\partial_zV =-g(r,z)z,$ where $g(r,z):=\frac{4(1+2\alpha)}{\alpha\sqrt{(r^2+(1+2\alpha)z^2)^{3}}}>\delta>0$, the continuous argument $\theta(t)=\mathrm{arg}(p_z(t)+iz(t))$ solves the equation
$$
\dot \theta(t)=\cos^2\theta(t) +g(r(t),z(t))\sin^2\theta(t)>\min\{1,\delta\}>0.
$$
Hence, each surface $\Sigma_{s}$ is transverse to the Hamiltonian flow.

As $z$ approaches $\pm \infty$, the dynamics on $\mathfrak M$ is approximated by the dynamics of the decoupled $z$-invariant Hamiltonian
\begin{equation}\label{equ: decoupled Ham}
H_\infty(p_r,p_z,r,z)=\frac{p_r^2+p_z^2}{2}+\frac{\varpi^2}{2r^2}-\frac{1}{r},
\end{equation}
restricted to the energy surface $H_\infty^{-1}(-1)\cong S^2 \times \R$. The two-sphere $\mathcal S_{\pm \infty}:=H_\infty^{-1}(-1)/ \R$ is called the two-sphere at $z=\pm \infty$. It contains a periodic orbit $\zeta_{\pm \infty}:=\mathcal S_{\pm\infty}\cap \{p_z=0\}$, called the periodic orbit of $\mathfrak{M}$ at $z=\pm \infty$.

The whole family of planes and cylinders $\Sigma_s\subset \mathfrak M , s\in \R / \Z$, and the spheres at infinity $\mathcal S_{\pm \infty}$, form a foliation of $\mathfrak M \cup \mathcal S_{\pm \infty}$ which resembles the finite energy foliations introduced by Hofer, Wysocki and Zehnder in \cite{fols} and further investigated in \cite{dPS1, dPHKS, dPKSS, LS2025}. Here, the (degenerate) orbits $\zeta_{\pm \infty}$ at infinity play the role of the index-$2$ hyperbolic orbits of the usual finite energy foliations, while the hemispheres in $\mathcal S_{\pm\infty}\setminus \zeta_{\pm \infty}$ at infinity and the cylinders $\Sigma_{\pm 1/4}$ play the role of the rigid leaves. See Figure \ref{fig: foliations}.

\begin{figure}[hpbt]
    \centering    \includegraphics[width=0.5\linewidth]{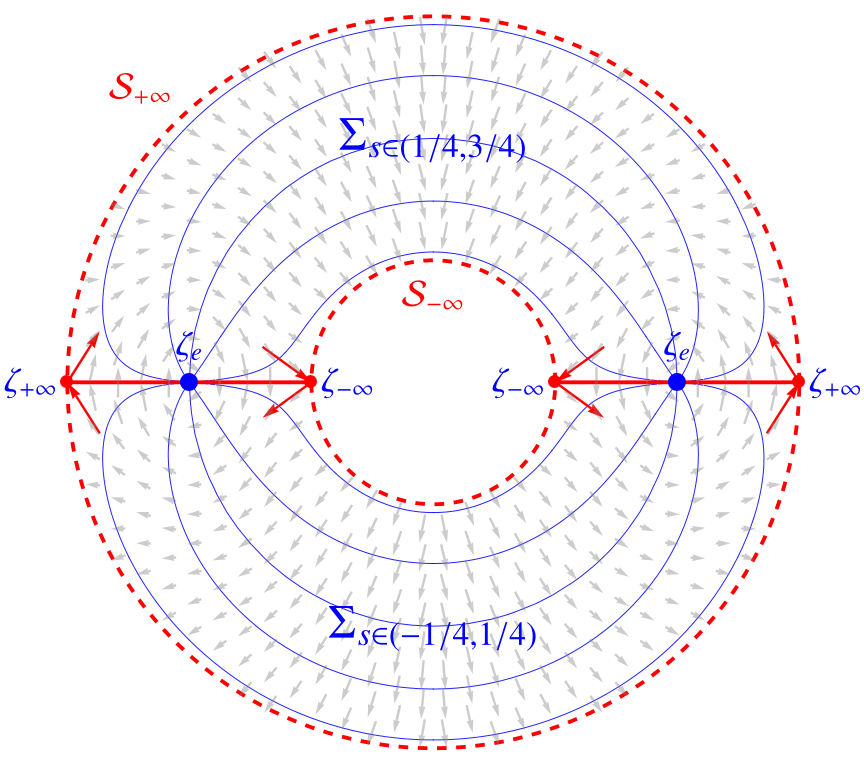}
    \caption{Singular foliation in the unbounded energy surface}
    \label{fig: foliations}
\end{figure}

It is well known since McGehee \cite{McGehee1973} that $\mathfrak{M}$ contains a real-analytic branch $W^{\rm s,u}(\zeta_{+\infty})$ of the stable/ unstable manifold of $\zeta_{+\infty}$, both branches converging to $z=+\infty$. A similar conclusion holds for $\zeta_{-\infty}$. Near $z=\pm \infty$, these branches separate trajectories that monotonically converge to $z=\pm \infty$ with non-vanishing asymptotic velocity from those whose velocity eventually changes sign and return to $z=0$. Understanding the intersections between the invariant manifolds $W^{\rm s,u}(\zeta_{\pm \infty})$ is a difficult task due to the complexity of the flow on $\mathfrak{M}$. However, standard arguments demonstrate that intersections between them, possibly between $W^{\rm s}(\zeta_{+\infty})$ and $W^{\rm u}(\zeta_{-\infty})$, always exist, and the dynamics on $\mathfrak{M}\equiv S^2 \times \R$ has positive topological entropy if one such intersection is transverse.

The trajectories in $\mathfrak{M}$ that converge to $\zeta_{\pm \infty}$ are called {\em parabolic}. The local branches of $W^{\rm s}_{\rm loc}(\zeta_{+ \infty})$ and $W^{\rm u}_{\rm loc}(\zeta_{-\infty})$ intersect each $\Sigma_s, s\in(-1/4,1/4)$, on a real-analytic circles $C^{\rm s}_{s,+}$ and $C^{\rm u}_{s,-}$, respectively. The points in $C^{\rm s}_{s,+}$ and in $C^{\rm u}_{s,-}$ monotonically converge to $z=+\infty$ and $z=-\infty$ as $t\to +\infty$ and $t\to -\infty$, respectively.


Our next result confirms the existence of infinitely many periodic orbits and infinitely many parabolic trajectories both forward and backward in time.

\begin{thm}\label{thm: main5}
Let $(\beta,\mathfrak e)\in(0,1)\times (0,1)$ satisfies $\beta^2+\mathfrak e^2>1$. The following statements hold:
\begin{itemize}
\item[(i)] The energy surface $\mathfrak M = \mathfrak M_{\beta,\mathfrak e}$ carries infinitely many periodic orbits, at least one $z$-symmetric brake orbit, and infinitely many parabolic trajectories satisfying $|z(t)|\to +\infty$ as $t\to \pm \infty$.

\item[(ii)] If $W^{\rm s}(\zeta_{+\infty})$ coincides with either $W^{\rm u}(\zeta_{+\infty})$ or $W^{\rm u}(\zeta_{-\infty})$, then  $\mathfrak M$ has  infinitely many $z$-symmetric periodic orbits, that is, their projection to the $(r,z)$-plane is symmetric with respect to the reflection $(r,z)\mapsto (r,-z)$.

\item[(iii)] If $W^{\rm s}(\zeta_{+\infty})$ coincides with neither
$W^{\rm u}(\zeta_{+\infty})$ nor $W^{\rm u}(\zeta_{-\infty})$,
then $\mathfrak M$ has infinitely many non-$z$-symmetric brake orbits,
and infinitely many brake-parabolic trajectories.

\end{itemize}
\end{thm}

The proof of Theorem \ref{thm: main5} is based on dynamical arguments of twist maps. Applying McGehee's degenerate stable manifold theorem \cite{McGehee1973}, we obtain a first hitting map from $A_0\subset \Sigma_0$ to $\Sigma_{1/2}$, where $A_0$ is the annulus-like region in $\Sigma_0$ bounded by $\zeta_e$ and $C^{\rm s}_{0,+}$. This map has an infinite twist near $C^{\rm s}_{0,+}$ and thus a standard argument using intersection of curves gives the desired periodic orbits and parabolic trajectories. Our methods do not directly apply to the critical case $\beta^2 + \mathfrak e^2=1$ even though the main conclusions of Theorem \ref{thm: main5} should also hold in this case.

The results above show that the global dynamics of the spatial isosceles three-body problem is organized by a geometric structure connecting spectral constraints, foliation structures, and twist dynamics.

\section{Proof of Theorem \ref{thm: main1}}\label{sec: rotation number}

The proof of Theorem~\ref{thm: main1} proceeds in two steps.
First, we establish the inequality $\rho_{\beta,\fe} > \rho_{\beta,0}$
for every $(\beta,\fe) \in (0,1) \times (0,1)$.
This follows from Theorems~\ref{thm: for 0<beta<5/28} and
\ref{thm: for beta>5/28}, which treat separately the cases
$\beta \in (0,5/28)$ and $\beta \in [5/28,1)$, respectively.
Theorem~\ref{thm: for 0<beta<5/28} is proved via a monotonicity argument,
while Theorem~\ref{thm: for beta>5/28} relies on a quantitative estimate
of the Morse index. In the second step, we prove that, for every fixed $\beta \in (0,1)$,
the function $\fe \mapsto \rho_{\beta,\fe}$ on $[0,1)$ is non-decreasing.
This follows from a global monotonicity argument based on
Lemmas~\ref{lem: for 1-degenerate curves},
\ref{lem: negative definite}, and
\ref{lem: trace and determinant are positive}.

Recall that the minimal Reeb period and the rotation number of the Euler orbit $\zeta_e=\mathfrak M\cap \{p_z=z=0\}$ are denoted by $T_e>0$ and $\rho_e>2$, respectively. It can be solved from \eqref{equ: Ham system} that
\begin{equation}\label{equ: Euler}
p_{r,e}(t)= \frac{\mathfrak{e}}{\varpi\beta}\sin\theta(t),\quad r_e(t)=\frac{\varpi^2\beta}{1+\mathfrak{e}\cos\theta(t)},\quad \theta(T_e)=2\pi,
\end{equation}
where $\theta(t)$ is determined by $\dot \theta(t)=\varpi/r^2_e(t)$ and $\mathfrak{e}=\sqrt{1-2\varpi^2\beta^2}$. Recall from \cite{HLOSY2023} that the rotation number $\rho_e=\rho_{\beta,\fe}>2$ for every $(\beta,\fe)\in(0,1)\times[0,1)$.
For every $(\beta,\mathfrak{e}) \in(0,1)\times[0,1)$, the linearized flow along $\zeta_e$ is decoupled. The subsystem on the $(p_z,z)$-plane satisfies
\begin{equation}\label{equ: linear}
\dot \xi_1=\begin{pmatrix}
0 & -\dfrac{4\alpha^{-1}(1+2\alpha)}{r_e(t)^3}\\
1 & 0
\end{pmatrix}\xi_1,\quad \xi_1=(\xi_{p_z},\xi_z)^T.
\end{equation}
Let $t(\theta)$ be determined by $t'(\theta)=r_e^2(\theta)/\varpi$. Using the following transformation from \cite{HLOSY2023}
$$
\xi_2(\theta) = \mathcal{R}(\theta) \xi_1(t(\theta)), \quad \mathcal{R}(\theta):=\begin{pmatrix}
 \dfrac{r_e}{\sqrt{\varpi}}  &   - \dfrac{\mathfrak{e} \sin \theta}{\sqrt{\varpi^3}\beta} \\
 0 & \dfrac{\sqrt\varpi}{ r_e} \end{pmatrix}, \quad \forall \theta \in \mathbb{R} / 2\pi \mathbb{Z},
$$
the equation \eqref{equ: linear} becomes
\begin{equation} \label{equ: new linear equation of Euler}
\xi_2'=J_2 \begin{pmatrix}
1 & 0\\
 0 & 1+\dfrac{7\beta}{1+\mathfrak{e}\cos\theta}
\end{pmatrix}\xi_2,\quad \beta=\frac{1}{1+4/\alpha}.
\end{equation}

Let $\gamma_{\beta,\mathfrak e}:\mathbb{R}\rightarrow \mathrm{Sp}(2)$ be the fundamental solution of \eqref{equ: new linear equation of Euler}. Due to Proposition~7.11 in \cite{HLOSY2023}, the rotation number of the Euler orbit $\zeta_e$ is $\rho_{\beta,\mathfrak e}=\rho(\gamma_{\beta,\mathfrak e})+1$, where $\rho(\gamma_{\beta,\fe})$ is the rotation number of $\gamma_{\beta,\fe}$, see Appendix \ref{sec: maslov-type indices and rotation numbers} for the definition of $\rho(\cdot)$. If $\mathfrak e=0$, then
\begin{equation}\label{equ: gamma beta,0}
\gamma_{\beta,0}(\theta)=\begin{pmatrix}
\cos\sqrt{1+7\beta}\theta & -\sqrt{1+7\beta}\sin\sqrt{1+7\beta}\theta\\
\frac{1}{\sqrt{1+\beta}}\sin\sqrt{1+7\beta}\theta & \cos\sqrt{1+7\beta}\theta
\end{pmatrix},\quad \theta\in \mathbb R.
\end{equation}
Moreover, by definition, the rotation number $\rho_{\beta,0}=\sqrt{1+7\beta}+1$.

For every $(\beta,\mathfrak{e}) \in (0,1) \times (0,1)$ and $\omega\in \mathbf{U}:=\{|z|=1\}$, we define the self-adjoint operator
$$
(\mathcal{A}_{\beta,\mathfrak{e}}^\omega\cdot x)(\theta):=-x''(\theta)-\left(1+\frac{7\beta}{1+\mathfrak{e}\cos\theta}\right)x(\theta), \qquad \forall \theta \in [0,2\pi],
$$
defined on a suitable dense subset $\Lambda(\omega)$ of $L^2([0,2\pi], \C)$ containing $W^{2,2}([0,2\pi], \C)$, and satisfying the boundary conditions $x(2\pi)=\omega x(0)$ and $x'(2\pi)=\omega x'(0)$. The spectrum of $\mathcal{A} = \mathcal{A}_{\beta,\mathfrak{e}}^\omega$ is formed by countably many real eigenvalues which are bounded from below and accumulate only at $+\infty$. The eigenspace associated with each eigenvalue has complex dimension at most $2$. The nullity of $\mathcal{A}$ is defined as the complex dimension of its kernel, i.e. $\nu_\omega(\mathcal{A}):=\dim_\C \ker \mathcal{A}$. The Morse index $m_\omega^-(\mathcal{A})$ of $\mathcal{A}$ is defined as the total multiplicity of negative eigenvalues. Both notions are closely related to the solutions of the corresponding first-order linear system
\begin{equation*}
\xi'(\theta)=J_2 \begin{pmatrix}
1 & 0\\
 0 & 1+\dfrac{7\beta}{1+\mathfrak{e}\cos\theta}
\end{pmatrix}\xi(\theta), \qquad J_2 = \begin{pmatrix} 0 & -1 \\ 1 & 0 \end{pmatrix},
\end{equation*}
which coincides with \eqref{equ: new linear equation of Euler} and represents the transverse linearized dynamics along the Euler orbit in a suitable frame.

As above, let $\gamma_{\beta,\mathfrak e}:\mathbb{R}\rightarrow \mathrm{Sp}(2)$ be a fundamental solution of the above equation starting from the identity $\gamma_{\beta,\mathfrak{e}}(0)= I_2$. Then
\begin{equation}\label{index equ beta, e}
m^-_{\omega}(\mathcal{A}_{\beta,\mathfrak e}^\omega)=i_{\omega}(\gamma_{\beta,\mathfrak{e}}), \quad \nu_{\omega}(\mathcal{A}_{\beta,\mathfrak e}^\omega)=\nu_{\omega}(\gamma_{\beta,\mathfrak{e}}),
\end{equation}
where $\nu_\omega(\gamma_{\beta,\mathfrak e})$ is the geometric multiplicity of $\omega$ as an eigenvalue of $\gamma_{\beta,\mathfrak{e}}(2\pi)$, and $i_\omega(\gamma_{\beta,\mathfrak{e}})$ is an algebraic count of intersections between $\gamma_{\beta,\mathfrak{e}}$ and the singular set ${\rm Sp}(2)_\omega^0:=\{M \in {\rm Sp(2)}: \det(M -\omega I_2)=0\}$, see Section \ref{sec: maslov-type indices and rotation numbers} for more details.

Finally, for each $j\in \N$ and $\omega=e^{2\pi \nu i} \in \mathbf{U}\setminus \{-1\}$, $\nu\in [0,1)$, we define the subset $\Gamma_j^\omega \subset [0,1) \times [0,1)$
$$
\Gamma_j^\omega := \{(\beta,\mathfrak{e}):\rho(\gamma_{\beta,\mathfrak{e}})=j+\nu,\ \gamma_{\beta,\mathfrak{e}}(2\pi) \approx R(2\pi \nu)\}, \quad R(2\pi \nu ) = \begin{pmatrix}\cos (2\pi \nu) & - \sin (2\pi \nu) \\ \sin (2\pi \nu) & \cos (2\pi \nu) \end{pmatrix},
$$
where $\rho(\gamma_{\beta,\mathfrak{e}})$ is the rotation number of $\gamma_{\beta,\mathfrak{e}}$. It can be proved that $\Gamma_j^\omega$ is the graph of a real-analytic function $\mathfrak{e} \mapsto \beta_j^\omega(\mathfrak{e}), \mathfrak{e}\in [0,1),$ see the discussion before Lemma \ref{lem: the second derivitive}.

For $j\in \N$ and $\omega = -1$, we define the subsets $\Gamma_{j,\pm} \subset (0,1) \times [0,1)$ by
$$
\Gamma_{j,\pm}:= \left\{(\beta,\mathfrak{e}): \rho(\gamma_{\beta,\mathfrak{e}}) = j+1/2,\ \gamma_{\beta,\mathfrak{e}}(2\pi) \approx N(-1,\pm 1)\right\}, \quad N(-1,\pm 1) = \begin{pmatrix} -1 & \pm 1\\ 0 & -1 \end{pmatrix}.
$$

As before, each $\Gamma_{j,\pm}$ is the graph of a real-analytic function $\mathfrak{e} \mapsto \beta_{j, \pm}, \mathfrak{e} \in [0,1)$. This will be clarified in Lemma \ref{lem: degenerate}.

We say that $\mathcal{A}$ is $\omega$-degenerate if $\nu_{\omega}(\mathcal{A})\neq 0$. It follows that $\mathcal A$ is $\omega$-degenerate along $\Gamma^\omega_j$ and $\Gamma^{\bar \omega}_j$ for every $j\in \N$.

The following propositions summarize some useful properties of the curves $\Gamma^\omega_{j}, \Gamma_{j,\pm}$.

\begin{prop}[{\cite[Theorem 2.9, Remark 7.16]{HLOSY2023}, \cite[Theorem 1.1]{HOT2023}}] \label{prop: degenerate curves}
The following statements hold:
\begin{itemize}
\item[(i)]  If $\omega=e^{2\pi \nu i}, \omega' = e^{2\pi \nu' i} \in \mathbf{U} \setminus \{-1\},$ satisfy $0\leq \nu<\nu'<1$,  then   $\beta^{\omega}_j<\beta^{\omega'}_j< \beta^{\omega}_{j+1}$ for every $j$.  Moreover, $\Gamma^1_1\subset \{\beta=0\}$ and $\beta_j^{\omega}(0)=((j+\nu)^2-1)/7$.

\item[(ii)]  $\beta_{j,-}(\mathfrak{e})<\beta_{j,+}(\mathfrak{e}), \forall \mathfrak{e} \in (0,1)$, and $\beta_{j,-}(0)=\beta_{j,+}(0)=((j+1/2)^2-1)/7$ for every $j$.

\item[(iii)]  $\lim_{\mathfrak e\to 1^-}\beta_{1}^\omega(\mathfrak e)=0$ for every $\omega=e^{2\pi \nu i},  \nu\in [0,1/2)$.
Also, $\lim_{\mathfrak e\to 1^-}\beta_{j}^\omega(\mathfrak e)=1/56$ for every $j\geq 2$, or for $j=1$ and $\nu\in (1/2,1)$. Moreover, $\lim_{\mathfrak e\rightarrow 1^-}\beta_{1,-}(\mathfrak e)=0$ and $\lim_{\mathfrak e\rightarrow 1^-}\beta_{1,+}(\mathfrak e)=1/56$.

\item[(iv)] $(\beta_{j}^\omega)'(0)=(\beta_{j,\pm})'(0)=0$ for every $\omega \in \mathbf U\setminus \{-1\}$ and $j$.

\item[(v)] For each $\mathfrak{e} \in (0,1)$ fixed, the continuous function $\beta \mapsto \rho(\gamma_{\beta,\mathfrak e})$ is strictly increasing on $[0,+\infty)\setminus\cup_{j=1}^{+\infty}[\beta_{j,-}(\mathfrak e),\beta_{j,+}(\mathfrak e)]$ and coincides with $j+1/2$ on $ [\beta_{j,-}(\mathfrak e),\beta_{j,+}(\mathfrak e)]$ for every $j$.
\item[(vi)] For every fixed $\omega\in \mathbf U,\fe\in[0,1)$, the function $\beta\mapsto i_{\omega}(\gamma_{\beta,\fe})$ is non-decreasing on $[0,+\infty)$.
\end{itemize}
\end{prop}

\begin{figure}[ht]
\centering
\includegraphics[width=0.6\textwidth]{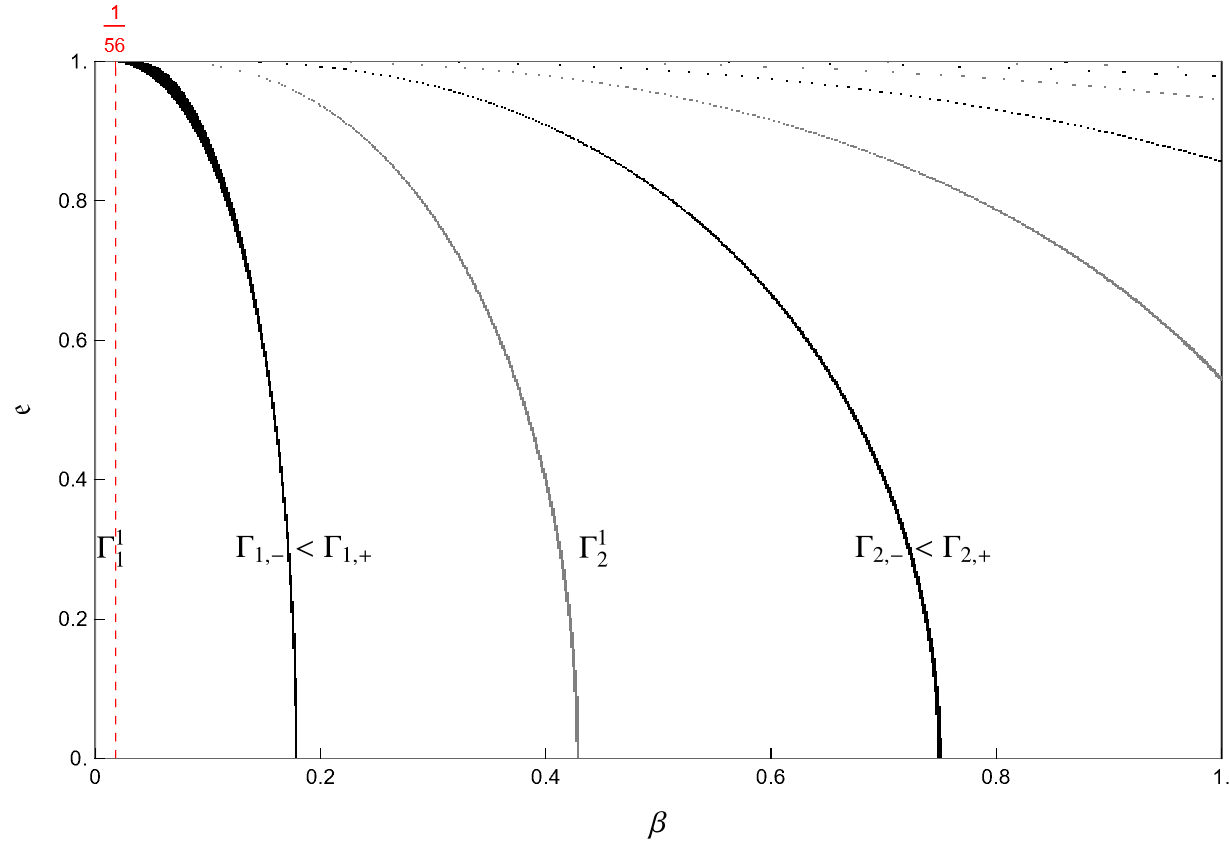}
\caption{The curves $\Gamma_j^1$ and $\Gamma_{j,\pm}$ in $[0,1] \times [0,1]$.}
\label{picture of Gamma}
\end{figure}

To prove Theorem \ref{thm: main1}, we first prove the inequality $\rho_{\beta,\mathfrak e}>\rho_{\beta,0},\fe\in(0,1)$, in two cases: $\beta\in (0,5/28)$ and $\beta\in [5/28,+\infty)$. Notice that $5/28 = \beta_{1,-}(0)=\beta_{1,+}(0)$. We start with the case $\beta\in (0,5/28)$. Recall that $\rho_e=\rho_{\beta,\mathfrak e}=\rho(\gamma_{\beta,\mathfrak e})+1$ for every $(\beta,\mathfrak{e}) \in (0,1) \times [0,1)$.
The inequality $\rho_{\beta,\mathfrak e}>\rho_{\beta,0}$ for $(\beta,\mathfrak{e})\in (0, 5/28) \times (0,1)$ follows directly from Theorem~\ref{thm: for 0<beta<5/28} below.

\begin{thm}\label{thm: for 0<beta<5/28}
For every $\omega=e^{2\pi \nu i} , \nu \in (0,1/2]$,  we have $(\beta_1^\omega)' (\mathfrak{e}), (\beta_{1,\pm})'(\mathfrak{e})<0$ for every $\mathfrak e\in(0,1)$. In particular, for every $\nu \in (0,1/2)$, we have $\rho_{\beta,\mathfrak e}>\rho_{\beta,0}=\sqrt{1+7\beta}+1$ for every $(\beta,\mathfrak{e})\in (0, 5/28) \times (0,1)$.
\end{thm}

Recall that $\mathcal A_{\beta,\mathfrak e}^\omega=-d^2/d\theta^2-1-7\beta/(1+\mathfrak e\cos\theta)$ and that $\mathcal A_{\beta_j^\omega(\fe),\mathfrak e}^\omega$ is $\omega$-degenerate for every $j$. If $\beta_j(\fe)=\beta_j^\omega(\fe)\ (\omega\neq 1)$ or $\beta_{j,\pm}(\fe)\ (\omega=-1)$, then $\nu_\omega(\mathcal A^\omega_{\beta_j(\fe),\mathfrak e})=1,\forall \mathfrak e\in [0,1)$. If $\beta_j(\fe)=\beta_j^1(\fe)$, then $\nu_1(\mathcal A_{\beta_j(\fe),\mathfrak e}^1)= 2,\forall \fe\in[0,1)$. Hence, for every $\omega\in \mathbf U$, there always exists a family $x_\fe\in \ker \mathcal A_{\beta_j(\fe),\mathfrak e}^\omega\setminus\{0\}$, which  analytically depends on $\fe\in(0,1)$.
To prove Theorem \ref{thm: for 0<beta<5/28}, we need the following two lemmas.

\begin{lem}\label{lem: beta'e}
Fix $\omega\in \mathbf U$ and $j\in \N$. Assume that $x_{\mathfrak{e}}\in \ker \mathcal{A}^\omega_{\beta(\fe),\fe}\setminus\{0\}$ is analytic in $\fe \in (0,1)$, where $\beta(\fe)=\beta_j^\omega(\fe)\ (\omega\neq -1)$ or $\beta_{j,\pm}(\fe)\ (\omega=-1)$. Let $x_\fe(\theta) = \sum_{n\in \Z} a_ne^{i(n+\nu)\theta}, a_n\in \C,$ be its Fourier expansion. Then the sign of $\beta'(\mathfrak e)$ coincides with the sign of
$$
\langle \partial_{\mathfrak{e}} \mathcal{A}^\omega_{\beta(\fe),\mathfrak{e}} \cdot x_{\mathfrak{e}},x_{\mathfrak{e}}\rangle
=\frac{2\pi}{7\mathfrak{e}\beta(\fe)}\sum_{n\in\mathbb{Z}}C_n|A_n|^2,
$$
where $\langle\cdot,\cdot\rangle$ denotes the $L^2$-product, $A_n:=(1-(n+\nu)^{2})a_n$ and $C_n:=7\beta(\fe)/((n+\nu)^{2}-1)-1, n\in \mathbb{Z}$.
\end{lem}	

\begin{proof}
Since $x_{\fe}$ analytically depends on $\mathfrak{e}$, and satisfies $\langle \mathcal{A}^\omega_{\beta(\fe),\mathfrak{e}} \cdot x_{\mathfrak{e}},x_{\mathfrak{e}}\rangle =0$ for every $\mathfrak{e}\in(0,1)$. Taking the derivative with respect to $\fe$ on both sides and using that $\mathcal{A}_{\beta(\fe)}^\omega$ is self-adjoint, we obtain
\begin{equation}\label{equ: beta'}
\beta'(\fe)\langle \partial_\beta \mathcal{A}^\omega_{\beta(\fe),\mathfrak{e}} \cdot x_{\mathfrak{e}},x_{\mathfrak{e}}\rangle +
\langle \partial_{\mathfrak{e}} \mathcal{A}^\omega_{\beta(\fe),\mathfrak{e}} \cdot x_{\mathfrak{e}},x_{\mathfrak{e}}\rangle=0, \quad \forall \fe\in (0,1).
	\end{equation}

Since $\langle \partial_\beta \mathcal{A}^\omega_{\beta(\fe),\mathfrak{e}} \cdot x_{\mathfrak{e}},x_{\mathfrak{e}}\rangle
=-\int_{0}^{2\pi}7|x_{\mathfrak{e}}|^2/(1+\mathfrak{e}\cos\theta)d\theta$ is always negative, the sign of $\beta'(\fe)$ coincides with the sign of $\langle \partial_{\mathfrak{e}} \mathcal{A}^\omega_{\beta(\fe),\mathfrak{e}} \cdot x_{\mathfrak{e}},x_{\mathfrak{e}}\rangle$. Moreover, since $x_{\mathfrak{e}}\in \ker\mathcal{A}^\omega_{\beta(\fe),\fe}$, we have
\begin{equation}\label{equ: x equation}
\fe\cos\theta(x_{\mathfrak{e}}''(\theta)+x_{\mathfrak{e}}(\theta))
=-(x_{\mathfrak{e}}''(\theta)+x_{\mathfrak{e}}(\theta))-7\beta(\fe) x_{\fe}(\theta), \quad \forall \theta\in \R/2\pi\Z.
\end{equation}
Using the Fourier expansion $x_{\mathfrak{e}}(\theta)=\sum_{n\in\mathbb{Z}}a_{n}e^{i(n+\nu)\theta},a_n\in\mathbb{C}$, we obtain from \eqref{equ: x equation} that
\begin{equation}\label{equ: iteration of An}
C_n A_n=\frac{\mathfrak{e}}{2}\left(A_{n-1}+A_{n+1}\right),\quad \forall n\in \mathbb{Z},
\end{equation}
where $A_n$ and $C_n$ are as in the statement.

Combining \eqref{equ: x equation} and \eqref{equ: iteration of An}, we compute
$$
\begin{aligned}
\langle \partial_{\mathfrak{e}} \mathcal{A}^\omega_{\beta(\fe),\mathfrak{e}} \cdot x_{\mathfrak{e}},x_{\mathfrak{e}}\rangle
&=\int_{0}^{2\pi}7\beta(\fe)\cos\theta(1+\mathfrak{e}\cos\theta)^{-2}x_{\mathfrak{e}}(\theta) \bar{x}_{\mathfrak{e}}(\theta)d\theta\\ &=\frac{1}{7\beta(\fe)}\int_{0}^{2\pi}\cos\theta(x_{\mathfrak{e}}''(\theta)+x_{\mathfrak{e}}(\theta))(\overline{x_{\mathfrak{e}}''(\theta)+x_{\mathfrak{e}}(\theta)})d\theta\\
&=\frac{1}{7\mathfrak{e}\beta(\fe)}\int_{0}^{2\pi}(-x_{\mathfrak{e}}''(\theta)-x_{\mathfrak{e}}(\theta)-7\beta(\fe) x_{\mathfrak{e}}(\theta)) (\overline{x_{\mathfrak{e}}''(\theta)+x_{\mathfrak{e}}(\theta)})d\theta\\
&=\frac{2\pi}{7\mathfrak{e}\beta(\fe)} \sum_{n\in\mathbb{Z}}C_n|A_n|^2.
\end{aligned}$$
This proves the lemma.
\end{proof}

\begin{lem}\label{lem: the second derivitive}
Fix $\omega\in \mathbf U$ and $j\in \mathbb \N$. Let $\beta(\fe)=\beta_j^\omega(\fe)\ (\omega\neq -1)$ or $\beta_{j,\pm}(\fe)\ (\omega=-1)$, we have
$
\beta''(0)=-\frac{3\beta(0)(1+7\beta(0))}{3+28\beta(0)}< 0.$ In particular, we have
$(\beta_{j,\pm})''(0)=-\frac{3(1+2j)^2((j+1/2)^2-1)}{112j(1+j)}< 0.$
\end{lem}

\begin{proof}
We first consider $\omega\neq -1$. Let $y_n:=y^{(n)}=\partial_{\mathfrak{e}}^n x_{\mathfrak{e}}$ and $\beta_{j,0}:=\beta_j^\omega(0)$. Recall from Proposition \ref{prop: degenerate curves}-(i) that $\omega=e^{2\pi i\sqrt{1+7\beta_{j,0}}}$. Recall also that $x_\fe(2\pi)=\omega x_\fe(0)$.
Taking the $n$-th derivative of \eqref{equ: x equation} with respect to $\mathfrak e$ and computing it at $\mathfrak e=0$, we obtain
\begin{equation}\label{equ: iteration of yn}
-y''_n-y_n-7\beta_{j,0}^\omega y_n=n\cos\theta(y_{n-1}''+y_{n-1})+7\sum_{k=1}^{n}C_{n}^{k}\partial^k_\fe\beta_j^\omega(0)y_{n-k},\quad  y_n(2\pi)=\omega y_n(0).
\end{equation}

For $n=0$, we may take $y_0(\theta)=x_0(\theta):=e^{i\sqrt{1+7\beta_{j,0}}\theta}$. Indeed, $y_0$ solves $y_0''+(1+7\beta_{j,0}) y_0=0$ and satisfies the $\omega$-boundary condition. Recall from Proposition \ref{prop: degenerate curves}-(iv) that $(\beta_j^\omega)'(0)=0$ for every $\omega, j$. If $n=1$, we obtain $y_1(2\pi)=\omega y_1(0)$ and
$$-y''_1-y_1-7\beta_{j,0} y_1=\cos\theta(y''_0+y_0) =-\frac{7\beta_{j,0}}{2}\left(e^{i(\sqrt{1+7\beta_{j,0}}+1)\theta}+e^{i(\sqrt{1+7\beta_{j,0}}-1)\theta}\right).
$$
By solving the equation above, we obtain
$$y_1=\mu_1e^{i(\sqrt{1+7\beta_{j,0}}+1)\theta}+\mu_{-1}e^{i(\sqrt{1+7\beta_{j,0}}-1)\theta}+\mu_0e^{i\sqrt{1+7\beta_{j,0}}\theta},$$
where $2\mu_{\pm 1}=-7\beta_{j,0}/((\sqrt{1+7\beta_{j,0}}\pm 1)^2-(1+7\beta_{j,0}))$ and $\mu_0\in \C$.

To compute $(\beta_j^\omega)''(0)$, we recall from Lemma \ref{lem: beta'e} that
$$
(\beta_j^\omega)'(\fe)
=\frac{\int_{0}^{2\pi}\cos\theta (x_{\mathfrak{e}}''+x_{\mathfrak{e}})(\overline{x_{\mathfrak{e}}''
+x_{\mathfrak{e}}})d\theta}{49\beta_j^\omega(\fe)\int_{0}^{2\pi}|x_{\mathfrak{e}}|^2/(1+\mathfrak{e}\cos\theta)d\theta}
=\frac{\int_{0}^{2\pi}\cos\theta(x_{\mathfrak{e}}''+x_{\mathfrak{e}})(\overline{x_{\mathfrak{e}}''+x_{\mathfrak{e}}})d\theta}{-7\int_{0}^{2\pi}(x_{\mathfrak{e}}''
+x_{\mathfrak{e}})\bar{x}_{\mathfrak{e}}d\theta},
$$
and thus
$$
\begin{aligned}
(\beta_j^\omega)''(0)&=\frac{2\int_{0}^{2\pi}\cos\theta\cdot \mathrm{Re}((y_1''+y_1)(\overline{y_0''+y_0}))d\theta}{-7\int_{0}^{2\pi}(y_0''+y_0)\bar{y}_0d\theta}\\
&=\frac{-\beta_{j,0}^2\cdot 3(1+7\beta_{j,0})\int_0^{2\pi}2\cos^2\theta d\theta}{(3+28\beta_{j,0})\cdot 2\pi\beta_{j,0}} =-3\beta_{j,0}\cdot\frac{1+7\beta_{j,0}}{3+28\beta_{j,0}}.
\end{aligned}
$$
Here, we have used the equation $y''_1+y_1=-7\beta_{j,0} y_1-\cos\theta(y''_0+y_0)$ which implies
$$
\begin{aligned}
\mathrm{Re}((y''_1+y_1)(\overline{y''_0+y_0}))&=49\beta_{j,0}^2\mathrm{Re}\left(\mu_1e^{i\theta}+\mu_{-1}e^{-i\theta}+\mu_0-\cos\theta\right)\\
		&=49\beta_{j,0}^2\left((\mu_{1}+\mu_{-1}-1)\cos\theta+\mu_0\right)\\
        &=49\beta_{j,0}^2\left(-\frac{3(1+7\beta_{j,0})}{3+28\beta_{j,0}}\cos\theta+\mu_0\right).
\end{aligned}
$$
The proof for $\omega\neq -1$ is complete.

Now assume that $\omega=-1$. Recall that for any $(\beta,\fe)$, with $\beta>0$ and $0\leq \fe<1$, the domain of $\mathcal A_{\beta,\fe}:=\mathcal A^{-1}_{\beta,\fe}$ as $\Lambda_{-1}:=\{x\in W^{2,2}([0,2\pi],\mathbb C): x(2\pi)=-x(0),\dot x(2\pi)=-\dot x(0)\}$. Consider the linear operator $g$ defined on $\Lambda_{-1}$, given by $g\cdot x(\theta):=x(2\pi-\theta)$. We see that $g$ is an involution and satisfies $\mathcal A_{\beta,\mathfrak e}\circ g=g\circ \mathcal A_{\beta,\mathfrak e}$. In particular, $g$ admits two eigenspaces
$$\begin{aligned}
&E_1:=\{x\in \Lambda_{-1}:x(\theta)=x(2\pi-\theta)\}\cong\{x\in W^{2,2}([0,\pi],\mathbb C):x(0)=0,\dot x(\pi)=0\},\\
&E_2:=\{x\in \Lambda_{-1}:x(\theta)=-x(2\pi-\theta)\}\cong\{x\in W^{2,2}([0,\pi],\mathbb C):\dot x(0)=0,x(\pi)=0\}.
\end{aligned}$$

Let $\mathcal A_i:=\mathcal A_{\beta,\fe}=(-d^2/d\theta^2-1-7\beta/(1+\fe \cos \theta))|_{E_i},i=1,2,$ and see from the boundary conditions above that they correspond to Sturm-Liouville operators. The commutativity of $\mathcal A_{\beta,\mathfrak e}$ and $g$ induces the splitting $\mathcal A_{\beta,\mathfrak e}=\mathcal A_1\oplus \mathcal A_2$. For each $(\beta,\fe)$, the eigenvalues of $\mathcal{A}_1, \mathcal{A}_2$ are given by
$\lambda_0< \lambda_1 < \ldots <\lambda_j < \cdots \to +\infty$ and $\bar \lambda_0< \bar \lambda_1 < \cdots <\bar \lambda_j < \cdots \to +\infty$, respectively, and the eigenspaces are all one-dimensional. If $\lambda_j=0$ for some $(\beta_0,\fe_0)$ and $j$, with $\beta_0>0$ and $0\leq \fe_0 <1$, then the strictly inequality $\partial_\beta \mathcal{A}_1 <0$ implies that there exists a real-analytic curve $\beta = \beta_j(\fe)$, defined near $\fe_0$ and satisfying $\beta_j(\fe_0)=\beta_0$, such that $\mathcal{A}_1$ has a one-dimensional kernel along $(\beta_j(\fe), \fe)$. The same holds for $\mathcal{A}_2$. In particular, the curves $\Gamma_{j,\pm}$ are real-analytic for every $j$. They intersect precisely at $(\beta, \fe) = (((j+1/2)^2-1)/7,0)$, where $\lambda_j = \bar \lambda_j=0$ and the kernels of $\mathcal{A}_1$ and $\mathcal{A}_2$ are both one-dimensional.  The proof that $\beta_{j,-}(\fe) < \beta_{j,+}(\fe)$ for every $\fe \in (0,1)$, is given in Proposition \ref{prop: degenerate curves}. The behavior of $\Gamma_{j,\pm}$ near $\fe=1$ is studied in \cite{HOT2023}. Notice that the same argument can be used to show that the curves $\Gamma^\omega_j$ are real-analytic for general values of $\omega$ and $j$.

To better understand the restricted operators $\mathcal{A}_1$ and $\mathcal{A}_2$ along the curves $\Gamma_{j,\pm}$, we have the following lemma.

\begin{lem}\label{lem: degenerate}
Fix $j\in \N$, and let $\mathcal A_{j,i}^\pm:=\mathcal A_{\beta_{j,\pm}(\fe),\mathfrak e}|_{E_i},i=1,2$. The following statements hold:
\begin{itemize}
\item[(i)] $\dim \ker \mathcal A_{j,1}^-=1$ and $\dim \ker \mathcal A_{j,2}^-=0$.
\item[(ii)] $\dim \ker \mathcal A_{j,2}^+=1$ and $\dim \ker \mathcal A_{j,1}^+=0$.

\end{itemize}
\end{lem}

\begin{proof}
For any $(\beta,\fe)$, with $\beta>0$ and $0\leq \fe <1$, we first observe that
$\ker(\mathcal A_i)  = \ker(\bar{\mathcal A}_i),$ where
$$
\bar{\mathcal A}_i:=\left[-\frac{1+\mathfrak e \cos \theta}{7\beta}+\left(-\frac{d^2}{d\theta^2}-1\right)^{-1}\right]\bigg|_{E_i}, i=1,2.
$$

Consider the basis $\mathcal B_1=\{e_k(\theta):=\sin((k+1/2)\theta),k\geq 0\}$ of $E_1$ and  $\mathcal B_2=\{f_k(\theta):=\cos((k+1/2)\theta),k\geq 0\}$ of $E_2$. We compute
$$
\begin{aligned}
&\quad\ \frac{1}{\pi}\int_0^{2\pi} e_k(\theta)e_l(\theta) d\theta =\delta_{kl}.\\
&\quad\ \frac{1}{\pi}\int_0^{2\pi}e_k(\theta) e_l(\theta) \cos \theta d\theta =\left\{\begin{aligned}
&-1/2,\quad \text{if}\ k=l=0,\\
&1/2,\quad \text{if}\ |k-l|= 1,\\
&0,\quad \text{if}\ |k-l|>1.
\end{aligned}\right.\\
&\quad\ \frac{1}{\pi}\int_0^{2\pi}e_k(\theta) \cdot \left[\left(\frac{d^2}{d\theta^2}+1\right)^{-1} \cdot e_l(\theta)\right]d\theta =\frac{\delta_{kl}}{-(k+1/2)^2+1}.
\end{aligned}
$$
and
$$\begin{aligned}
&\quad\ \frac{1}{\pi}\int_0^{2\pi}f_k(\theta)f_l(\theta)d\theta =\delta_{kl}.\\
&\quad\ \frac{1}{\pi}\int_0^{2\pi} f_k(\theta) f_l(\theta)\cos\theta d\theta =\left\{\begin{aligned}
&1/2,\quad \text{if}\ k=l=0\ \text{or}\ |k-l|= 1,\\
&0,\quad \text{if}\ |k-l|>1.
\end{aligned}\right.\\
&\quad\ \frac{1}{\pi}\int_0^{2\pi}f_k(\theta)\cdot \left[\left(\frac{d^2}{d\theta^2}+1\right)^{-1}\cdot f_l(\theta)\right]d\theta =\frac{\delta_{kl}}{-(k+1/2)^2+1}.
\end{aligned}
$$

Let $\bar A_i:=(\bar{\mathcal A}_{i,kl})_{k,l\geq 0}$ be the matrization of $\bar {\mathcal A}_i$ with respect to the basis $\mathcal B_i,i=1,2$, that is, $\mathcal{A}_{1,kl} = (1/\pi)
\int_0^{2\pi} \bar {\mathcal{A}}_1(e_k)(\theta) e_l(\theta) d\theta$ and  $\mathcal{A}_{2,kl} = (1/\pi)
\int_0^{2\pi} \bar {\mathcal{A}}_2(f_k)(\theta) f_l(\theta) d\theta$ for every $k,l$.
From the formulas above, we can check that $\bar A_1\geq \bar A_2$. In particular, $\bar{\mathcal A}_{1,00}=\bar {\mathcal A}_{2,00}+\fe/(7\beta)$ and $\bar{\mathcal A}_{1,kl}=\bar{\mathcal A}_{2,kl},\forall k+l>0$.

It suffices to prove (i) and (ii) for $\fe>0$ sufficiently small. We know that $\bar{ \mathcal A}_1$ admits a sequence of eigenvalues $\lambda_{0,1}>\lambda_{1,1}>\lambda_{2,1}>\cdots>\lambda_{j,1}$ and $\bar{\mathcal A}_2$ admits a sequence of eigenvalues $\lambda_{0,2}>\lambda_{1,2}>\lambda_{2,2}>\cdots>\lambda_{j,2}$ for $\fe>0$ sufficiently small. In particular, $\mathcal A_{\beta,\fe}$ degenerates if some eigenvalue $\lambda_{j,i}$ becomes zero. We also know that $\lambda_{j,i}$ depends on $(\beta,\fe)$ in a real-analytic way near $\fe=0$. From the reasoning above, $\lambda_{j,1}(\beta,\fe)\geq \lambda_{j,2}(\beta,\fe)$ for every $(\beta,\fe)$. Since both $\bar{\mathcal A}_1$ and $\bar{\mathcal A}_2$ are strictly increasing in $\beta>0$, we conclude that $\partial_{\beta}\lambda_{j,i}(\beta,\fe)>0$. Moreover,  $\dim\ker \bar {\mathcal A}_1=1=\dim\ker \bar{\mathcal A}_2$ if $(\beta,\fe)=(\beta_{j,\pm}(0),0)$. By the implicit function theorem, there exists a real-analytic function $\beta_{j,i}(\fe)$, defined for $\fe$ sufficiently close to $0$, such that $\lambda_{j,i}(\beta_{j,i}(\fe),\fe)=0$ and $\beta_{j,i}(0)= \beta_{j,\pm}(0)$. We know that
$
0=\lambda_{j,2}(\beta_{j,2}(\fe),\fe)=\lambda_{j,1}(\beta_{j,1}(\fe),\fe)
\geq \lambda_{j,2}(\beta_{j,1}(\fe),\fe).$
The monotonicity of $\lambda_{j,i}$ in $\beta$ implies that $\beta_{j,2}(\fe)\geq \beta_{j,1}(\fe)$ for every $\fe\geq 0$ sufficiently small. This implies that $\Gamma_{j,-}$ is the graph of $\beta_{j,1}$ and $\Gamma_{j,+}$ is the graph of $\beta_{j,2}$. In other words,
$\beta_{j,-} =\beta_{j,1}$ and $\beta_{j,+}=\beta_{j,2}$. Hence, (i) holds along $\Gamma_{j,-}$ and  (ii) holds along $\Gamma_{j,+}$.

Since the curves $\Gamma_{j,-}$ and $\Gamma_{j,+}$ are real-analytic graphs in $(0,1)\times(0,1)$ and satisfy
$\beta_{j,-}(\fe) < \beta_{j,+}(\fe)$ for every $\fe \in (0,1)$
by Proposition~\ref{prop: degenerate curves}-(ii), they do not intersect for $\fe>0$. Along each such curve, the kernel of $\mathcal A_{\beta,\fe}$ is one-dimensional.
Therefore, the degeneracy cannot switch from $E_1$ to $E_2$ (or vice versa) without producing a simultaneous degeneracy in both subspaces, which would force an intersection of the curves.
Consequently, the identification of the kernel with $E_1$ along
$\Gamma_{j,-}$ and with $E_2$ along $\Gamma_{j,+}$, proved for $\fe$ sufficiently small, holds for all $\fe \in (0,1)$.
\end{proof}

Back to the proof of Lemma \ref{lem: the second derivitive}, we see that the following holds for $\fe=0$
$$
\ker \mathcal A_{j,1}^-=\{c_1\sin((j+1/2)\theta):c_1\in \mathbb C\},\quad \ker \mathcal A_{j,2}^+=\{c_2\cos((j+1/2)\theta):c_2\in \mathbb C\}.
$$

To compute $(\beta_{j,-})''(0)$, we consider a family of solutions $\fe \mapsto x_\mathfrak e\in \ker \mathcal A_{j,1}^-$ so that $x_0=\sin((j+1/2)\theta)$. Let $y_n$ be as before. We have $y_1\in E_1$ and
$$-y''_1-(j+1/2)^2 y_1=\cos\theta(y''_0+y_0) =-((j+1/2)^2-1)\cos\theta\sin((j+1/2)\theta).
$$
Solving this equation for $y_1$, we obtain
$$
y_1(\theta)=\frac{(j+1/2)^2-1}{4j(j+1)}\left((1+j)\sin((j-1/2)\theta) - j \sin((j+3/2)\theta)\right) + \mu_0 \sin((j+1/2)\theta),
$$
for some $\mu_0\in \C$. We directly compute
$$
\begin{aligned}
\mathrm{Re}\left((y''_1+y_1)(\overline{y''_0+y_0})\right)
&=\mu_0 ((j+1/2)^2-1)^2\sin^2((j+1/2)\theta)-\frac{(1 + 2 j) ((j+1/2)^2-1)^2}{32 j (1 + j)}\times\\
&\big[3(1 + 2 j) \cos\theta +(1 + j) (-3 + 2 j) \cos(2 j \theta) - j (5 + 2 j) \cos(2 (1 + j) \theta)\big].
\end{aligned}
$$
Using that $\beta_{j,\pm}(0)=((j+1/2)^2-1)/7$, we obtain
$$
\begin{aligned}
(\beta_{j,-})''(0)&=\frac{2\int_{0}^{2\pi}\cos\theta\cdot \mathrm{Re}((y_1''+y_1)(\overline{y_0''+y_0}))d\theta}{-7\int_{0}^{2\pi}(y_0''+y_0)\bar{y}_0d\theta}=-\frac{3(1+2j)^2((j+1/2)^2-1)^2\int_0^{2\pi}2\cos^2\theta d\theta}{32j(1+j)\cdot 7((j+1/2)^2-1)\cdot \pi} \\
&=-\frac{3(1+2j)^2((j+1/2)^2-1)}{112j(1+j)}=-3\beta_{j,\pm}(0)\cdot \frac{1+7\beta_{j,\pm}(0)}{3+28\beta_{j,\pm}(0)}.
\end{aligned}
$$
The computation of $(\beta_{j,+})''(0)$ using $\mathcal A_{j,2}^+$ is similar.
\end{proof}

\begin{proof}[Proof of Theorem \ref{thm: for 0<beta<5/28}]

Denote $\mathcal A^\omega_{j,\fe}:=\mathcal A^\omega_{\beta^\omega_j(\fe),\fe}$ and $\mathcal A_{j,\fe,\pm}:=\mathcal A_{\beta_{j,\pm}(\fe),\fe}$ for every $\omega$ and $j$. Let $j\in \N$, $\omega=e^{2\pi i\nu},$ $\nu\in (0,1/2),$ $\fe \in (0,1)$ and $x_{\mathfrak{e}}\in \ker \mathcal{A}^\omega_{1,\fe}\setminus \{0\}$. Recall from Lemma \ref{lem: beta'e} that $A_n=(1-(n+\nu)^{2})a_n$ and $C_n=7\beta_j^\omega(\fe)/((n+\nu)^{2}-1)-1$ for every $n\in \Z$, where $a_n\in \mathbb C$ is the $n$-th Fourier coefficient of $x_\fe$. The following properties are easily checked:
\begin{itemize}
    \item[(a)] If $\beta_j^\omega(\fe) \in (0,5/28]$ then $C_n<0$ for every $n\neq -2,1.$

    \item[(b)] If $\beta_j^\omega(\fe)\in(0,((2-\nu)^2-1)/7)$, then $C_{-2}=7\beta_j^\omega(\fe)/(((2-\nu)^2-1)-1< 0$.

    \item[(c)] If $\beta_j^\omega(\fe)\in(0,((1+\nu)^2-1)/7)$, then $C_{1}=7\beta_j^\omega(\fe)/((1+\nu)^2-1)-1< 0$.
    \end{itemize}

We now restrict to the case $j=1$. Recall from Proposition \ref{prop: degenerate curves}-(i) that $\beta_1^\omega(0)=((1+\nu)^2-1)/7$, where $\nu\in (0,1/2)$.
From Proposition \ref{prop: degenerate curves}-(iv) and Lemma \ref{lem: the second derivitive}, we have $(\beta_1^\omega)'(0)=0$ and $(\beta_1^\omega)''(0)<0$, which implies that $(\beta_1^\omega)'(\mathfrak e)<0$ for every $\mathfrak e>0$ sufficiently small. For fixed $\omega$, we want to prove that $(\beta_1^\omega)'(\fe)<0$ for every $\fe\in (0,1)$. By contradiction, assume the existence of $\mathfrak e_0\in(0,1)$ so that $(\beta_1^\omega)'(\mathfrak e_0)=0$ and $(\beta_1^\omega)'(\mathfrak e)<0$ for every $0<\mathfrak e<\mathfrak e_0$. We thus have $0<\beta_1^\omega(\fe_0)<\beta_1^\omega(0)=((1+\nu)^2-1)/7<5/28$. It then follows from the estimates (a)-(c) above that $C_{n}<0$ for every $n\in \mathbb Z$. Lemma \ref{lem: beta'e} implies that $(\beta_1^\omega)'(\mathfrak e_0)<0$, a contradiction. We conclude that $(\beta_1^\omega)'(\fe)<0$ for every $\fe\in (0,1)$. In particular, $\beta_1^\omega$ is strictly decreasing in $[0,1]$.

The case $\omega=-1$ is similar. Let $\mathfrak e\in (0,1)$ and $x_{\fe,\pm}\in \ker \mathcal A_{1,\fe,\pm}\setminus \{0\}$. We have $A_{n,\pm}=(1-(n+1/2)^{2})a_{n,\pm}$ and $C_{n,\pm}=7\beta_{1,\pm}(\fe)/((n+1/2)^{2}-1)-1$ for every $n$, where $a_{n,\pm}\in \mathbb C$ is the $n$-th Fourier coefficient of $x_{\fe,\pm}$. From Proposition \ref{prop: degenerate curves}-(ii), we know that $\beta_{1,\pm}(0)=5/28$. We see that $C_{n,\pm}<0$ for every $n\notin \{-2,1\}$, and $C_{-2,\pm}=C_{1,\pm}=28\beta_{1,\pm}(\fe)/5-1\leq 0$ if $\beta_{1,\pm}(\fe)\in(0,5/28]$. By Proposition \ref{prop: degenerate curves}-(iv) and Lemma~\ref{lem: the second derivitive}, we have $(\beta_{1,\pm})'(0)=0$ and $(\beta_{1,\pm})''(0)<0$, which implies that $(\beta_{1,\pm})'(\mathfrak e)<0$ for every $\mathfrak e>0$ sufficiently small. As in the previous case, assume the existence of $\mathfrak e_{0,\pm}\in (0,1)$ such that $(\beta_{1,\pm})'(\mathfrak e_{0,\pm})=0$ and $(\beta_{1,\pm})'(\mathfrak e)<0$ for every $0<\mathfrak e<\mathfrak e_{0,\pm}$. Since $\{C_{n,\pm}\}$ are non-positive and not identically zero, Lemma \ref{lem: beta'e} implies that $(\beta_{1,\pm})'(\mathfrak e_{0,\pm})<0$, a contradiction. Hence, we conclude that $\beta_{1,\pm}(\fe)$ are strictly decreasing in $\mathfrak e\in[0,1)$.

Finally, since $\beta=\beta_1^\omega(0)\in (0,5/28)$ for some $\omega\in \mathbf U\setminus \{\pm 1\}$ and $\rho_{\beta,\fe}$ strictly increases in $\beta\in(0,\beta_{1,-}(\fe))$ for every fixed $\fe\in(0,1)$, we have $\rho_{\beta,0}=\rho_{\beta_1^\omega(\fe),\fe}<\rho_{\beta,\fe}$ for every $(\beta,\fe)\in(0,5/28)\times (0,1)$. The proof is now complete.
\end{proof}

Now we consider the case $\beta\in [5/28,1)$. As we will see, the proof also covers part of the case $\beta\geq1$. To show that $\rho_{\beta,\mathfrak e}>\rho_{\beta,0}$ for $(\beta,\fe) \in [5/28,1) \times (0,1)$, we need to estimate the Morse index $m^{-}_\omega(\mathcal A^\omega_{\beta,\mathfrak e})$ of $\mathcal A^\omega_{\beta,\mathfrak e}$.

\begin{prop}\label{prop: estimates of Morse index}
Fix $\fe\in(0,1),j\in \mathbb N$. Let $\beta(\fe)=\beta_j^\omega(\fe)\ (\omega\neq -1)$ or $\beta_{j,\pm}(\fe)\ (\omega=-1)$. For every $\beta(0)\in [5/28,(6/\fe+3)/7]$, we have
$m^-_\omega(\mathcal{A}^\omega_{\beta(0),\fe})\geq\left\lfloor2\sqrt{1+7\beta(0)}\right\rfloor+1$ and $\beta(\fe)<\beta(0)$. In particular, we have $m^-_{-1}(\mathcal{A}_{\beta_{j,\pm}(0),\fe})\geq \left\lfloor2\sqrt{1+7\beta_{j,\pm}}(0)\right\rfloor+1$ and $\beta_{j,\pm}(\fe)<\beta_{j,\pm}(0)$ if $\omega=-1$.
\end{prop}

Before proving Proposition \ref{prop: estimates of Morse index}, we need to introduce some notations. Let $u_n(\theta):=e^{i(n+\sqrt{1+7\beta})\theta},\forall n\in\mathbb{Z}$, and recall that $\mathcal{A}=\mathcal A_{\beta,\mathfrak e}^\omega=-d^2/d\theta^2-1-7\beta/(1+\mathfrak e\cos \theta)$ with domain $\Lambda(\omega)$. We compute

\begin{equation}\label{item}
\begin{aligned}			
\mathcal{A}_{k,l}:=\langle\mathcal{A} \cdot u_k,u_l\rangle
&=2\pi\left((k+\sqrt{1+7\beta})^2-1\right)\delta_{k,l}-7\beta\int_{0}^{2\pi}\frac{\cos(k-l)\theta}{1+\fe\cos \theta}d\theta\\
&=2\pi\left((k+\sqrt{1+7\beta})^2-1\right)\delta_{k,l}-\frac{14\pi\beta}{\sqrt{1-\fe^2}}\left(\frac{\sqrt{1-\fe^2}-1}{\fe}\right)^{|k-l|}.
\end{aligned}
\end{equation}

Let $s:=(\sqrt{1-\fe^2}-1)/\fe\in(-1,0)$ and $\lambda_n:=-(\sqrt{1-\fe^2}/(14\pi\beta)) \mathcal{A}_{n,n}=-(\sqrt{1-\mathfrak e^2}/(7\beta))((n+\sqrt{1+7\beta})^2-1)+1$. Hence, we define the matrix
$$
A_N:= -\frac{\sqrt{1-\fe^2}}{14\pi \beta} \left(\mathcal{A}_{k,l}\right)_{-N\leq k,l\leq 1},
$$	
which is given by
\begin{equation}\label{equ: matrixA}
A_N:=\begin{pmatrix}
	\lambda_1 & s & s^2 & \cdots & s^{N} & s^{N+1}\\
	s & \lambda_0 & s & \cdots & s^{N-1} & s^{N}\\
	s^2 & s & \lambda_{-1} & \cdots & s^{N-2} & s^{N-1} \\
	\cdots & \cdots & \cdots & \cdots & \cdots & \cdots \\
	s^{N} & s^{N-1} & s^{N-2} & \cdots & \lambda_{-N+1} & s\\
	s^{N+1} & s^{N} & s^{N-1} & \cdots & s & \lambda_{-N}
\end{pmatrix}.
\end{equation}
We compute
$$\begin{aligned}
\lambda_1&=\frac{\sqrt{1-\fe^2}}{-14\pi\beta}\mathcal{A}_{1,1}=\frac{\sqrt{1-\fe^2}}{-7\beta}((1+\sqrt{1+7\beta})^2-1)+1,\\
\lambda_0&=\frac{\sqrt{1-\fe^2}}{-14\pi\beta}\mathcal{A}_{0,0}=-\sqrt{1-\mathfrak e^2}+1=-\fe s,\\ \lambda_{-1}&=\frac{\sqrt{1-\fe^2}}{-14\pi\beta}\mathcal{A}_{-1,-1}=\frac{\sqrt{1-\fe^2}}{-7\beta}((-1+\sqrt{1+7\beta})^2-1)+1.
\end{aligned}$$

Let $\mathfrak{n}=\mathfrak{n}(\beta):=\lfloor2\sqrt{1+7\beta}\rfloor$.
In order to prove that $m^-_\omega(\mathcal{A})\geq\mathfrak{n}+1$ for every $\beta\in[5/28,(6/\fe+3)/7]$, we need to prove that, for some $N>0$, $A_N$ has at least $ \mathfrak{n}+1$ positive eigenvalues for every $\beta\in[5/28,(6/\fe+3)/7]$.

Let $m^+(A)$ denote the total multiplicity of positive eigenvalues of matrix $A$.

\begin{lem}\label{lem: pro A}
Assume that $0<\mathfrak e< 1$, $\beta\geq 5/28$. Then $m^+(A_{\mathfrak n+1})\geq m^+(H)$, where
$$
H:=\begin{pmatrix}
	E & 0 & 0\\
	0 & K & 0\\
	0 & 0 & \hat E
\end{pmatrix},
$$
where $K:=G-F^TE^{-1}F-\hat F^T\hat E^{-1}\hat F$ is a self-adjoint $(\mathfrak n-1)\times(\mathfrak n-1)$-matrix,
\begin{equation}\label{matrix E-G}
\begin{aligned}
E&:=\begin{pmatrix}
	\lambda_1+(\lambda_0-2)s^2 & s(1-\lambda_0)\\
	s(1-\lambda_0) & \lambda_0+(\lambda_{-1}-2)s^2
\end{pmatrix}\quad
&F:=\begin{pmatrix}
0 & 0_{1\times(\mathfrak{n}-2)}\\
s(1-\lambda_{-1}) & 0_{1\times(\mathfrak{n}-2)}
\end{pmatrix},\\
\hat E&:=\begin{pmatrix}
	\lambda_0+(\lambda_{-1}-2)s^2 & s(1-\lambda_0)\\
	s(1-\lambda_0) & \lambda_1+(\lambda_0-2)s^2
\end{pmatrix}\quad
&\hat F:=\begin{pmatrix}
	0_{1\times(\mathfrak{n}-2)} & s(1-\lambda_{-1})\\
	0_{1\times(\mathfrak{n}-2)} & 0
\end{pmatrix},
\end{aligned}
\end{equation}
and $G:=(G_{ij})_{(\mathfrak n-1)\times(\mathfrak n-1)}$ is a Toeplitz matrix with $G_{ii}:=\lambda_{-1},\forall i$ and $G_{ij}:=s^{|i-j|},\forall i\neq j$.
\end{lem}
\begin{proof}
For every $\beta\geq 5/28$, we see that $\mathfrak{n}\geq3$ and
\begin{equation}\label{eig1}
\lambda_{-\mathfrak{n}-1}\geq\lambda_1,\quad \lambda_{-\mathfrak{n}}\geq\lambda_0\quad \text{and}\quad  \lambda_{k}\geq\lambda_{-1},\quad \forall -\mathfrak{n}+1\leq k\leq -2.\nonumber
\end{equation}

 Consider the symmetric $(\mathfrak n+3)\times(\mathfrak n+3)$-matrix given by
$$C:=\begin{pmatrix}
	\lambda_1 & s &  \cdots & s^{\mathfrak{n}+1} & s^{\mathfrak{n}+2}\\
	s & \lambda_0 &  \cdots & s^{\mathfrak{n}} & s^{\mathfrak{n}+1}\\
	\cdots & \cdots  & G & \cdots & \cdots\\
	s^{\mathfrak{n}+1} & s^{\mathfrak{n}} &  \cdots  & \lambda_0 & s\\
	s^{\mathfrak{n}+2} & s^{\mathfrak{n}+1} &  \cdots  & s & \lambda_1\\
\end{pmatrix}.
$$
Then $A_{\mathfrak n+1}\geq C$ so that the number of positive eigenvalues of $A_{\mathfrak n+1}$ is at least the number of positive eigenvalues of $C$.
We aim to transform the matrix $C$ into the diagonal block matrix $H$ using congruent transformations. Let
$$
P_1=\begin{pmatrix}
	1 & -s & 0\\
	0 & 1 & -s\\
	0 & 0 & 1
\end{pmatrix},\quad
P=\begin{pmatrix}
P_1 & 0 & 0\\
0 & I_{\mathfrak{n}-3} & 0\\
0 & 0 & P_1^{T}
\end{pmatrix},\quad
Q=\begin{pmatrix}
	I_2 & -E^{-1}F & 0\\
	0 & I_{\mathfrak{n}-1} & 0\\
	0 & -\hat E^{-1}\hat F & I_2
\end{pmatrix}.
$$
Here, we have used the fact that both $E$ and $E'$ are nondegenerate, as proved in Lemma \ref{lem: eig E} below. We compute
$$
D=PCP^T=\begin{pmatrix}
	E & F & 0\\
	F^T & G & \hat F^T\\
	0 & \hat F & \hat E
\end{pmatrix},\quad
H=Q^TDQ=\begin{pmatrix}
	E & 0 & 0\\
	0 & G-F^TE^{-1}F-\hat F^T\hat E^{-1}\hat F & 0\\
	0 & 0 & \hat E
\end{pmatrix}.
$$
In particular, $H$ has the same number of positive and negative eigenvalues as $C$. Hence, $A_{\mathfrak n+1}$ has at least $m^+(H)$ positive eigenvalues.
\end{proof}

The eigenvalues of $H$ are determined by $E, \hat E$, and $K$. To estimate the number of positive eigenvalues of $H$, we need the following lemma.

\begin{lem}\label{lem: eig E}
Assume that $0<\mathfrak e<1$ and $\beta>0$. The following statements hold:
\begin{itemize}
\item[(i)] $\det E=\det \hat E<0$. In particular, both $E$ and $\hat E$ have a positive and a negative eigenvalue.
\item[(ii)] If $0< \beta\leq (6/\fe+3)/7$, then $K$ is positive definite.
\end{itemize}
\end{lem}

\begin{proof}
Let us prove (i). Using that $\lambda_0=-\fe s$ and $\fe s^2+2s+\fe =0$, we obtain
$$
\begin{aligned}
(\lambda_0-2)s^2&=-(\fe s+2)s^2=-(\fe s^2+2s)s=\fe s=-\lambda_0\\
\lambda_0+(\lambda_{-1}-2)s^2&=\lambda_0+(\lambda_{-1}-\lambda_0)s^2+(\lambda_0-2)s^2=(\lambda_{-1}-\lambda_0)s^2.
\end{aligned}
$$

We thus compute
$$\begin{aligned}
	\det E=\det \hat E&=\det\begin{pmatrix}
	\lambda_1-\lambda_0 & s(1-\lambda_0)\\
	s(1-\lambda_0) & (\lambda_{-1}-\lambda_0)s^2
\end{pmatrix}\\
    &=s^2((\lambda_1-\lambda_0)(\lambda_{-1}-\lambda_0)-(1-\lambda_0)^2)\\
	&=-\frac{s^2(1-\fe^2)}{49\beta^2}(1+7\beta)(3+7\beta)<0.
\end{aligned}
$$
Here, we have used the identity $s=(\sqrt{1-\fe^2}-1)/\mathfrak e$. The proof of (i) is finished.

To prove (ii), recall that $K=G-F^TE^{-1}F-\hat F^T\hat E^{-1}\hat F$. We compute
$$\begin{aligned}
	F^TE^{-1}F&=\begin{pmatrix}
	\kappa_1 & 0_{1\times(\mathfrak{n}-2)}\\
	0_{(\mathfrak{n}-2)\times1} & 0_{(\mathfrak{n}-2)\times(\mathfrak{n}-2)}
\end{pmatrix},\quad
\hat F^T\hat E^{-1}\hat F&=\begin{pmatrix}
	0_{(\mathfrak{n}-2)\times(\mathfrak{n}-2)} & 0_{(\mathfrak{n}-2)\times1}\\
	0_{1\times(\mathfrak{n}-2)} & \kappa_1
\end{pmatrix},
\end{aligned}$$
where
$$\kappa_1:=\frac{s^2(\lambda_1-\lambda_0)(1-\lambda_{-1})^2}{\det E}
	=\frac{\sqrt{1-\fe^2}(1+2\sqrt{1+7\beta})(\sqrt{1+7\beta}-2)^2}{7\beta(3+7\beta)}\geq 0.$$
Then the smallest diagonal element of $K$ is
$$\begin{aligned}
\kappa:=\lambda_{-1}-\kappa_1=-\frac{\sqrt{1-\fe^2}(7\beta-3)}{3+7\beta}+1=\frac{2(3+7\beta s^2)}{(3+7\beta)(1+s^2)}>0.
\end{aligned}$$
Here, we have used that $\fe s^2+2s +\fe=0$.

Let $S_\kappa:=(s_{ij})_{(\mathfrak n-1)\times(\mathfrak n-1)}$ with $s_{ii}=\kappa,\forall i$ and $s_{ij}=s^{|i-j|},\forall i\neq j$. We thus have $K\geq S_{\kappa}$. To further simplify $S_\kappa$, we let
$$P=\begin{pmatrix}
		1 & -s &    &  \\
		  & \ddots &\ddots  &  \\
		  &     & 1 & -s\\
		  &     &   & 1
\end{pmatrix}.$$
If $\kappa=1$, we compute that $PS_1P^T=\mathrm{diag}(1-s^2,\cdots,1-s^2,\kappa)$ is positive-definite. When $\kappa>1$, $S_\kappa=S_1+(\kappa-1)I$ is also positive definite. For $\kappa\in(0,1)$, we have $s(1-\kappa)<0$ and
$$PS_\kappa P^T=-s(1-\kappa)\begin{pmatrix}
		\frac{\kappa+s^2(\kappa-2)}{-s(1-\kappa)} & -1 &   &   &  \\
		-1 & \frac{\kappa+s^2(\kappa-2)}{-s(1-\kappa)} & \ddots &   &  \\
		   & \ddots & \ddots & \ddots &   \\
		   &   & \ddots & \frac{\kappa+s^2(\kappa-2)}{-s(1-\kappa)} & -1\\
		   &   &   & -1 & \frac{\kappa}{-s(1-\kappa)}
\end{pmatrix}.$$
If $-(\kappa+s^2(\kappa-2))/(s(1-\kappa))\geq2$, or equivalently if $-2s/(1-s)\leq \kappa< 1$, we see that $PS_{\kappa}P^T$ is positive-definite, since one can replace all the diagonal elements by $2$ to obtain a positive-definite matrix. Therefore, we conclude that $K\geq S_{\kappa}>0$ for every $-2s/(1-s)\leq \kappa$. Finally, using $\kappa=2(3+7\beta s^2)/((3+7\beta)(1+s^2))$ and $s=(\sqrt{1-\mathfrak e^2}-1)/\mathfrak e$, we compute that $-2s/(1-s)\leq \kappa$ is equivalent to $\beta\leq(-3(1+s^2)/s+3)/7=(6/\fe+3)/7$. Hence, (ii) follows.
\end{proof}

\begin{proof}[Proof of Proposition \ref{prop: estimates of Morse index}]
Fix $\omega\in\mathbf U$ and let $\beta_0$ be the initial point at $\fe=0$ of the corresponding
$\omega$--degenerate curve, that is,
$
\beta_0=\beta_j^\omega(0)$ or $\beta_0=\beta_{j,\pm}(0),
$
and assume $\beta_0\in\big[5/28,(6/\fe+3)/7\big]$. Set
$
\fn:=\lfloor 2\sqrt{1+7\beta_0}\rfloor .
$
Recall that the $(\fn+3)\times(\fn+3)$ matrix of $\mathcal A=\mathcal A^\omega_{\beta_0,\fe}$ in the
truncated Fourier basis is
$$
(\mathcal A_{k,l})_{-\fn-1<k,l<1}
=\frac{-14\pi\beta_0}{\sqrt{1-\fe^2}}\,A_{\fn+1}.
$$
Let $\gamma_{\beta,\fe}$ be the fundamental solution of \eqref{equ: new linear equation of Euler}.
By \eqref{index equ beta, e} we have
$
i_\omega(\gamma_{\beta_0,\fe})=m^-_\omega(\mathcal A^\omega_{\beta_0,\fe}).
$
Since the scalar factor $-14\pi\beta_0/\sqrt{1-\fe^2}$ is negative, negative directions of
$\mathcal A^\omega_{\beta_0,\fe}$ correspond to positive directions of $A_{\fn+1}$, hence
$$
i_\omega(\gamma_{\beta_0,\fe})
=m^-_\omega(\mathcal A^\omega_{\beta_0,\fe})
\ge m^+(A_{\fn+1})
\ge m^+(H)
\ge \fn+1,
\qquad \forall\,\fe\in(0,1).
$$
On the other hand, by definition of the $\omega$--index (Section~\ref{sec: maslov-type indices and rotation numbers})
and \eqref{equ: gamma beta,0}, we obtain
$$
i_\omega(\gamma_{\beta_j^\omega(\fe),\fe})
=i_\omega(\gamma_{\beta_0,0})
<\Big\lfloor 2\sqrt{1+7\beta_0}\Big\rfloor+1
=\fn+1,
\qquad \forall\,\fe\in(0,1).
$$
Finally, since $\beta\mapsto i_\omega(\gamma_{\beta,\fe})$ is non-decreasing on $[0,+\infty)$
for each fixed $\fe$ (Proposition~\ref{prop: degenerate curves}-(vi)),
the two inequalities imply
$
\beta_0>\beta_j^\omega(\fe)$ or $\beta_0>\beta_{j,\pm}(\fe),
\forall\,\fe\in(0,1),
$
as claimed.
\end{proof}

\begin{thm}\label{thm: for beta>5/28}
Let $\rho_{\beta,\mathfrak e}=\rho(\gamma_{\beta,\mathfrak e})+1$ be the rotation number of the Euler orbit $\zeta_e$ on the energy surface $\mathfrak M$. For every $\fe\in(0,1)$, $\beta\in[5/28,(6/\fe+3)/7]$, we have $\rho_{\beta,\fe}>\rho_{\beta,0}=\sqrt{1+7\beta}+1$.
\end{thm}

\begin{proof}
Fix $\fe\in(0,1)$ and $\beta\in[5/28,(6/\fe+3)/7]$.
Choose $\omega\in\mathbf U$ and $j\in\mathbb Z_+$ such that
$\beta=\beta_j^\omega(0)$ $(\omega\neq -1)$ or
$\beta=\beta_{j,\pm}(0)$ $(\omega=-1).$
By Proposition~\ref{prop: estimates of Morse index}, we know that
$\beta_j^\omega(\fe)<\beta_j^\omega(0) =\beta$ or $\beta_{j,\pm}(\fe)<\beta_{j,\pm}(0) =\beta$.
By definition of the degenerate curves, the rotation numbers along $\Gamma_j^\omega$ and $\Gamma_{j,\pm}$ are constants. Hence,
$
\rho_{\beta,0}
=\rho_{\beta_j^\omega(0),0}
=\rho_{\beta_j^\omega(\fe),\fe}
$
or
$
\rho_{\beta,0}
=\rho_{\beta',\fe},\forall \beta'\in[\beta_{j,-}(\fe),\beta_{j,+}(\fe)].
$

Now we fix $\fe$. By Proposition~\ref{prop: degenerate curves}-(v), the function
$
\beta\mapsto \rho_{\beta,\fe}
$
is non-decreasing on $[0,+\infty)$, and it is strictly increasing except on the intervals
$
[\beta_{k,-}(\fe),\beta_{k,+}(\fe)], k\in\mathbb Z_+,
$
where it remains constant.
If $\beta=\beta_{j,\pm}(0)$, then $\beta_{j,+}(\fe)<\beta$ and monotonicity imply
$\rho_{\beta,0}< \rho_{\beta,\fe}$. Assume that $\beta=\beta_j^\omega(0),\,\omega\neq -1$.
Since $\beta_j^\omega(\fe)<\beta$, monotonicity gives
$
\rho_{\beta,0}
=\rho_{\beta_j^\omega(\fe),\fe}
\le \rho_{\beta,\fe}.
$
We claim that the inequality is strict. Suppose by contradiction that
$
\rho_{\beta,0}=\rho_{\beta,\fe}.
$
Then the function $\rho_{\beta,\fe}$ must remain constant on the whole interval
$
[\beta_j^\omega(\fe),\beta].
$
Therefore, this interval must be contained in one of the plateau intervals
$
[\beta_{k,-}(\fe),\beta_{k,+}(\fe)]
$
for some $k\in\mathbb Z_+$.
However, on such intervals the rotation number corresponds to the case $\omega=-1$.
This contradicts our choice $\omega\neq -1$.
Hence,
$
\rho_{\beta,0}<\rho_{\beta,\fe},
$
which completes the proof.
\end{proof}

For every $(\beta,\fe)\subset (0,1)\times (0,1)$, the inequality $\rho_{\beta,\fe}>\rho_{\beta,0}$ directly follows from Theorems \ref{thm: for 0<beta<5/28} and \ref{thm: for beta>5/28}. To complete the proof Theorem \ref{thm: main1}, it is suffice to show that $\fe\mapsto \rho_{\beta,\fe},\fe\in(0,1)$, is non-decreasing for every $\beta\in(0,1)$. To do this, we need to explore the monotonicity of the function $\beta^\omega_j$ and $\beta_{j,\pm}$ for every $\omega\in \mathbf U\setminus\{-1\}$ and $j\in \mathbb N$. Recall the self-adjoint operator $\mathcal A^\omega_{\beta,\mathfrak{e}}=-d^2/d\theta^2-1-7\beta/(1+\mathfrak{e}\cos\theta)$ with domain $\Lambda(\omega)$. Fix $\mathfrak e\in[0,1),j\in \mathbb N$, and simply denote $\beta(\mathfrak{e}):=\beta_j^{\omega}(\mathfrak{e})$ or $\beta_{j,\pm}(\mathfrak{e})$. In particular, $\ker \mathcal A^\omega_{\beta(\mathfrak{e}),\mathfrak{e}}\neq \{0\}$. Assume $x_{\mathfrak{e}}\in\ker \mathcal A^\omega_{\beta(\mathfrak{e}),\mathfrak{e}}\setminus \{0\}$, which $x_{\mathfrak{e}}$ satisfies the equation \eqref{equ: x equation}. Recall $\omega=e^{2\pi \nu i}$ and $\nu=\rho(\gamma_{\beta,\mathfrak e})-j\in[0,1)$. 
Recall from \eqref{equ: iteration of An} that $C_n(\nu)A_n=(\mathfrak{e}/2) \left(A_{n-1}+A_{n+1}\right),\forall n\in \mathbb{Z},$
where $A_n=(1-(n+\nu)^{2})a_n$, $C_n(\nu)=7\beta(\fe)/((n+\nu)^{2}-1)-1$ and $\{a_n\in\mathbb{C}\}$ are Fourier coefficients of $x_\fe$, that is, $x_{\mathfrak{e}}(\theta)=\sum_{n\in\mathbb{Z}}a_{n}e^{i(n+\nu)\theta}$. Note that $C_n(\nu)<0$ for every $\nu\in(0,1)$, $\beta\in(0,1)$ and $n\notin\{-3,-2,1,2\}$. Moreover, both real and imaginary part of the sequence $\{A_n\}$ satisfy the same relations \eqref{equ: iteration of An}. We may assume $\{a_n\}$, $\{A_n\}$ are real sequences.

For $\nu=0$, we have the following result.
\begin{lem}\label{lem: for 1-degenerate curves}
Assume $j\in\mathbb N$. For every $(\beta_j^1(\mathfrak{e}),\mathfrak{e})\in\Gamma^1_j$ in $(0,1)\times (0,1)$, we have $(\beta_j^1)'(\mathfrak{e})<0$.
\end{lem}
\begin{proof}
Since $\mathcal A^1_{\beta(\fe),\mathfrak{e}},\beta(\fe)=\beta_j^1(\mathfrak{e})$, is degenerate under the $1$-boundary condition, we have a non-trivial solution $x_{\mathfrak{e}}$ satisfying \eqref{equ: x equation} and \eqref{equ: iteration of An}.
Assume $x_{\mathfrak{e}}$ is real without loss of generality. Since $x_{\mathfrak{e}}=\bar x_{\mathfrak{e}}$ and $C_n=C_n(0)=7\beta(\mathfrak{e})/(n^2-1)-1$, we conclude that $a_{-n}=a_{n}$ and $C_{-n}=C_n$, which implies that $A_{-n}=A_n$ for every $n\in\mathbb{Z}$.
In particular, we have $A_{1}=A_{-1}=0$,
$C_{1}|A_1|^2=C_{-1}|A_{-1}|^2=0$ and $C_n|A_n|^2\leq 0$ for every $|n|\geq 3$.
Moreover, $C_{n}|A_n|^2\leq 0$ for $|n|=2$ and every $\beta(\mathfrak e)\in(0,3/7]$.

When $\beta=\beta(\mathfrak e)\in(3/7,1)$, using \eqref{equ: iteration of An}, we compute
\begin{align*}
C_{0}A_{0}&=(-7\beta-1)A_{0}=\frac{\mathfrak{e}}{2}(A_{-1}+A_{1})=0\ &\Rightarrow&\ A_{0}=a_{0}=0,\\
C_1A_1&=-7\beta a_1=\frac{\mathfrak{e}}{2}(A_{0}+A_{2})=\frac{\mathfrak{e}}{2}A_2,\ &\Rightarrow &\ \text{we assume}\ A_2=\frac{\mathfrak{e}^2}{4},\\
C_2A_2&=(\frac{7\beta}{3}-1)A_2=\frac{\mathfrak{e}}{2}(A_{1}+A_{3})=\frac{\mathfrak{e}}{2}A_3\ &\Rightarrow& \ A_3=\frac{\mathfrak{e}}{2}(\frac{7\beta}{3}-1),\\
C_3A_3&=(\frac{7\beta}{8}-1)A_3=\frac{\mathfrak{e}}{2}(A_{2}+A_{4})=\frac{\mathfrak{e}}{2}(\frac{\mathfrak{e}^2}{4}+A_4)\ &\Rightarrow& \ A_4=(\frac{7\beta}{8}-1)(\frac{7\beta}{3}-1)-\frac{\mathfrak{e}^2}{4}.
\end{align*}
We aim to show that $I_4(\beta,\mathfrak{e}^2):=\sum_{n=0}^4C_n|A_n|^2\leq0$ for every $(\beta,\mathfrak{e})\in(0,1)\times(0,1)$. We compute
$$\begin{aligned}
I_4(\beta,\mathfrak{e}^2)&=(\frac{7\beta}{3}-1)\frac{\mathfrak{e}^4}{16}+(\frac{7\beta}{8}-1)\frac{\mathfrak{e}^2}{4}(\frac{7\beta}{3}-1)^2 +(\frac{7\beta}{15}-1)((\frac{7\beta}{8}-1)(\frac{7\beta}{3}-1)-\frac{\mathfrak{e}^2}{4})^2,
\end{aligned}$$
and $\hat I_4(\beta,\mathfrak{e}^2):=480\partial_{\mathfrak{e}^2} I_4(\beta,\mathfrak{e}^2)=120 - 217 \beta - 294 \beta^2 + 343 \beta^3 - (120 - 168 \beta) \mathfrak{e}^2$. Since both $\hat I_4(\beta,0)=(-8 + 7\beta) (-3 + 7\beta) (5 + 7\beta)\leq 0$ and $\hat I_4(\beta,1)=49\beta(-1 + \beta) (1 + 7 \beta)\leq0$ are non-positive on $[3/7,1)$, we conclude that $\hat I_4(\beta,\mathfrak{e}^2)\leq 0$ and $I_4(\beta,\cdot)$ is non-increasing on $(0,1)$ for every $\beta\in [3/7,1)$.
Moreover, we compute $8640I_4(\beta,0)=(7\beta-15)(7\beta-8)^2(7\beta-3)^2< 0$ on $(3/7,1)$. Together with the case in $\beta\in(0,3/7]$, we conclude that $I_4\leq 0$ on $(0,1)\times (0,1)$. Since $\{a_{n}\}_{n>4}$ cannot be identically zero, whenever $0<\beta,\mathfrak{e}<1$, we have $\sum_{n\in\mathbb{Z}}C_n|A_n|^2<I_4$. Finally, Lemma \ref{lem: beta'e} implies that $\langle \partial_{\mathfrak{e}} \mathcal A^1_{\beta(\mathfrak{e}),\mathfrak{e}}x_{\mathfrak{e}},x_{\mathfrak{e}}\rangle <0$ and $\beta'(\mathfrak e)<0$ for every $\mathfrak e\in(0,1)$. The proof is now completed.
\end{proof}

Now we consider the case $\nu\in(0,1)$. Let $D_k(\nu):=\mathrm{diag}(C_{k+1}(\nu),C_k(\nu))$ and denote
$$
S_k(\nu):=\left\{\begin{aligned}
&I,\ k=-1,\quad M_k\cdots M_1M_0,\ k\geq 0,\\
&(M_{-1}M_{-2}\cdots M_{k+1})^{-1},\quad\ k\leq -2.
\end{aligned}\right.\quad \text{where}
\quad M_{n}:=\begin{pmatrix}2C_n(\nu)/\mathfrak{e} & -1 \\ 1 & 0\end{pmatrix}.
$$
From \eqref{equ: iteration of An}, we know that $(A_{n+1},A_n)^T=M_n(\nu)(A_n,A_{n-1})^T$ and $(A_{k+1},A_k)^T=S_k(\nu)(A_0,A_{-1})^T,\forall k\in\mathbb{Z}$.
Now we define
\begin{equation}\label{def: Q4}
Q_4(\nu):=\sum^2_{k=-2}S_{2k-1}(\nu)^TD_{2k-1}(\nu)S_{2k-1}(\nu).
\end{equation}

Unlike the case $\nu=0$, we cannot determine the initial values $(a_{-1},a_0)$ for $\nu\in(0,1)$. Instead, we provide the following lemma.
\begin{lem}\label{lem: negative definite}
Fix $\nu\in(0,1),j\in \N$, and let $\beta(\fe)=\beta_j^\omega(\fe)\ (\nu\neq 1/2)$ or $\beta_{j,\pm}(\fe)\ (\nu=1/2)$. If $Q_4(\nu)$ satisfies the following condition:
\begin{itemize}
\item[($Q_\nu$):] $Q_4(\nu)$ is non-positive definite for every $0<\beta,\mathfrak{e}<1$,
\end{itemize}
then $\beta'(\mathfrak{e})<0$ for every $(\beta(\mathfrak{e}),\mathfrak{e})\in(0,1)\times (0,1)$.
\end{lem}
\begin{proof}
Let $I_4(\beta(\mathfrak{e}),\mathfrak{e}):=\sum^{4}_{n=-5}C_i(\nu)|A_i|^2$. We rewrite $I_4(\beta,\mathfrak{e})=(A_0,A_{-1}) Q_4(\nu) (A_0,A_{-1})^T$ due to \eqref{def: Q4}. Then condition $(Q_\nu)$ shows that $\sum_{n\in \mathbb{Z}}C_n(\nu)|A_n|^2< I_4(\beta(\fe),\mathfrak{e})\leq 0$, for every $\mathfrak{e}\in(0,1)$. From Lemma \ref{lem: beta'e}, we conclude that $\beta'(\fe)<0$ and $\langle\partial_{\mathfrak{e}} \mathcal A^\omega_{\beta(\mathfrak{e}),\mathfrak{e}}\cdot x_{\mathfrak{e}},x_{\mathfrak{e}}\rangle <0$ for every $\mathfrak{e}\in(0,1)$. Hence, this lemma follows.
\end{proof}
However, it is extremely hard to check condition ($Q_\nu$). To further simplify this problem, we observe that
$C_n(\nu)=C_{-n-1}(1-\nu)$ and $A_n(\nu)=A_{-n-1}(1-\nu)$ for every $\nu\in(0,1)$, which imply that
$$
D_{n}(\nu)=ND_{-n-2}(1-\nu)N,\quad  M_n(\nu)=NM_{-n-1}^{-1}(1-\nu)N,\quad N=\begin{pmatrix}0 & 1\\ 1& 0\end{pmatrix}.
$$
Then we obtain that $S_k(\nu)=NS_{-k-2}(1-\nu)N$ and
$$S_{2k-1}(\nu)^TD_{2k-1}(\nu)S_{2k-1}(\nu)=NS_{-2k-1}(1-\nu)^TD_{-2k-1}(1-\nu)S_{-2k-1}(1-\nu)N.$$
Let $Q_4^+(\nu):=\mathrm{diag}(C_0(\nu),0)+S_{1}(\nu)^TD_{1}(\nu)S_{1}(\nu)+S_3(\nu)^TD_3(\nu)S_3(\nu)$. We conclude that
\begin{equation*}\label{equ: Q4 and Q4+}
Q_4(\nu)=Q_4^+(\nu)+NQ_4^{+}(1-\nu)N.
\end{equation*}
In order to prove $(Q_\nu)$ for every $0<\nu<1$, it is sufficient to prove the following condition:
\begin{itemize}\label{def: Q4+}
\item[($Q_4^+$)]:  $Q_4^+(\nu)$ is non-positive definite for every $\nu\in(0,1)$ and $0<\beta,\mathfrak{e}<1$.
\end{itemize}

\begin{rem}\label{rem: 0 to ((n+nu)^2-1)/7}
Recall that $C_{n}(\nu)<0$ on $\nu\in (0,1)$, for every $\beta\in(0,1)$ and $n\notin \{-3,-2,1,2\}$. Moreover, $C_n(\nu)\leq 0$ on $\nu\in (0,1)$ if $\beta\in(0,((n+\nu)^2-1)/7]$ for $n=1,2$. Then both $D_{-1}(\nu), D_3(\nu)$ are negatively definite for every $\nu\in(0,1)$ and $\beta\in(0,1)$, and $D_1(\nu)$ is negatively definite for every $\nu\in(0,1)$ and $\beta\in(0,(2\nu+\nu^2)/7]$. Therefore, we have $Q^+_4(\nu)<0$ for every $\nu\in(0,1)$ and $(\beta,\mathfrak{e})\in(0,(2\nu+\nu^2)/7]\times (0,1)$. 
\end{rem}

Now we aim to check condition ($Q_4^+$) for every $(\beta,\mathfrak{e})\in \Gamma^\omega_j\ (\omega\neq -1)$ or $(\beta,\mathfrak{e})\in \Gamma_{j,\pm}$ ($\omega=-1$) in $(0,1)\times (0,1)$, where $\omega=e^{2\pi i \nu},\nu\in (0,1)$ and $j\in \mathbb Z_+$. Due to Lemma \ref{lem: negative definite}, it is sufficient to prove that
\begin{equation}\label{equ: condition Q4+<0}
\mathrm{tr}(Q_4^+(\nu))<0 \quad \mathrm{and} \quad \det(Q_4^+(\nu))>0,\quad \forall\ 0<\beta,\mathfrak e<1,\ 0<\nu<1.
\end{equation}
For simplicity, instead of considering $Q^+_4$, it is sufficient to consider
$$\begin{aligned}
\widetilde Q^+_4=(S_1^{-1})^TQ^+_4S_1^{-1}&=(S_1^{-1})^T\mathrm{diag}(C_0,0)S_1^{-1}+D_1+(M_3M_2)^TD_3(M_3M_2)\\
&=:\frac{1}{\mathfrak{e}^4}\begin{pmatrix}
d_1 & d_2 \\ d_2 & d_3
\end{pmatrix},
\end{aligned}$$
where
$$\begin{aligned}
d_1(\mathfrak{e},\beta,\nu)&=16 C_2^2C_3^2 C_4+(4C_2^2C_3-8C_2C_3C_4)\mathfrak{e}^2+(C_0+C_2+C_4)\mathfrak{e}^4,\\
d_2(\mathfrak{e},\beta,\nu)&=\mathfrak{e}(-8C_2C_3^2C_4-2(C_0C_1+C_3(C_2-C_4))\mathfrak{e}^2),\\
d_3(\mathfrak{e},\beta,\nu)&=\mathfrak{e}^2(4C_0C_1^2+4C_3^2C_4+(C_1+C_3)\mathfrak{e}^2).
\end{aligned}$$
Denoting $\mathfrak{e}^2$ as a new variable $e$, we can compute the trace and determinant of $\widetilde Q^+_4(\nu)$ as
$$\begin{aligned}
\mathrm{tr}^+_4(e)&=e^{-2}(\mathrm{tr}_0+\mathrm{tr}_1e+\mathrm{tr}_2e^2),\\
\mathrm{det}^+_4(e)&=e^{-3}(
\mathrm{det}_1+\mathrm{det}_2e+\mathrm{det}_3e^2+\mathrm{det}_4e^3).
\end{aligned}$$
where $\mathrm{tr}_i=\mathrm{tr}_i(\beta,\nu)$ and $\mathrm{det}_j=\mathrm{det}_j(\beta,\nu)$ are
$$\begin{aligned}
\mathrm{tr}_0&=16 C_2^2C_3^2 C_4,\\ \mathrm{tr}_1&=4(C_0C_1^2+C_3^2C_4+C_2^2C_3-2C_2C_3C_4),\\
\mathrm{tr}_2&=C_0+C_1+C_2+C_3+C_4,
\end{aligned}$$
and
$$\begin{aligned}
\mathrm{det}_1&=64C_0C_1^2C_2^2C_3^2C_4,\quad
\mathrm{det}_2=16C_1C_2C_3(C_0C_1C_2-2C_0C_1C_4-2C_0C_3C_4+C_2C_3C_4),\\
\mathrm{det}_3&=4C_2C_3(C_1+C_3)(C_2-2C_4)+4(C_0+C_2+C_4) (C_0C_1^2+C_3^2C_4)-4(C_0C_1+C_3(C_2-C_4))^2,\\
\mathrm{det}_4&=(C_1+C_3)(C_0+C_2+C_4).
\end{aligned}$$
Note that $C_0,C_3,C_4<0$ for every $0\leq \beta,\nu\leq 1$. We directly obtain that $\mathrm{tr}_0<0$ and $\mathrm{det}_1>0$.

Now we introduce the following lemmas. Since the proof is rather technical, we refer the proof of Lemma \ref{lem: trace and determinant are positive} to Section \ref{sec: supplementary proofs}, where some algebraic manipulations were performed by Mathematica \cite{Mathematica}.
\begin{lem}\label{lem: trace and determinant are positive}
Assume $\nu\in(0,1)$. Then the following estimates hold:
\begin{itemize}
\item[(i)] $\mathrm{tr}^+_4(e)<0$ for every $0< \beta,e< 1$,
\item[(ii)] $\mathrm{det}^+_4(e)>0$ for every $0< \beta,e< 1$.
\end{itemize}
Moreover, $Q_4^+(\nu)$ and $Q_4(\nu)$ are negative definite for every $0< \beta,\mathfrak e< 1$.
\end{lem}
Finally, by Lemmas \ref{lem: for 1-degenerate curves}, \ref{lem: negative definite} and \ref{lem: trace and determinant are positive}, we obtain that $(\beta_j^\omega)'(\mathfrak{e})<0$ and $(\beta_{j,\pm})'(\mathfrak{e})<0$ for every $(\beta_j^\omega(\mathfrak{e}),\mathfrak{e})$ and $(\beta_{j,\pm}(\mathfrak{e}),\mathfrak{e})\in(0,1)\times (0,1)$. We know that for every $\omega=e^{2\pi\nu i}\in\mathbf U\setminus\{-1\}$, the degenerate curves satisfy $\beta^\omega_j(0)> \beta_j^\omega(\mathfrak e)$ and $\beta_{j,\pm}(0)>\beta_{j,\pm}(\mathfrak e)$ for every $\mathfrak e\in(0,1)$ and $j\in \mathbb N$. Write $\beta(\fe)$ instead of $\beta^\omega_j(\fe)$ or $\beta_{j,\pm}(\fe)$, shortly. If $\beta=\beta(0)\in(0,1)$, then by Proposition~\ref{prop: degenerate curves}-(v), we have $\rho_{\beta,\mathfrak e}\geq \rho_{\beta(\mathfrak e),\mathfrak e}=\rho_{\beta,0}$ for every $\fe\in(0,1)$. Hence, $\rho_{\beta,\mathfrak e}$ is non-decreasing in $\mathfrak e\in(0,1)$ for every fixed $\beta\in(0,1)$. The proof of Theorem \ref{thm: main1} is now complete.

\section{Proof of Theorem \ref{thm: main2}}\label{sec: estimate the volume}

We want to estimate the volume ${\rm vol}(\mathfrak M) = H^{-1}(-1)$, where $H$ is the Hamiltonian in \eqref{equ: Ham1}
$$
H(p_r,p_z,r,z)=\frac{p_r^2+p_z^2}{2}+V(r,z),\quad V(r,z):=\frac{\varpi^2}{2r^2}-\frac{1}{r}-\frac{4\alpha^{-1}}{\sqrt{r^2+(1+2\alpha)z^2}}.
$$
Recall that $\lambda$ is a contact form on $\mathfrak M$ given by the restriction of the $1$-form $\tilde \lambda:=i_Y \tilde \omega_0$ to $\mathfrak M$, where $Y$ is a Liouville vector field defined on a tubular neighborhood $\mathcal{U}\subset \R^4$ of $\mathfrak M$, $Y$ is transverse to $\mathfrak M,$ and $d\tilde \lambda= \tilde \omega_0 = dp_r \wedge dr + dp_z \wedge dz$ is the standard symplectic form. The sublevel set $\mathfrak B=H^{-1}((-\infty,-1])$ is a compact $4$-ball bounded by $\mathfrak{M}$ and hence we may assume that $\tilde \lambda$ extends to a $1$-form also denoted $\tilde \lambda$, defined on a neighborhood $\mathcal{U}'\subset \R^4$ of $\mathfrak B$, satisfying $d\tilde \lambda = \tilde\omega_0$. We thus have
\begin{equation}\label{formula_volume}
{\rm vol}(\mathfrak{M}) = \int_{\mathfrak{M}} \lambda \wedge d\lambda = \int_{\mathfrak{M}} \tilde \lambda \wedge d\tilde \lambda =\int_{\mathfrak B} \tilde \omega_0 \wedge \tilde \omega_0 =  2 \int_{\mathfrak{B}} dp_r \wedge dr \wedge dp_z \wedge dz.
\end{equation}
We see that ${\rm vol}(\mathfrak M)$ does not depend on the choice of $\lambda$.

We shall compare ${\rm vol}(\mathfrak{M})$ with the contact volume of $\mathfrak M_2=H_2^{-1}(-1)$, where
$$
H_{2}(p_r,p_z,r,z)=\frac{p_r^2+p_z^2}{2}+V_2(r,z),\quad V_2(r,z)=\frac{\varpi^2}{2r^2}-\frac{1}{\beta\sqrt{r^2+(1+7\beta)z^2}}.
$$
Recall that $\beta = (1+4\alpha^{-1})^{-1},$ where $\alpha \in (1, +\infty)$.

\begin{lem}\label{lem_volume2} Let $\mathfrak{B}_2:=H_2^{-1}((-\infty,-1])$. Then $\mathfrak B_2$ is a $4$-ball contained in $\mathfrak B \setminus \mathfrak M$. Moreover,
\begin{equation}
   {\rm vol}(\mathfrak M) > {\rm vol}(\mathfrak M_2)= 2\pi^2\frac{(1-\sqrt{2} \varpi \beta)^2}{\beta^2\sqrt{1+7\beta}}=\frac{T_e^2}{\sqrt{1+7\beta}},
\end{equation}
where
$$
T_e=2\pi \frac{1-\sqrt{2}\varpi\beta}{\sqrt{2}\beta}= 2\pi\left(\frac{1-\sqrt{1-\mathfrak e^2}}{\sqrt{2}\beta}\right).
$$
\end{lem}

\begin{proof}
It is immediate to check that $H_2$ has only one critical point, which coincides with the critical point of $H,$ a nondegenerate minimum. Since $7\beta = 7\alpha/(\alpha+4)<2\alpha,$ we have
$$
0<\partial_{z^2}V=\frac{4+2/\alpha}{\sqrt{(r^2+(1+2\alpha)z^2)^{3}}}<\partial_{z^2}V_2=\frac{4+2/\alpha}{\sqrt{(r^2+(1+7\beta)z^2)^{3}}},\quad \forall (r,z)\in \mathbb{R}^+\times(\mathbb{R}\setminus \{0\}).
$$
Since $V$ and $V_2$ coincide for $z=0$, the inequality above implies that $V(r,z)<V_2(r,z)$ for every $(r,z)\in \mathbb{R}^+\times (\mathbb{R}\setminus \{0\})$, and thus $\mathfrak{M}_2:=H_2^{-1}(-1)\subset \R^4$ is a sphere-like hypersurface bounding a $4$-ball $\mathfrak B_2 \subset \mathfrak B$. In particular, ${\rm vol}(\mathfrak M)>{\rm vol}(\mathfrak M_2).$

Both $\mathfrak{M},\mathfrak{M}_2$ project to the same disk $\Upsilon$ in the $(p_r,r)$-plane determined by
$$
\frac{p_r^2}{2} + \frac{\varpi^2}{2r^2} - \frac{1}{\beta r} \leq -1.
$$
The boundary $\partial \Upsilon$ coincides with the projection of $\zeta_e$ to the $(p_r,r)$-plane. Denote by $r_+>r_->0$ the two points given by the intersection of $\partial \Upsilon$ with the $r$-axis. They satisfy $r^2-\beta^{-1}r+\varpi^2/2=0$ and are given by $r_\pm = (1\pm \sqrt{1-2\varpi^2\beta^{2}})/(2\beta)=(1\pm \fe)/(2\beta).$

The Reeb period of $\zeta_e$ coincides with the symplectic area of $\Upsilon$. A straightforward computation gives
$$
\begin{aligned}
T_e & =\int_{\zeta_{e}}\lambda =\int_{\Upsilon}dp_r\wedge dr=
2\int_{r_{-}}^{r_{+}}\sqrt{-2-\frac{\varpi^2}{r^2}+\frac{2}{r\beta}}dr\\
& = 2\sqrt{2}\int_{r_-}^{r_+}
\frac{\sqrt{(r-r_-)(r_+-r)}}{r} dr
=\frac{1-\sqrt{2}\varpi\beta}{\sqrt{2}\beta}2\pi=2\pi\left( \frac{1-\sqrt{1-\mathfrak e^2}}{\sqrt{2}\beta} \right).
\end{aligned}
$$

To compute the contact volume of $\mathfrak M_2$, we first observe that the projection of $\mathfrak B_2$ to the $(p_r,p_z)$-plane is a disk of radius $\sqrt{-2-2V_2(r,z)}$. Moreover, the Hill region of $\mathfrak M_2$ is determined by $z\in [-z_{2,+}(r), z_{2,+}(r)], r\in [r_-,r_+]$, where
$$
z_{2,+}(r)=\frac{r}{\sqrt{1+7\beta}}\sqrt{\frac{4r^2}{\beta^2(\varpi^2+2r^2)}-1}.
$$
Hence,

$$\begin{aligned}
\mathrm{vol}(\mathfrak{M}_2)
&=\int_{\mathfrak{M}_2}\lambda\wedge d\lambda
=2\int_{\mathfrak B_2}dp_r\wedge dp_z\wedge dr\wedge dz\\
& = 8\pi \int_{r_-}^{r_+}\int_0^{z_{2,+}(r)}\left(-1-\frac{\varpi^2}{2r^2}+\frac{1}{\beta\sqrt{r^2+(1+7\beta)z^2}}\right) drdz\\
& = 2\pi^2\frac{(1-\sqrt{2} \beta \varpi)^2}{\beta^2\sqrt{1+7\beta}},
\end{aligned}
$$
where the last identity follows from a long but straightforward computation.
\end{proof}

Let us finish the proof of Theorem \ref{thm: main2}. Let $(\beta,\fe)\in (0,1) \times (0,1)$ satisfy $\beta^2 + \fe^2 <1$ and denote $\rho_{\beta,\fe}=\rho_e>2$ and $T_{\beta,\fe}=T_e$ the rotation number and the Reeb period of the Euler orbit $\zeta_e$. Theorem~\ref{thm: main1} gives us the crucial estimate $\rho_{\beta,\fe} > \rho_{\beta,0} = \sqrt{1+7\beta}+1$.
Hence $(\rho_{\beta,\fe}-1)^{-1}<1/\sqrt{1+7\beta}$. Using Lemma \ref{lem_volume2}, we obtain
$
{\rm vol}(\mathfrak{M}) > T_{\beta,\fe}^2(\rho_{\beta,\fe}-1)^{-1}$ as desired.

\subsection{Volume estimates}
We can further derive other estimates on the volume ${\rm vol}(\mathfrak M)$ depending on the parameters $(\beta,\fe)$.

\begin{prop}\label{prop_monotonicity of volume}
The following statement holds:
\begin{itemize}
\item[(i)] Fix $\varpi\in(1/\sqrt{2},+\infty)$ and let $(\beta(\vartheta),\mathfrak e(\vartheta)):=(\cos \vartheta/(\sqrt{2}\varpi),\sin \vartheta),\vartheta\in(0,\pi/2)$. The volume $\mathrm{vol}(\mathfrak M_{\beta(\vartheta),\mathfrak e(\vartheta)})$ strictly increases in $\vartheta\in(0,\pi/2)$. Moreover, $\mathrm{vol}(\mathfrak M_{\beta(\vartheta),\mathfrak e(\vartheta)})\rightarrow 0^+$ as $\vartheta\rightarrow 0^+$ and $\mathrm{vol}(\mathfrak M_{\beta(\vartheta),\mathfrak e(\vartheta)})\rightarrow +\infty$ as $\vartheta\rightarrow (\pi/2)^-$.
\item[(ii)] Fix $\mathfrak e\in(0,1)$. The volume $\mathrm{vol}(\mathfrak M_{\beta,\mathfrak e})\rightarrow +\infty$  as $\beta\rightarrow 0^+$.
\item[(iii)] Fix $\beta\in(0,1)$. The volume $\mathrm{vol}(\mathfrak M_{\beta,\mathfrak e})$ strictly increases in $\mathfrak e^2\in(0,1-\beta^2)$. Moreover, the limit of $\mathrm{vol}(\mathfrak M_{\beta,\mathfrak e})$ exists, as $\mathfrak e^2\rightarrow (1-\beta^2)^-$.
\end{itemize}
\end{prop}

\begin{proof}
Recall that $\mathfrak e=\sqrt{1-2\varpi^2\beta^2}$ and
$$
r_\pm=r_{\pm}(\varpi,\beta)=\frac{1\pm \sqrt{1-2\varpi^2\beta^2}}{2\beta},\quad
z_+=z_+(r,\beta,\varpi)=\frac{r\sqrt{1-\beta}}{\sqrt{1+7\beta}}\sqrt{\frac{4(1-\beta)^2r^2}{\beta^2(\varpi^2-2r+2r^2)}-1},
$$
where $z_+$ solves $V(r,z_{+})=-1$ for every $r\in[r_-,r_+]$.

Fix $\beta\in(0,1)$, we see that $r_+,z_+$ decrease in $2\varpi^2\in(1,\beta^{-2})$ and $r_-$ increases in $2\varpi^2\in(1,\beta^{-2})$. Let $(\beta,\mathfrak e)=(\cos \vartheta/(\sqrt{2}\varpi),\sin \vartheta)$. Fix $2\varpi^2\in(1,+\infty)$, we see that $r_+,z_+$ decrease in $\beta=\cos \vartheta/(\sqrt{2}\varpi)\in(0,1/(\sqrt{2}\varpi))$ and then increase in $\vartheta\in(0,\pi/2)$. Since $\partial_\beta r_-=(1-\sqrt{1-2\varpi^2\beta^2})/(2\beta\sqrt{1-2\varpi^2\beta^2})>0$, then $r_-$ increases in $\beta\in(0,1/(\sqrt{2}\varpi))$ and decrease in $\vartheta\in(0,\pi/2)$. Hence, the Hill region $\mathcal H=\mathcal H_{\beta,\mathfrak e}:=V^{-1}((-\infty,-1])$ satisfies $\mathcal H_{\beta,\mathfrak e}\subset \mathcal H_{\beta,\mathfrak e'}$ for every $\mathfrak e<\mathfrak e'\in (0,1)$, and
$\mathcal H_{\beta(\vartheta),\mathfrak e(\vartheta)}\subset \mathcal H_{\beta(\vartheta'),\mathfrak e(\vartheta')}$ for every $\vartheta<\vartheta'\in(0,\pi/2)$.
Now we compute
\begin{equation}\label{equ: vol of M}
\begin{aligned}
\mathrm{vol}(\mathfrak{M}_{\beta,\mathfrak e})&=\int_{\mathfrak{M}} \lambda\wedge d\lambda=2\int_{\mathfrak B} dp_r\wedge dp_z\wedge dr\wedge dz\\
&=8\pi\int_{r_-}^{r_+}\int_0^{z_+} -1-\frac{\varpi^2}{2r^2}+\frac{1}{r}+\frac{4\alpha^{-1}}{\sqrt{r^2+(1+2\alpha)z^2}}drdz.
\end{aligned}
\end{equation}
Fix $\varpi\in(1/\sqrt{2},+\infty)$. We see that $\mathrm{vol}(\mathfrak{M}_{\beta,\mathfrak e})$ strictly decreases in $\beta\in(0,1/(\sqrt{2}\varpi))$ and then strictly increases in $\vartheta\in(0,\pi/2)$. In particular, as $\vartheta\rightarrow 0^+$, the volume $\mathrm{vol}(\mathfrak{M}_{\beta,\mathfrak e})\rightarrow 0^+$, since $\mathfrak M$ converges to the nondegenerate global minimal point $(0,0,\varpi^2\beta,0)$ with minimal value $-1/(2\varpi^2\beta^2)$. As $\vartheta\rightarrow (\pi/2)^+$, we see from Theorem \ref{thm: main2} that
$$
\mathrm{vol}(\mathfrak{M}_{\beta,\mathfrak e})>2\pi^2\frac{(1-\sqrt{2}\varpi \beta)^2}{\beta^2\sqrt{1+7\beta}}\rightarrow +\infty.
$$
Hence, item (i) holds. Similarly,  (iii) follows.

Fix $\beta\in(0,1)$, we see that $\mathrm{vol}(\mathfrak{M}_{\beta,\mathfrak e})$ strictly decreases in $2\varpi^2\in(1,\beta^{-2})$ and then strictly increases in $\mathfrak e\in(0,\sqrt{1-\beta^2})$.
Finally, (ii) follows from Proposition \ref{prop: boundness of vol} below. The proof is complete.
\end{proof}

\begin{prop}\label{prop: boundness of vol}
Fix $0<\beta<1$. As $\mathfrak{e}^2\rightarrow (1-\beta^2)^-$, the limit of volume $\mathrm{vol}(\mathfrak{M}_{\beta,\mathfrak e})$ is finite.
\end{prop}
\begin{proof}
Consider the polar coordinates
$(r,z)=(\rho\cos\varphi,\rho\sin\varphi/\sqrt{1+2\alpha})$. Let $\varphi^\pm$ denote the maximal and minimal value of $\varphi$, and let $\rho^\pm(\varphi)$ denote the maximal and minimal value of $\rho$ for every $\varphi\in[\varphi^-,\varphi^+]$. We compute that
$$\rho^{\pm}(\varphi)= \frac{1\pm\sqrt{1-2\varpi^2(1+4\alpha^{-1}\cos\varphi)^{-2}}}{2((\cos\varphi)^{-1}+4\alpha^{-1})^{-1}},\quad \varphi^{\pm}=\pm\arccos\frac{\sqrt{2\varpi^2}-1}{4\alpha^{-1}}.
$$
Using \eqref{equ: vol of M} and the variables
$$a=4\alpha^{-1},\quad w=2\varpi^2\in (1,\beta^{-2}),\quad s=((\cos\varphi)^{-1}+a)^{-1}\in(s_w:=\frac{\sqrt{w}-1}{a\sqrt{w}},\beta),$$
we compute
$$\begin{aligned}
\mathrm{vol}(\mathfrak{M}_{\beta,\mathfrak e})
&=\frac{4\pi}{\sqrt{1+2\alpha}}\int_{\varphi^-}^{\varphi^+}\int_{\rho^-(\varphi)}^{\rho^+(\varphi)}-\rho-\frac{\varpi^2}{2\rho\cos^2\varphi}+\frac{1}{\cos\varphi}+4\alpha^{-1}d\rho d\varphi\\
&=\frac{4\pi}{\sqrt{1+2\alpha}}\int_{\varphi^-}^{\varphi^+}\bigg(\frac{\sqrt{1-2\varpi^2(1+4\alpha^{-1}\cos\varphi)^{-2}}}{2((\cos\varphi)^{-1}+4\alpha^{-1})^{-2}}\\
&\qquad\qquad\qquad\qquad -\frac{\varpi^2}{2\cos^2\varphi}\ln\frac{1+\sqrt{1-2\varpi^2(1+4\alpha^{-1}\cos\varphi)^{-2}}}{1-\sqrt{1-2\varpi^2 (1+4\alpha^{-1}\cos\varphi)^{-2}}}\bigg)d\varphi\\
&=\frac{4\pi}{\sqrt{1+2\alpha}}\int_{s_w}^{\beta}\frac{2\sqrt{1-w(1-as)^2}-w(1-as)^2\ln\big(\frac{1+\sqrt{1-w(1-as)^2}}{1-\sqrt{1-w(1-as)^2}}\big)}{2s^2(1-as)\sqrt{(1-as)^2-s^2}}ds.
\end{aligned}$$
Let $v=\sqrt{w}(1-as),x=\sqrt{1-v^2}\in (0,\mathfrak e)$. Since $2x-(1-x^2)\ln(\frac{1+x}{1-x})\leq 2x^3$ on $x\in[0,1]$, we have
$$\begin{aligned}
\mathrm{vol}(\mathfrak{M}_{\beta,\mathfrak e})
&\leq \frac{4\pi w^{3/2} a}{\sqrt{1+2\alpha}}\int^{1}_{\sqrt{w}\beta}\frac{\sqrt{(1-v^2)^3}}{(\sqrt{w}-v)^2v\sqrt{v^2-(\sqrt{w}-v)^2/a^2}}dv\\
&\leq \frac{4\pi w^{3/2}(1-\beta)^2}{\beta\sqrt{1+7\beta}} \int^{1}_{\sqrt{w}\beta}\frac{\sqrt{8/(\sqrt{w}\beta)}}{v\sqrt{\sqrt{w}-v}\sqrt{v-\sqrt{w}\beta}}dv\\
&= \frac{8\sqrt{8}\pi w^{3/4}(1-\beta)^2}{\beta^2\sqrt{1+7\beta}} \arctan\sqrt{\frac{1-\sqrt{w}\beta}{\beta(\sqrt{w}-1)}}\rightarrow \frac{8\sqrt{2}\pi^2(1-\beta)^2}{\beta^2\sqrt{1+7\beta}},\quad \mathrm{as}\quad w\rightarrow 1^+.
\end{aligned}$$
Therefore, the volume $\mathrm{vol}(\mathfrak M_{\beta,\mathfrak e})$ is finite as $\mathfrak e^2\rightarrow (1-\beta^2)^-$. The proof is now complete.
\end{proof}

\section{Proof of Theorem \ref{thm: main4}}\label{sec: interval}

Consider the family of symplectic diffeomorphisms $\psi_s:(\Upsilon, dp_r \wedge dr) \to (\Upsilon, dp_r \wedge dr), s\in \R, \psi_0 = {\rm Id},$ representing the flow from $\Sigma_0$ to $\Sigma_s$, where $\Upsilon\subset \R^2$ is a smooth compact disk given by the closure of the projection $\dot \Upsilon$ of each $\Sigma_s$ to the $(p_r,r)$-plane, see \eqref{equ: Upsilon}. As mentioned before, the linearized dynamics near the Euler orbit $\zeta_e$ implies that $\psi_s$ is continuous along $\partial \Upsilon$, \cite{HLOSY2023}. Then following unexpected proposition states that $\psi_{1/2}$ is a symplectic diffeomorphism on the closed disk $\Upsilon$.

\begin{prop}\label{prop: smooth_map} The first return map $\psi_1=\psi_1(p_r,r):(\Upsilon,dp_r\wedge dr) \to (\Upsilon, dp_r\wedge dr)$ is a smooth area-preserving diffeomorphism.
\end{prop}

\begin{proof} Recall that $\Sigma_s,s\in \mathbb R/\mathbb Z$, are open disks in $\mathfrak M$ with common boundary $\zeta_e$ forming an open book, and the projection of each $\Sigma_s$ to the $(p_z,z)$-plane is a line-segment through the origin with argument $2\pi s$. Denote by $\varphi^R_t:\mathfrak M\rightarrow \mathfrak M,t\in \R$, the Reeb flow on $\mathfrak M$. Let $\psi_{s,t}:=\varphi^R_{\tau_{s,t}}:\Sigma_s\rightarrow \Sigma_{t}$ be the first hitting map along the flow, where $\tau_{s,t}$ denotes the first hitting time from $\Sigma_s$ to $\Sigma_t$ for every $0\leq s<t\leq 1$. Denote by
$$
\mathcal S:=\mathfrak M\cap\{z=0\}=\{(p_r,p_z,r,0): (p_r^2+p_z^2)/2+V(r,0)=-1\},
$$
which is a smooth two-sphere in $\mathfrak M$ given by the closure of $\Sigma_0\cup\Sigma_{1/2}=\mathcal S\setminus \zeta_e$. Let $\bar \psi=\varphi^R_{\bar \tau}:\mathcal S\setminus\zeta_e \rightarrow \mathcal S\setminus\zeta_e$ be the first return map (interchanging $\Sigma_0$ and $\Sigma_{1/2}$) with first return time $\bar \tau:\mathcal S \setminus \zeta_e \to \R_{>0}$. Due to the $z$-symmetry of the system, we have
\begin{equation}\label{symmetries}
\bar \psi \circ N = N \circ \bar \psi \qquad \mbox{ and } \qquad \bar \tau = \bar \tau \circ N,
\end{equation}
where $N$ is the involution $(p_r,p_z,r,z) \mapsto (p_r,-p_z,r,-z)$. Notice that $\bar \psi^2$ is the first return map to each $\Sigma_0$ and $\Sigma_{1/2}$.

The following proposition shows that $\bar \psi$ smoothly extends to $\mathcal S$.

\begin{prop}\label{prop: smoothness 1}
The maps $\bar\psi:\mathcal S \setminus \zeta_e \to \mathcal S \setminus \zeta_e$ and $\bar \tau:\mathcal S \setminus \zeta_e \to \R_{>0}$ smoothly extend to $\mathcal S$. In particular, the identities in \eqref{symmetries} hold on $\mathcal S$.
More precisely, there exist smooth coordinates $(\rho=p_z,\vartheta)\in [-\epsilon,\epsilon] \times \R / 2\pi \Z, \epsilon>0$ sufficiently small, defined on an annular neighborhood of $\zeta_e\subset \mathcal S,$ with $\zeta_e \equiv \{\rho =0\}$ such that in these coordinates $\bar \psi$ and $\bar \tau$ are smooth and \eqref{symmetries} takes the form
\begin{equation}\label{equ: symmetries 2}
\bar \psi (-\rho,\vartheta)=N\circ \bar \psi(\rho,\vartheta) \quad \mbox{ and } \quad \bar \tau(-\rho,\vartheta) =
\bar \tau (\rho,\vartheta)
\end{equation}
for every $(\rho,\vartheta)\in[-\epsilon,\epsilon] \times \R / 2\pi \Z.$
\end{prop}

\begin{proof}
Since $\ddot z = -\partial_zV =-g(r,z)z,$ with smooth function $g(r,z)>\delta>0$, we know that $\bar \tau$ and $\bar \psi$  continuously extend to $\mathcal S$, see \cite[Section 7]{HLOSY2023} for a discussion.

We aim to show that $\bar \tau$ and $\bar \psi$ are smooth on $\mathcal S$. Consider polar coordinates $$
(p_z,z):=(\rho\cos \theta,\rho\sin \theta) \quad \mbox{ and } \quad (p_r,u) = (\ell \cos \vartheta, \ell \sin \vartheta),$$
with $u=u(r)=-\varpi/r+1/(\varpi \beta)$ (or $r=\varpi^2\beta/(1-\varpi\beta u)$). Then $\mathcal S=\{p_r^2+u^2+p_z^2=-2+1/(\varpi\beta)^2>0\}$ and $(p_z,z,\vartheta)\in[-\epsilon_0,\epsilon_0]\times [-\epsilon_0,\epsilon_0]\times \R / 2\pi \Z, \epsilon_0>0$ small, are smooth coordinates on a tubular neighborhood of $\zeta_e\subset \mathfrak M$. We also consider the two-fold coordinates $(\rho,\theta,\vartheta)\in [-\epsilon_0,\epsilon_0] \times \R / 2\pi \Z \times \R / 2\pi \Z),\,\rho\neq 0,\, \epsilon>0$ small, on this neighborhood removing $\zeta_e$. Fix $\theta=0$, the values $\rho>0$ correspond to a neighborhood of $\zeta_e$ in $\Sigma_{0}$ and the values $\rho<0$ correspond to a neighborhood of $\zeta_e$ in $\Sigma_{1/2}$, which smoothly extend to $\zeta_e$ as $\rho\rightarrow 0^\pm$. Thus $(\rho=p_z,\vartheta)\in[-\epsilon_0, \epsilon_0] \times \R / 2\pi \Z$ are smooth coordinates on a neighborhood of $\zeta_e\subset \mathcal S$.

The equations of motion \eqref{equ: Ham system} in these coordinates become
\begin{equation}\label{equ: theta and rho}
\begin{aligned}
\dot \rho & =\frac{p_z\dot p_z + z \dot z}{\rho}= \rho \left(1- \frac{4\alpha^{-1}(1+2\alpha)}{\sqrt{(r^2+(1+2\alpha)\rho^2\sin^2\theta)^{3}}}\right)\cos \theta \sin \theta,\\
\dot \theta & =\frac{p_z\dot z-z\dot p_z}{\rho^2}=\cos^2\theta+\frac{4\alpha^{-1}(1+2\alpha)}{\sqrt{(r^2+(1+2\alpha)\rho^2\sin^2\theta)^{3}}}\sin^2\theta>\min \{\delta,1\}>0,\\
\dot \vartheta & = \frac{p_r^2\dot u-u \dot p_r}{\ell^2}=\frac{\varpi}{r^2}\cos^2\vartheta+\frac{1}{\ell}\left(-\frac{\varpi^2}{r^3} +\frac{1}{r^2} +\frac{4\alpha^{-1}r}{\sqrt{(r^2+(1+2\alpha)\rho^2\cos^2\theta)^{3}}} \right) \sin \vartheta,
\end{aligned}
\end{equation}
where $r=r(u)=r(\ell\sin \vartheta)=\varpi^2\beta/(1-\varpi\beta \ell \sin \vartheta)$ is a smooth function of $\ell \sin \vartheta$, and
$\ell=\sqrt{p_r^2+u^2}>0$
is a smooth function of $(\rho,\theta, \vartheta)$ near $\rho=0$, due to implicit function theorem and
$$
\partial_\ell[H(\ell\cos\vartheta,p_z,r(\ell\sin\vartheta),z)]=l+\frac{4}{\alpha}\left(\frac{r}{\sqrt{(r^2+(1+2\alpha)z^2)^{3}}}-\frac{1}{r^2}\right)\partial_\ell r>0,\quad \mathrm{near}\ \rho =0.
$$
Hence, \eqref{equ: theta and rho} are smooth differential equations in $(\rho,\theta,\vartheta)$ near $\rho=0$ (including $\rho=0$), so that $\dot \rho$ is odd with respect to $\rho$, and the remaining equations are even with respect to $\rho$. In particular, $\theta(t)$ smoothly depends on the initial data of $(\rho,\theta,\vartheta)$. Since $\dot \theta$ is everywhere positive, we conclude from the implicit function theorem that the hitting time from $\theta=0$ to $\theta= \pi$ (or, equivalently, from $\theta=\pi$ to $\theta= 2\pi$) is a smooth function $\bar \tau = \bar \tau(\rho,\vartheta)$ defined for every $|\rho|\leq \epsilon< \epsilon_0$, for $\epsilon_0>0$ sufficiently small and $\vartheta\in \R / 2\pi \Z$. Since $\rho=p_z$ on a neighborhood of $\zeta_e \subset \mathcal S$ and we can rewrite $\bar \tau = \bar\tau(\rho,\vartheta)= \bar\tau(p_z,\vartheta)$ on $\mathcal S$. We conclude that $\bar \tau$ and $\bar \psi$ are both smooth on $\mathcal S$ in coordinates $(p_z,\vartheta)$.
\end{proof}

In smooth coordinates $(p_z,\vartheta)\in [-\epsilon, \epsilon] \times \R/ 2\pi \Z$ given in Proposition \ref{prop: smoothness 1}, the area form $\omega_0|_{\mathcal S}$, which is preserved by the map $\bar \psi=\bar \psi(p_z, \vartheta)$, vanishes at $p_z=0$. In non-smooth coordinates $(p_r,r)\in \Upsilon$, however, the area form $\omega_0|_{\mathcal S}$ is constant given by $dp_r\wedge dr$. We thus want to show that, in coordinates $(p_r,r)$, $\bar \psi$ is a smooth area-preserving diffeomorphism of $\Upsilon$.

We write $\bar\psi(p_z,\vartheta)=(\tilde p_z(p_z,\vartheta),\tilde \vartheta(p_z,\vartheta))$, defined for $(p_z,\vartheta)\in [-\epsilon, \epsilon] \times \R / 2\pi \Z.$ Due to the symmetries \eqref{equ: symmetries 2}, $\tilde p_z^2$ and $\tilde \vartheta$ are even functions of $p_z$. By Whitney's Theorem in \cite{Whitney1942} (or Schwarz \cite{Schwarz1975}), there exist smooth (even analytic) functions $f_0=f_0(t,\vartheta)$ and $f_1=f_1(t,\vartheta)$, defined in $[0,\epsilon) \times \R / 2\pi \Z$ for some $\epsilon>0$ small, such that $\tilde p_z^2(p_z,\vartheta) = f_0(p_z^2,\vartheta)$ and $\tilde \vartheta(p_z,\vartheta) = f_1(p_z^2,\vartheta)$.

Denote by $\pi_{s}:\overline{\Sigma}_s\rightarrow \Upsilon,s\in \R / \Z,$ the homeomorphism given by the projection $\overline{\Sigma}_s \ni (p_r,p_z,r,z)\mapsto (p_r,r)\in \Upsilon$.  Due to \eqref{symmetries}, the maps $\bar \psi|_{\overline{\Sigma}_0}$ and $\bar \psi|_{\overline{\Sigma}_{1/2}}$ induce the same area-preserving homeomorphism $\bar g:\Upsilon \to \Upsilon$ given by
$$
\bar g=\bar g(p_r,r):=\pi_{1/2} \circ \bar \psi|_{\overline{\Sigma}_0} \circ \pi_0^{-1} = \pi_{0} \circ \bar \psi|_{\overline{\Sigma}_{1/2}} \circ \pi_{1/2}^{-1}.
$$
We want to show that $\bar g$ is indeed a smooth diffeomorphism.

Instead of coordinates $(p_r,r)$, we consider coordinates $(\ell, \vartheta) \in [\ell^*-\epsilon,\ell^*] \times \R / 2\pi \Z$, $\epsilon>0$ small, on a neighborhood of $\partial \Upsilon \subset \Upsilon$, where $\ell= (p_r^2 +u(r)^2)^{1/2}$ and $u(r) = -\varpi/r+1/(\varpi \beta)$ are as in the proof of Proposition \ref{prop: smoothness 1}, and $\ell=\ell^* := (1/(\varpi\beta)^2-2)^{1/2}>0$ corresponds to $\partial \Upsilon$. Notice that $(\ell,\vartheta)$ are smooth coordinates with respect to $(p_r,r)$ on a neighborhood of $\partial \Upsilon\subset \Upsilon$. It is then enough to show that $\bar g$ is smooth in coordinates $(\ell,\vartheta)$.

We write $\bar g(\ell,\vartheta)=(\bar \ell(\ell,\vartheta),\bar \vartheta(\ell,\vartheta))$.  Because of the mechanical nature of the Hamiltonian $H$, the following crucial algebraic relation holds on $\mathcal{S}$
\begin{equation}\label{pz_ell}
p_z^2+\ell^2 = -2 + \frac{1}{\varpi^2\beta^2}.
\end{equation}
Hence, we obtain
$$
\bar \vartheta(\ell,\vartheta)=\tilde \vartheta(p_z(\ell,\vartheta),\vartheta)=f_1(p_z^2(\ell,\vartheta),\vartheta)=f_1\left(1/(\varpi\beta)^2-2-\ell^2,\vartheta\right),
$$
which is a smooth function of $(\ell,\vartheta)$.
To see that $\bar \ell=\bar \ell(\ell,\vartheta)$ is also smooth, we observe that
$$
\begin{aligned}
\bar \ell(\ell,\vartheta) & =(1/(\varpi \beta)^2-2-\tilde p_z^2(p_z,\vartheta))^{1/2}=(1/(\varpi \beta)^2-2-f_0(p_z^2(\ell,\vartheta),\vartheta))^{1/2}\\
& = (1/(\varpi \beta)^2-2-f_0(1/(\varpi\beta)^2-2-\ell^2,\vartheta))^{1/2},
\end{aligned}
$$
which is a positive smooth function near $\ell = \ell^*$. Hence, $\bar g$ is smooth in coordinates $(\ell,\vartheta)$ and thus also smooth in coordinates $(p_r,r)$. We conclude that the first return map to $\Sigma_0$, given by $\psi=\bar g^2$ in coordinates $(p_r,r)\in \Upsilon$, is also smooth. This finishes the proof of Proposition \ref{prop: smooth_map}. \end{proof}

\begin{rem}\label{rem: smooth of psi_s}
Following the ideas in the proof of Proposition \ref{prop: smooth_map}, it is possible to show that every map $\psi_s: \Upsilon \to \Upsilon, s\in \R,$ is a smooth symplectic diffeomorphism in coordinates $(p_r,r)$.  To see this, we let $\mathcal S_{s}:=\mathfrak M\cap \{\arg(p_z+iz)=2\pi s\}=\bar \Sigma_{s}\cup \bar \Sigma_{s+1/2},s\in \R/\Z$. We see that $(\rho,\vartheta)\in [-\epsilon,\epsilon]\times \R/2\pi \Z$ are smooth coordinates on a neighborhood of $\zeta_e\subset \mathcal S_s$ for every $s$.
As in Proposition \ref{prop: smoothness 1}, the first hitting time from $\mathcal S$ to $\mathcal S_s$ is a smooth function $\bar \tau_{s}=\bar \tau_s(\rho=p_z,\vartheta):\mathcal S\rightarrow \R_{>0}$ defined for every $|\rho|<\epsilon<\epsilon_0$ and $\vartheta\in \R/2\pi \Z$, where $\epsilon_0>0$ is sufficiently small. Therefore, $\bar \tau_s:\mathcal S \to \R_{>0}$ and $\bar \psi_s=\varphi^R_{\bar \tau_s}:\mathcal S\rightarrow \mathcal S_s$ are both smooth. We write $\bar \psi_s=(\tilde \rho_s(p_z,\theta),\tilde \vartheta_s(p_z,\theta))$ defined for $(p_z,\vartheta)\in [-\epsilon, \epsilon] \times \R / 2\pi \Z.$ The symmetries in \eqref{symmetries} imply that $\tilde \rho_s^2(p_z,\theta)$ and $\tilde \vartheta_s(p_z,\theta)$ are even analytic functions of $p_z$. Therefore, by Whitney's theorem, we can find smooth functions $f_{0,s}(t,\vartheta)$ and $f_{1,s}(t,\vartheta)$ so that
$\tilde \rho_s^2(p_z,\vartheta)=f_{0,s}(p_z^2,\vartheta)$ and $\tilde \vartheta_s(p_z,\vartheta)=f_{1,s}(p_z^2,\vartheta)$.

As in the proof of Proposition \ref{prop: smooth_map}, we consider coordinates $(\ell,\vartheta)$ on a neighborhood of $\partial \Upsilon \subset \Upsilon$, which are smooth with respect to $(p_r,r)$. We write $\psi_s(\ell,\vartheta)=(\bar \ell_s(\ell,\vartheta),\bar \vartheta_s(\ell,\vartheta))$. Similar to \eqref{pz_ell}, the following algebraic relation holds on $\mathcal S_s$
$$
\rho^2\cos^22\pi s+\ell^2=-2+\frac{8}{\alpha}\left(\frac{1}{\sqrt{r^2+(1+2\alpha)\rho^2\sin^2 2\pi s}}-\frac{1}{r}\right)+\frac{1}{(\varpi\beta)^2},
$$
where $r=r(u)=r(\ell \sin \vartheta)$ is a smooth function of $(\ell, \vartheta)$. Hence, by the implicit function theorem we obtain a smooth function $\ell_s=\ell_s(\rho,\vartheta)$ such that $\ell=\ell_s(\rho^2,\vartheta)>0$ on $\mathcal S_s$ near $\zeta_e$. Therefore, we obtain $\bar \vartheta_s(\ell,\vartheta)=\tilde \vartheta_s(p_z(\ell,\vartheta),\vartheta)=f_{1,s}(p_z^2(\ell,\vartheta),\vartheta)=f_{1,s}(1/(\varpi\beta)^2-2-\ell^2,\vartheta)$, which is a smooth function of $(\ell,\vartheta)$. Moreover,
we have
$$
\begin{aligned}
\bar \ell_s(\ell,\vartheta) & = \ell_s(\tilde \rho^2(p_z,\vartheta),\tilde \vartheta(p_z,\vartheta))= \ell_s(f_{0,s}(p_z^2,\vartheta),f_{1,s}(p_z^2,\vartheta))\\
& = \ell_s(f_{0,s}(1/(\varpi\beta)^2 -2 -\ell^2,\vartheta), f_{1,s}(1/(\varpi\beta)^2-2-\ell^2,\vartheta)),
\end{aligned}
$$ which is also a positive smooth function near $\ell=\ell^*>0$. Hence, $\psi_s$ is smooth in coordinates $(\ell,\vartheta)$ near $\partial \Upsilon$ and thus $\psi_s$ is smooth in coordinates $(p_r,r)$ on $\Upsilon$ for every $s$. This proves the claim. We conclude that $\psi_s:\Upsilon \to \Upsilon,s\in[0,1],$ is a smooth isotopy of diffeomorphisms preserving the area form $dp_r\wedge dr$, satisfying $\psi_0={\rm Id}$ and $\psi_1=\psi$. This isotopy admits a natural smooth extension for  $s\in \R$ such that $\psi_{s+1} = \psi_s \circ \psi_1$ for every $s\in \R$.
\end{rem}

Before continuing the proof of Theorem \ref{thm: main4}, we introduce some useful results from Hutchings \cite{Hutchings2016}, Pirnapasov \cite{Pirnapasov2021} and Bechara \cite{Bechara2023} to our setting. The following theorem is a generalized version of Hutchings' mean action theorem.

\begin{thm}[Hutchings \cite{Hutchings2016}, Pirnapasov \cite{Pirnapasov2021}] \label{thm: Hutchings2016}
Let $\phi:\D \to \D$ be a smooth diffeomorphism preserving the area form $\omega_0=dx\wedge dy$, where $x+iy$ are coordinates in $\D$. Let $\phi_t:\D \to \D,t\in [0,1],$ be a smooth isotopy of $\omega_0$-preserving diffeomorphisms satisfying $\phi_0 = {\rm Id}$ and $\phi_1 = \phi$. Let $\alpha$ be a smooth primitive of $\omega_0$ on $\D$, and let $\sigma:\D \to \R$ be the smooth function satisfying $d\sigma=\phi^* \alpha - \alpha$ and $\sigma(z) = \int_{c_z} \alpha$ for every $z\in \partial \D$, where $c_z$ is the curve $[0,1] \ni t \mapsto \phi_t(z)\in \partial \D$. Denote by $\sigma_\infty(z)$ the almost everywhere well-defined mean action given by $\sigma_\infty(z) := \lim_{n \to +\infty} (1/n) \sum_{i=0}^{n-1} \sigma(\phi^i(z)),z\in \D$. Let ${\rm CAL}(\phi_t):=(1/\pi)\int_\D \sigma \omega_0$ and ${\rm Rot}(\phi_t)$ be the Calabi invariant and the boundary rotation number associated with the isotopy $\phi_t,t\in [0,1]$. Let $\mathcal P_\phi\subset \D$ be the set of periodic points of $\phi$. If
$
{\rm CAL}(\phi_t) > \pi {\rm Rot}(\phi_t)$, then $ \sup\{ \sigma_\infty(z), z\in\mathcal{P}_\phi\}\geq \mathrm{CAL}(\phi_t).
$
\end{thm}

The following useful lemma extends a smooth diffeomorphism of $\D$ to a larger disk $\D_{1+\epsilon}$ so that it is the indentity near $\partial \D_{1+\epsilon}$.

\begin{lem}[Pirnapasov \cite{Pirnapasov2021}]\label{lem: extension 0}
Let $\phi:\D \to \D$ be a smooth diffeomorphism preserving the area form $\omega_0 = dx \wedge dy$. Let $\phi_t:\D \to \D,t\in [0,1],$ be a smooth isotopy of $\omega_0$-preserving diffeomorphisms satisfying $\phi_0 = {\rm Id}$ and $\phi_1 = \phi$. Assume that the boundary rotation number  ${\rm Rot}(\phi_t) \in (-1,1)$. Let $\alpha_0=(r^2/2)d\theta$ be the primitive of $\omega_0$ in polar coordinates, and let $\sigma:\D \to \R$ be a smooth function satisfying $d\sigma=\phi^* \alpha_0 - \alpha_0$ and $\sigma(z) = \int_{c_z} \alpha_0$ for every $z\in \partial \D$, where $c_z$ is the curve $[0,1] \ni t \mapsto \phi_t(z)\in \partial \D$. Denote by $\sigma_\infty(z)$ the almost everywhere well-defined mean action given by $\sigma_\infty(z) := \lim_{n \to +\infty} (1/n) \sum_{i=0}^{n-1} \sigma(\phi^i(z))$,  $z\in \D$. Given $0<\epsilon<1$, there exists a smooth $\omega_0$-preserving diffeomorphism $f_\epsilon:\D_{1+\epsilon} \to \D_{1+\epsilon}:= \{z\in \C: |z| \leq 1+\epsilon\}$, such that the following properties hold:
\begin{itemize}

\item $f_\epsilon|_{\mathbb{D}}= \phi$;

\item $f_\epsilon$ is the identity map near $\partial \mathbb{D}_{1+\epsilon}$;

\item Let $\sigma_\epsilon: \D_{1+\epsilon} \to \R$ be the action of $f_\epsilon$ satisfying $d\sigma_\epsilon=f_\epsilon^*\alpha_0-\alpha_0$ and $\sigma_\epsilon=0$ near $\partial \mathbb{D}_{1+\epsilon}$. Let $\sigma_{\epsilon,\infty}$ be the associated almost everywhere well-defined mean action. Then
\begin{itemize}
\item[(i)] $\sigma_{\epsilon}|_{\D}=\sigma$. In particular, $\sigma_{\epsilon,\infty}|_{\D}=\sigma_\infty$ and $\sigma_{\epsilon,\infty}|_{\partial \mathbb{D}}=\sigma_{\infty}|_{\partial \mathbb{D}}\equiv \pi{\rm Rot}(\phi_t)$.
\item[(ii)] $|\sigma_\epsilon|, |\sigma_{\epsilon,\infty}|\leq 4\pi$ on $\mathbb D_{[1,1+\epsilon]}=\{z\in \C: 1\leq |z|\leq 1+\epsilon\}$.
\end{itemize}
\end{itemize}
\end{lem}

\begin{proof}[Sketch of the proof of Lemma \ref{lem: extension 0}] In polar coordinates near $\partial \D$, we write $f:(r,\theta)\mapsto (R,\Theta)$, where $\Theta$ is a real-valued function and $\Theta|_{\partial \D}=\Theta_1(1,\cdot)$ is determined by the isotopy $\phi_t|_{\partial \D}:(1,\theta) \mapsto (1,\Theta_t(1,\theta)), \theta\in \R, t\in [0,1]$, with $\Theta_0={\rm Id}$. Since ${\rm Rot}(\phi_t) \in (-1,1)$, we have $|\Theta(1,\theta) - \theta| <2\pi$ for every $\theta\in \R$. Since $R(1,\theta)=1$, $\partial_\theta R(1,\theta)=0$ and $\partial_\theta\Theta(1,\theta)\in (c,1/c)$ for some $0<c<1$, the preservation of $\omega_0$ implies that $\partial_r R>c/2$ near $r=R=1$. Hence, we can write $r=r(R,\theta)$ near $r=R=1$, and thus there exists a generating function $G=G(R,\theta)$ defined near $\partial \D$ satisfying
\begin{equation*}
dG=\frac{R^2}{2}d\Theta-\frac{r^2}{2}d\theta+d\left(\frac{R^2}{2}(\theta-\Theta)\right)=R(\theta-\Theta)dR+\left(\frac{R^2}{2}-\frac{r^2}{2}\right)d\theta.
\end{equation*}
In particular, $\partial_RG=R(\theta-\Theta)$ and $\partial_{\theta}G=\frac{1}{2}(R^2-r^2)$. Since $\partial_\theta G|_{R=1}\equiv 0$, we may assume that $G|_{\partial\mathbb{D}}\equiv 0$. Also, the function $\sigma(r,\theta):=G(R(r,\theta),\theta)+\frac{1}{2}R(r,\theta)^2(\Theta(r,\theta)-\theta)$ is an action of $\phi$ with respect to $\alpha_0$, defined near $\partial \D$, and the associated mean action $\sigma_\infty$ along $\partial \D$ is  constant equal to $\pi{\rm Rot}(\phi_t)\in (-\pi, \pi)$.

Given $\epsilon>0$ sufficiently small, we extend $G=G(R,\theta)$ to a $C^1$-function $G_\epsilon: \mathbb{D}_{[1-\epsilon,1+\epsilon]}\rightarrow \mathbb{R}$,  by defining $G_\epsilon(R, \theta):=u(R)(\theta-\Theta(1,\theta))$, where $u:[1,1+\epsilon]\rightarrow [0,1]$ is a smooth function satisfying $u(R)=\frac{1}{2}(R^2-1)$ near $R=1$, $u'(R)\in[-c,1+\epsilon]$ for every $R\in [1,1+\epsilon]$ and $u\equiv 0$ near $R=1+\epsilon$. In this way, $\partial_R G_\epsilon, \partial_\theta G_\epsilon$ and $\partial^2_{R\theta}G_\epsilon$ are continuous along $\partial \D$.

Take a smooth generating function $W$ that is $C^1$-close to $G_\epsilon$ on $\mathbb{D}_{[1,1+\epsilon]}$, coincides with $G_\epsilon$ on $\D$ and on a small neighborhood of $\partial\D_{1+\epsilon}$, and satisfies $\|W-G_\epsilon\|_{C^1}+\|\partial^2_{R \theta}W-\partial^2_{R\theta}G_\epsilon\|_{C^0}<\delta$ for any apriori given $\delta>0$, see \cite[Lemma 3.3]{Pirnapasov2021}. This is always possible by an approximation theorem due to Whitney \cite{Whitney1934}. For every $R\in[1,1+\epsilon]$, we have
$$\begin{aligned}
 |\partial^2_{R\theta}W(R,\theta)-R| & \geq |\partial^2_{R\theta}G_\epsilon(R,\theta)-R|-\delta=|u'(R)(1-\partial_\theta\Theta(1,\theta))-R|-\delta\\
&\geq\left\{
\begin{aligned}
&1-(1+\epsilon)(1-c)-\delta\geq c/2, &&\partial_\theta \Theta(1,\theta)\in(c,1),\\
&1-c(1/c-1)-\delta\geq c/2, &&\partial_\theta \Theta(1,\theta)\in [1,1/c),
\end{aligned}
\right.
\end{aligned}$$
for every $0<\delta,\epsilon<c/4$.
Hence, by the implicit function theorem,  the equation $\partial_\theta W(R,\theta)=\frac{1}{2}(R^2-r^2)$ solves for $R=R(r,\theta)$. The equation $\partial_RW(R,\theta)=R(\theta-\Theta)$ then gives $\Theta=\Theta(r,\theta)$, obtaining the local diffeomorphism $f_\epsilon:(r,\theta)\mapsto (R(r,\theta),\Theta(r,\theta))$ preserving $\omega_0$. By standard degree theory, $f_\epsilon$ is a diffeomorphism.

Finally, we obtain an action $\sigma_\epsilon: \D_{1+\epsilon} \to \R$ associated with $f_\epsilon$ and $\alpha_0$ which coincides with  $W(R(r,\theta),\theta)+\frac{1}{2}R(r,\theta)^2(\Theta(r,\theta)-\theta)$ near $\partial \D$. In particular, it coincides with $\sigma$ on $\D$. Since $\Theta(r, \theta)=\theta$ for every $\theta\in \R$ and $r$ sufficiently close to $1+\epsilon$, we see that $\sigma_\epsilon$ vanishes near $\partial \D_{1+\epsilon}$.    This gives (i). Moreover, we have the following estimate on $\D_{[1,1+\epsilon]}$
$$
\begin{aligned}
|\sigma_\epsilon(r,\theta)|&\leq |u(R)(\theta-\Theta(1,\theta))+(R^2/2)(\Theta(r,\theta)-\theta)|+\delta\leq 2\pi+(R/2)|\partial_RW|+ \delta\\
&\leq 2\pi+(R/2)|\partial_RG_\epsilon|+2\delta\leq 2\pi+(R/2)|u'(R)(\theta-\Theta(1,\theta))|+2\delta\leq 4\pi,
\end{aligned}$$
which implies (ii).
\end{proof}

\begin{lem}[Bechara \cite{Bechara2023}] \label{lem: Fathi and Bechara}
Let $\phi:\D \to \D$ be a smooth diffeomorphism that preserves the area form $\omega_0 = dx \wedge dy$ and coincides with the identity map near $\partial \D$. Let $\phi_t:\D\rightarrow \D,t\in[0,1]$, be a smooth isotopy of $\omega_0$-preserving diffeomorphisms satisfying $\phi_0={\rm Id}$ and $\phi_1=\phi$, and so that every $\phi_t$ coincides with the identity map on a uniform neighborhood of $\partial \D$. Let $w_\infty(z,w,\phi_t)$ denote the mean relative winding number of $z\neq w\in \D$ with respect to the isotopy $\phi_t,t\in[0,1]$, as defined in \eqref{equ: winding number}. Let $\alpha$ be a primitive of $\omega_0$, and let $\sigma:\D \to \R$ be a smooth function satisfying $d\sigma=\phi^* \alpha - \alpha$ and $\sigma|_{\partial \D}\equiv 0$. Denote by $\sigma_\infty(z)$ the almost everywhere well-defined mean action given by $\sigma_\infty(z) = \lim_{n \to +\infty} (1/n) \sum_{i=0}^{n-1} \sigma(\phi^i(z)), z\in \D$. Then
$$
\sigma_\infty(z)=\int_{\mathbb{D}}w_\infty(z,w ,\phi_t)\omega_0(w),\quad \mathrm{a.e.}\quad z\in\mathbb{D}.
$$
\end{lem}

\begin{proof}[Sketch of the proof of Lemma \ref{lem: Fathi and Bechara}] Associated to the isotopy $\phi_t,t\in[0,1]$, there exists a time dependent Hamiltonian $H_t:\D \to \R,t\in[0,1],$ which vanishes on a uniform neighborhood of $\partial \D$. Denote by $X_t$ the Hamiltonian vector field of $H_t$ determined by $-dH_t=\iota_{X_t}\omega_0$, and let $\tilde X(z,t):=X_t(z)+\partial_t$ be the time-independent vector field on $\mathbb D\times [0,1]$. Then $\iota_{\tilde X}\tilde \omega_0=0$, where $\tilde \omega_0:=\omega_0-dH_t\wedge dt$. Consider the $1$-form $\tilde \alpha:=\alpha-H_tdt$ on $\D \times [0,1]$. Then $d\tilde \alpha=\tilde \omega_0$. Let $\tilde \phi_t:=(\phi_t, t), t\in [0,1]$, denote the flow of $\tilde X$. For every  $z\in \mathbb D\equiv \D \times\{0\}$, denote by $\Gamma(z,\tilde \phi_t)\subset \D \times [0,1]$ the trajectory $[0,1]\ni t\mapsto \tilde \phi_t(z)$. Given $p\in \partial \D$ and $z\in \D \setminus \partial \D$, let $S^p(z,\tilde \phi_t):=\{(z',t): z'\in [p,\tilde \phi_t(z)],t\in[0,1]\}$ be the surface consisting of a family of line-segments  $[p,\tilde \phi_t(z)] \subset \D \times \{t\}, t\in [0,1],$ connecting $p$ to $\tilde \phi_t(z)$. Given $p\in \partial \D$ and $z\neq w\in \D \setminus \partial \D$, let $I^p(z,w,\tilde \phi_t)$ denote the algebraic intersection number between $\Gamma(w,\tilde \phi_t)$ and $S^p(z,\tilde \phi_t)$. Here, $\D \times [0,1]$ is oriented by $dx\wedge dy\wedge dt>0$, the orientation of $\Gamma(w,\tilde \phi_t)$ is given by the flow, and the orientation of $S^p(z,\tilde \phi_t)$ is the one induced by $\Gamma(z,\tilde \phi_t)$. Let $h:S^p(z,\tilde \phi_t)\rightarrow \mathbb D\ ((z',t)\mapsto \phi_{-t}(z'))$. Then one can prove that $h^*\omega_0=\tilde \omega_0|_{S^p(z,\tilde \phi_t)}$, $I^p(z,w,\tilde \phi_t)$ coincides with the degree of $h$ at $w$, i.e.
$\deg (h,S^p(z,\tilde \phi_t),w):=\sum_{\hat w\in h^{-1}(w)}\mathrm{sgn}\det dh(\hat w)$, for every regular value $w\in \D$ of $h$. Notice that, by Sard's theorem, the regular values of $h$ form a full measure subset of $\D$.  By Stokes' theorem, we obtain
\begin{equation}\label{equ: I and sigma 1}
\begin{aligned}
\int_{\mathbb D} I^p(z,w,\tilde \phi_t)\omega_0(w)&=\int_{\D}\deg (h,S^p(z,\tilde \phi_t),w)\omega_0(w)=\int_{S^p(z,\tilde \phi_t)}\tilde \omega_0\\
&=\sigma(z)+\int_{[p,z]}\alpha-\int_{[p,\phi(z)]}\alpha.
\end{aligned}
\end{equation}
Consider the $n$-iterated isotopy $\tilde \phi_{nt},t\in[0,1]$, satisfying $\tilde \phi_n=\phi^n$. We obtain from \eqref{equ: I and sigma 1} that
$$
\frac{1}{n}\int_{\mathbb D} I^p(z,w,\tilde \phi_{nt})\omega_0|_z=\frac{1}{n}\int_{S^p(z,\tilde \phi_{nt})}\tilde \omega_0=\frac{1}{n}\left(\sum_{i=0}^{n-1}\sigma(\phi^i(z))+\int_{[p,z]}\alpha-\int_{[p,\phi^n(z)]}\alpha\right),\quad \forall n\in \mathbb N.
$$
Since $|\int_{[p,z]}\alpha|,|\int_{[p,\phi^n(z)]}\alpha|$ and $|w_\infty(z,w,\phi_t)-\frac{1}{n}I^p(z,w,\tilde \phi_{nt})|$ are uniformly bounded, see \cite[Proposition 2.10]{Bechara2023},  Lemma \ref{lem: Fathi and Bechara} follows by taking $n \to \infty$.
\end{proof}

Continuing the proof of Theorem \ref{thm: main4}, we first combine Lemmas \ref{lem: extension 0} and \ref{lem: Fathi and Bechara} to prove the following lemma.

\begin{lem}\label{lem: generalize Bechara}
Consider the isotopy $\psi_s:\Upsilon\rightarrow \Upsilon,s\in[0,1]$, of the first return map $\psi=\psi_1:\Upsilon\rightarrow \Upsilon$ as in Remark \ref{rem: smooth of psi_s}, so that each $\psi_s$ is a smooth diffeomorphism preserving the area form $dp_r\wedge dr$. The boundary $\partial \Upsilon$ admits a rotation number ${\rm Rot}(\psi_s)=(\rho_e-1)^{-1}$, see \cite[Lemma 7.7]{HLOSY2023}. The first return Reeb time $\tau=\tau_{0,1}\circ \pi_0^{-1}: \Upsilon\rightarrow \R_{>0}$ is the action of $\psi$ with respect to the smooth one-form $\bar \lambda$ given by $\bar \lambda := (\pi_0^{-1})^*(\lambda|_{\Sigma_0})$. In particular, $\tau$ satisfies $d\tau=\psi^*\bar \lambda-\bar\lambda$ and $\tau(z) = \int_{c_z} \bar \lambda$ for every $z\in \partial \Upsilon$, where $c_z(s) = \psi_s(z), s\in [0,1]$. Denote by $\tau_\infty$ the almost everywhere well-defined mean action given by $\tau_{\infty}(z):=\lim_{n\to +\infty} (1/n) \sum_{i=0}^{n-1} \tau(\psi^i(z)),z\in \D$. Then
\begin{equation}\label{equ: generalize Bechara}
\tau_\infty(z)=\int_{\Upsilon}w_\infty(z,w)(dp_r\wedge dr)(w),\quad \textrm{a.e.}\quad z\in\mathbb{D}.
\end{equation}
\end{lem}

\begin{proof}
The contact form $\lambda$ on $\mathfrak M$ is given by $\lambda=(p_rdr+p_zdz-\hat \epsilon d(p_r\partial_rV+p_z\partial_zV))|_{\mathfrak M}$, where $\hat \epsilon>0$ is fixed sufficiently small, see \cite{HZ94}. Hence, $\bar \lambda := (\pi_0^{-1})^*(\lambda|_{\Sigma_0})=p_rdr-\hat \epsilon d(p_r\partial_rV(r,0))$, which is a smooth primitive of $dp_r\wedge dr$, and well-defined in a neighborhood of $\Upsilon\subset \C$.

By changing coordinates using Moser's trick,  we identify $(\Upsilon,dp_r\wedge dr)\equiv (\mathbb D,\omega_0'=(T_e/\pi)\omega_0)$, where $\omega_0=dx\wedge dy$ and $z=x+iy\in \mathbb D$.
Write $\psi_s(r,\theta)=(R_s(r,\theta),\Theta_s(r,\theta))$ in polar coordinates near $\partial \D$, where $\Theta_s$ is a real-value function and $\Theta_s|_{\partial\D}=\Theta_s(1,\cdot)$ is determined by $\Theta_0(1,\cdot)=\mathrm{Id}$. For every $k\in \Z$, we choose the isotopy $\psi_{s,k}=e^{2\pi k is}\circ \psi_s:(r,\theta)\mapsto (R_s,\Theta_s+2\pi ks),s\in[0,1]$, whose boundary rotation number is given by $(\rho_e-1)^{-1}+k$. Hence, $|\Theta_{1,k}(1,\theta)-2\pi \lfloor (\rho_e-1)^{-1}+k\rfloor-\theta|<2\pi$ for every $\theta\in \mathbb R$.
Let $k_0:=\lfloor (\rho_e-1)^{-1} \rfloor$.
Since the isotopy $\psi_{s,-k_0},s\in[0,1]$, admits boundary rotation number $\mathrm{Rot}(\psi_{s,-k_0})=(\rho_e-1)^{-1}-k_0\in(-1,1)$, we obtain from Lemma \ref{lem: extension 0} a smooth map $\psi_{\epsilon}:\D_{1+\epsilon}\rightarrow \D_{1+\epsilon}$ so that $\psi_{\epsilon}|_{\D}=\psi_{1,-k_0}=\psi$ and $\psi_{\epsilon}\equiv \mathrm{Id}$ near $\partial \mathbb D_{1+\epsilon}$ for every $\epsilon>0$ sufficiently small, and $\tau_\infty-T_ek_0$ is the mean action associated with the one-form $\alpha_0 =T_e r^2 d\theta/(2\pi)$ and the isotopy $\psi_{s,-k_0},s\in[0,1]$, so that $(\tau_\infty-T_ek_0)|_{\partial \D}=T_e\mathrm{Rot}(\psi_{s,-k_0})\in (-T_e,T_e)$. Indeed, the mean action does not depend on the choice of the primitive. Moreover, we also obtain from Lemma \ref{lem: extension 0} the mean action $\tilde \tau_\infty:\mathbb{D}_{1+\epsilon}\rightarrow \R$ of $\psi_\epsilon$ satisfying $\tilde \tau_\infty|_{\mathbb D}=\tau_\infty-T_ek_0$. In particular, $\tilde \tau_\infty|_{\partial \mathbb D}=T_e\mathrm{Rot}(\psi_{s,-k_0})$ and $\tilde \tau_\infty\equiv 0$ near $\partial \mathbb D_{1+\epsilon}$.

Extend the isotopy $\psi_{s,-k_0},s\in[0,1]$, to an isotopy $\psi_{\epsilon,s}:\D_{1+\epsilon}\rightarrow \D_{1+\epsilon},s\in[0,1]$, continuously, so that each $\psi_{\epsilon,s}$ is a homeomorphism, $\psi_{\epsilon,0}=\mathrm{Id}$ and $\psi_{\epsilon,1}=\psi_\epsilon$. Using the relation $\psi_{\epsilon,s+1} = \psi_{\epsilon,s} \circ \psi_{\epsilon,1}, s\in \R,$ we can extend $\psi_s$ to an isotopy defined for every $s\in \R$. Denote by $w_\infty^\epsilon(z,w),z\neq w\in \D_{1+\epsilon}$, the mean relative winding number of the isotopy
$\psi_{\epsilon,s},s\in[0,1]$, as in \eqref{equ: winding number}. In particular, $w_\infty^\epsilon(z,w)+k_0 =w_\infty(z,w)$ for almost every $z\neq w\in \D$. Applying Lemma~\ref{lem: Fathi and Bechara}, we conclude that
\begin{equation}\label{equ: integral splitting}
\begin{aligned}
&\quad\ \tau_\infty(z)=\tilde\tau_\infty(z)+T_ek_0=\frac{1}{(1+\epsilon)^2} \int_{\mathbb{D}_{1+\epsilon}} w_\infty^\epsilon(z,w)\omega'_0(w)+T_ek_0\\
&= \frac{1}{(1+\epsilon)^2} \left(\int_{\mathbb{D}} w_\infty(z,w)\omega'_0(w) + \int_{\mathbb{D}_{[1,1+\epsilon]}} w_\infty^\epsilon(z,w)\omega'_0(w)\right)+\frac{\epsilon(2+\epsilon)}{(1+\epsilon)^2}T_ek_0,\quad \forall z\in \mathbb{D}.
\end{aligned}
\end{equation}
Since $\mathbb{D}_{[1,1+\epsilon]}$ is $\psi_{\epsilon,s}$-invariant, the arguments of $\psi_{\epsilon,s}(w)-\psi_{\epsilon,s}(z)$ and $\psi_{\epsilon,s}(w)=\psi_{\epsilon,s}(w)-0$ differ by at most $\pi/2$ for every $(z,w)\in \mathbb D\times \mathbb{D}_{(1,1+\epsilon]}$, and $s\in\mathbb R$. Then  $w^\epsilon_\infty(z,w)=w^\epsilon_\infty(0,w)$ due to the definition of the mean relative winding number.

To estimate $w^\epsilon_\infty(0,w)$, we consider the generating function $W_1(R_1^\epsilon,\theta)$ of the extended map $\psi_\epsilon(r,\theta)=(R^\epsilon_1(r,\theta),\Theta^\epsilon_1(r,\theta))$ in polar coordinates, which coincides with $(R_1(r,\theta),\Theta_1(r,\theta)-2\pi k_0)$ on $\D$, and satisfies
$$
dW_1=\frac{(R_1^\epsilon)^2}{2}d\Theta^\epsilon_1-\frac{r^2}{2}d\theta+d\left(\frac{(R_1^\epsilon)^2}{2}(\theta-\Theta^\epsilon_1)\right) =R_1^\epsilon(\theta-\Theta^\epsilon_1)dR_1^\epsilon+\left(\frac{(R_1^\epsilon)^2}{2}-\frac{r^2}{2}\right)d\theta.
$$
In particular, $W_1$ satisfies $R_1^\epsilon(\theta-\Theta^\epsilon_1)=\partial_{R_1^\epsilon }W_1$ and $\frac{1}{2}((R_1^\epsilon)^2-r^2)=\partial_\theta W_1$. Due to the proof of Lemma \ref{lem: extension 0}, $W_1$ is $C^1$-close to the function $G_\epsilon(R_1^\epsilon,\theta)=u(R_1^\epsilon)(\theta-\Theta^\epsilon_1(1,\theta))$ on $\D_{[1,1+\epsilon]}$ satisfying $\|W_1-G_\epsilon\|_{C^1}+\|\partial_{R_1^\epsilon\theta}^2W_1-\partial_{R_1^\epsilon\theta}^2G_\epsilon\|_{C^0}<\delta$ for any apriori given $\delta>0$, where $u$ is a smooth function on $[1,1+\epsilon]$ so that $u(R_1^\epsilon)=\frac{1}{2}((R_1^\epsilon)^2-1)$ near $R_1^\epsilon=1$, $u(R_1^\epsilon)\equiv 0$ near $1+\epsilon$ and $u'(R_1^\epsilon)\in[-c,1+\epsilon]$ for some $c\in(0,1)$. We obtain the following estimate on $\mathbb{D}_{[1,1+\epsilon]}$
$$
\begin{aligned}
|\Theta^\epsilon_1(r,\theta)-\theta| & =\frac{|\partial_{R_1^\epsilon}W_1(R_1^\epsilon(r,\theta),\theta)|}{R_1^\epsilon(r,\theta)}\leq |u'(R_1^\epsilon)(\theta-\Theta^\epsilon_1(1,\theta))|+\delta\\
& \leq (1+\epsilon)|\theta-\Theta^\epsilon_1(1,\theta)|+\delta=(1+\epsilon)|\theta-\Theta_1(1,\theta)+2\pi k_0|+\delta\leq 4\pi.
\end{aligned}
$$
This implies that $|w^\epsilon_\infty(z,w)|=|w^\epsilon_\infty(0,w)|\leq 2$ for every $(z,w)\in \mathbb{D}\times\mathbb{D}_{(1,1+\epsilon]}$. Thus, as $\epsilon\rightarrow 0$, we obtain \eqref{equ: generalize Bechara} from the inequality \eqref{equ: integral splitting}. The proof is complete.
\end{proof}

After changing coordinates using Moser's trick,  we identify $(\Upsilon,dp_r\wedge dr)\equiv (\mathbb D,\omega_0'=(T_e/\pi)\omega_0)$, where $\omega_0=dx\wedge dy$ and $z=x+iy\in \mathbb D$. By Proposition \ref{prop: smooth_map}, the first return map $\psi=\psi(x,y)$ associated with $\Sigma_0\subset \mathfrak M$ is a smooth diffeomorphism preserving $\omega_0'$. Moreover $\psi = \psi_1$, where $\psi_s,s\in \R,$ is a smooth family of symplectic diffeomorphisms of $\D$, each representing the first hitting map from $\Sigma_0$ to $\Sigma_s$, see Remark \ref{rem: smooth of psi_s}. In particular, the points in $\partial \D$ admit a rotation number with respect to $\psi_s,s\in[0,1]$, given by $(\rho_e-1)^{-1}$, see \cite[Lemma 7.7]{HLOSY2023}. By Theorem \ref{thm: main2}, we have $\mathrm{vol}(\mathfrak M)>T_e^2(\rho_e-1)^{-1}$. Recall that $\tau := \tau_{0,1}\circ \pi_0^{-1}:\D \to [0,+\infty)$ is the first return Reeb time (action) associated with the disk-like global surface of section $\Sigma_0\subset \mathfrak M$. Notice that $\tau$ satisfies $d\tau =\psi^*\bar \lambda- \bar \lambda$, where $\bar \lambda := (\pi_0^{-1})^*(\lambda|_{\Sigma_0})=p_rdr-\hat \epsilon d(p_r\partial_rV(r,0))$, which is a smooth primitive of $\omega_0'$, see the proof of Lemma \ref{lem: generalize Bechara}.

To apply Hutchings' mean action theorem, we need some normalization. Instead of $\omega_0'$, we consider the area-form $\omega_0$ on $\D$, then the primitive $\bar \lambda$ of $\omega_0'$ and the action $\sigma$ get multiplied by $\pi/T_e$. Thus the Calabi invariant of $\psi_s, s\in [0,1],$ satisfies
$$
{\rm CAL}(\psi_s) :=\frac{1}{\pi}\int_\D \frac{\pi}{T_e}\tau \omega_0= \frac{1}{T_e} \int_\D \tau \omega_0= \frac{\pi}{T_e^2} \int_\D \tau \omega_0'=\frac{\pi}{T_e^2} {\rm vol}(\mathfrak M)>\pi \frac{1}{\rho_e-1}=\pi{\rm Rot}(\psi_s).
$$
Therefore, Theorem \ref{thm: Hutchings2016} implies that for every $\epsilon>0$ small, there exists a periodic point $z_{1,\epsilon}\in \mathbb D$ with prime period $k_1$ under $\psi$, such that
$$
\frac{\pi}{T_e}\tau_\infty(z_{1,\epsilon})>\frac{\pi}{T_e^2} (\mathrm{vol}(\mathfrak M)-\epsilon)>\pi \frac{1}{\rho_e-1},
$$ for every $\epsilon>0$ sufficiently small. Recall that $\tau_\infty|_{\partial \D}$ is a constant given by $T_e(\rho_e-1)^{-1}.$ In particular, $z_{1,\epsilon} \in \D \setminus \partial \D$.

Lemma \ref{lem: generalize Bechara} gives
$$
\tau_\infty(z_{1,\epsilon})=\int_{\mathbb D} w_\infty(z_{1,\epsilon},z_2) \omega_0'(z_2)>\frac{{\rm vol}(\mathfrak M) - \epsilon}{T_e}.
$$
Hence, there exists $\hat z_{2,\epsilon}\in\mathbb{D}\setminus\{z_{1,\epsilon}\}$, not necessarily periodic, such that $w_\infty(z_{1,\epsilon},\hat z_{2,\epsilon})$ is well-defined and $$w_\infty(z_{1,\epsilon},\hat z_{2,\epsilon})>\frac{\mathrm{vol}(\mathfrak M)-\epsilon}{T_e^2}.$$ Since $({\rm vol}(\mathfrak M) -\epsilon)/T_e^2>(\rho_e-1)^{-1}=w_\infty(z_{1,\epsilon},w),\forall w\in \partial \D$, for every $\epsilon>0$ sufficiently small, we may assume that $\hat z_{2,\epsilon} \in \D \setminus (\partial \D\cup \{z_{1,\epsilon}\})$. Since the orbit of $z_{1,\epsilon}$ has measure zero, we may assume that $\hat z_{2,\epsilon}$ does not lie in the orbit of $z_{1,\epsilon}$.

Finally, we consider the iterated map $\psi^{k_1}$, which makes $z_{1,\epsilon}$ a fixed point, and consider the isotopy $\psi_{k_1s},s\in[0,1]$, of $\psi^{k_1}$, under which the points in $\partial \D$ admit a rotation number given by $k_1(\rho_e-1)^{-1}$. Define the mean relative winding number $w_\infty(z,w,\psi_{k_1s}), z\neq w\in \D$ as in \eqref{equ: winding number} with $\psi_s$ replaced with $\psi_{k_1s}$, and satisfying
$$\begin{aligned}
w_\infty(z,w,\psi_{k_1s}) &= k_1(\rho_e-1)^{-1}, \forall z\in \D \setminus \partial \D, w\in \partial \D, \ \ \mbox{ and } \\
w_\infty(z,w,\psi_{k_1s}) &= k_1 w_\infty(z,w),\ \mathrm{a.e.}\ z\neq w\in \D.
\end{aligned}$$
We obtain
$$
w_\infty(z_{1,\epsilon},\hat z_{2,\epsilon},\psi_{k_1s})=k_1 w_\infty(z_{1,\epsilon},\hat z_{2,\epsilon})> k_1 \frac{{\rm vol}(\mathfrak M)-\epsilon}{T_e^2} >k_1 \frac{1}{\rho_e-1} = w_\infty(z_{1,\epsilon},w,\psi_{k_1s}),\quad \forall w\in \partial \D.
$$

Choose a smooth family of area-preserving diffeomorphism $h_s:\D \to \D,s\in \R$, such that $h_{s+1}=h_s$, $h_0=\mathrm{Id}$, $h_s(\psi_{k_1s}(z_{1,\epsilon}))=z_{1,\epsilon}$ and $h_s$ is the identity on a uniform neighborhood of $\partial \D$. Consider the smooth family of area-preserving diffeomorphisms $\hat \psi_{s}:=h_s\circ \psi_{k_1s},s\in \R,$ which is also an isotopy of $\hat \psi_1 =\psi_{k_1}= \psi^{k_1}$,  starting from the identity and satisfying $\hat \psi_{s+1} = \hat \psi_s \circ \hat \psi_1$ for every $s\in \R$. By construction the isotopies $\hat \psi_s$ and $\psi_{k_1s}$ are in the same homotopy class and $\hat \psi_s(z_{1,\epsilon})=z_{1,\epsilon}$ for every $s\in \R$. Hence, the definition of the relative winding number \eqref{equ: winding number} gives
$$\begin{aligned}
w_\infty(z_{1,\epsilon},\hat z_{2,\epsilon},\hat \psi_{s})&=w_\infty(z_{1,\epsilon},\hat z_{2,\epsilon},\psi_{k_1s})=k_1 w_\infty(z_{1,\epsilon},\hat z_{2,\epsilon})>k_1\frac{\mathrm{vol}(\mathfrak M)-\epsilon}{T_e^2},\\
w_\infty(z_{1,\epsilon},w,\hat \psi_{s})&=w_\infty(z_{1,\epsilon},w,\psi_{k_1s})=k_1 \frac{1}{\rho_e-1}, \quad \forall w\in \partial \D.
\end{aligned}$$

Blowing up the family $\hat \psi_s, s\in \R,$ at $z_{1,\epsilon}$, we obtain a unique continuous family of area-preserving annulus homeomorphisms $f_s:\R/ \Z \times [0,1] \to \R / \Z \times [0,1],s\in\R,$ which uniquely lifts to a continuous family of area-preserving homeomorphisms $\tilde f_s: \R \times [0,1] \to \R \times [0,1]$ satisfying $f_0 = {\rm Id}$. Moreover, $\R/\Z \times \{0\} \equiv z_{1,\epsilon},$ $\R/\Z \times \{1\} \equiv \partial \D$, and the (horizontal) rotation number of $ \hat z_{2,\epsilon} \equiv \tilde z_{2,\epsilon}\in \R \times (0,1)$ and $w\in \R \times \{1\}$ with respect to $\tilde f_s$ coincides with $ k_1 w_\infty(z_{1,\epsilon},\hat z_{2,\epsilon})$ and $k_1 (\rho_e-1)^{-1}$, respectively.

The following theorem by Franks is crucial to find periodic orbits in the complement of $z_{1,\epsilon}$ with prescribed relative winding numbers.

\begin{thm}[{Franks \cite{ Franks1988a, Franks1988b, Franks03}}]\label{thm:  Franks}
Consider the annulus $\mathbb A:=\R/\Z\times [0,1]$, equipped with the standard area form $\bar \omega_0:=du\wedge dv$ in coordinates $(u,v)\in \mathbb A$.
Let $f:\mathbb A \to \mathbb A$ be a homeomorphism, isotopic to the identity map and preserving $\bar \omega_0$. Let $\tilde f_s:\R \times [0,1]\to \R \times  [0,1],s\in[0,1],$ be an isotopy of $f$ so that $\tilde f_0=\mathrm{Id}$ and $\tilde f_1$ is a lift of $f$. Extend $\tilde f_s$ to an isotopy $\tilde f_s, s\in \R,$ satisfying $f_{s+1} = f_s \circ f_1$ for every $s\in \R$. Given $x\in \mathbb A$, let $\mathrm{Rot}(x,\tilde f_s):=\lim_{s\rightarrow +\infty}P_1(\tilde f_s(\tilde x)-\tilde x)/s\in \R$ be the (horizontal) rotation number of $x\in \mathbb A$ with respect to the isotopy $\tilde f_s,s\in[0,1]$, where $\tilde x\in \R\times [0,1] $ is any lift of $x$, and $P_1: \R \times [0,1] \to \R$ is the projection into the first factor. If there exist $x_0,x_1\in \mathbb{A}$ and $p/q\in \mathbb{Q}$ such that
$\mathrm{Rot}(x_0,\tilde f_s) \leq  p/q \leq \mathrm{Rot}(x_1,\tilde f_s),$ then $f$ has at least two distinct periodic orbits of period $q$ and rotation number $p/q$. Moreover, the set of numbers $\{\mathrm{Rot}(x,\tilde f_s): x\in \mathcal P_f\}$ includes all rational numbers in the  interval $[\min \mathrm{Rot}(\cdot,\tilde f_s),\max \mathrm{Rot}(\cdot,\tilde f_s)]$. Here, $\mathcal P_f$ denotes the set of periodic points of $f$.
\end{thm}

By Theorem \ref{thm:  Franks}, every rational number $p/q$ satisfying
$$
k_1\cdot \frac{p}{q}\in k_1\cdot \left(\frac{1}{\rho_e-1}, w_\infty(z_{1,\epsilon},\hat z_{2,\epsilon})\right) \supset k_1\cdot\left(\frac{1}{\rho_e-1}, \frac{{\rm vol}(\mathfrak M)-\epsilon}{T_e^2}\right),
$$
is realized by some $\tilde z_{2,\epsilon}'\in \R \times (0,1)$, which projects to a $q$-periodic point $z_{2,\epsilon}\in \mathbb A\setminus \partial \mathbb A$ under $\psi^d$. This point corresponds to a $k_1q$-periodic point of $\psi$, still denoted $z_{2,\epsilon}\in \D \setminus (\partial \D\cup \{z_{1,\epsilon}\})$. Also, $z_{2,\epsilon}$ corresponds to a periodic orbit $\zeta_{2,\epsilon}\subset \mathfrak M\setminus \zeta_e$ which is geometrically distinct to the periodic orbit $\zeta_{1,\epsilon}\subset \mathfrak M\setminus \zeta_e$ corresponding to $z_{1,\epsilon}$.  Moreover,
$$
\begin{aligned}
w_\infty(z_{1,\epsilon}, z_{2,\epsilon}) & = \frac{1}{k_1}w_\infty(z_{1,\epsilon},z_{2,\epsilon}, \psi_{k_1s}) = \frac{1}{k_1}w_\infty(z_{1,\epsilon},z_{2,\epsilon}, \hat \psi_{s})\\ &  = \frac{1}{k_1}{\rm Rot}(\tilde z_{2,\epsilon}',\tilde f_s)=\frac{p}{q}\in \left(\frac{1}{\rho_e-1}, w_\infty(z_{1,\epsilon}, \hat z_{2,\epsilon})\right)\supset\left(\frac{1}{\rho_e-1}, \frac{{\rm vol}(\mathfrak M)-\epsilon}{T_e^2}\right).
\end{aligned}
$$
The proof of Theorem \ref{thm: main4} is complete by taking $\epsilon \to 0^+$.

\section{Proof of Theorem \ref{thm: main5}} \label{sec: unbounded}

Consider the Hamiltonian $H=H(p_r,p_z,r,z)$ as in \eqref{equ: Ham1}.
Hamilton's equations are
\[
\left\{\begin{aligned}
 \dot r = p_r,\qquad
 \dot p_r & = \frac{\varpi^2}{r^3}-\frac{1}{r^2}
            -\frac{4r}{\alpha\sqrt{(r^2+(1+2\alpha)z^2)^3}},\\
 \dot z = p_z,\qquad
 \dot p_z & = -\frac{4(1+2\alpha)z}{\alpha\sqrt{(r^2+(1+2\alpha)z^2)^3}}.
\end{aligned}\right.
\]

We introduce McGehee-type coordinates adapted to the region $z\gg 1$ by setting
$$
(x,y,u,v)
=
\Big(
\Big(\frac{8}{\alpha\sqrt{1+2\alpha}\,z}\Big)^{1/2},
\ p_z,\
\frac{\varpi}{r}-\frac{1}{\varpi},\
p_r
\Big)\in \R_{>0}\times\R^3,
\qquad
\frac{dt}{d\mathfrak t}=K:=\frac{16}{\alpha\sqrt{1+2\alpha}}.
$$
In these variables one has $r=r(u)=\varpi^2/(1+u\varpi)>0$, and the energy constraint
$H=-1$ becomes
\[
-1=\frac{y^2+u^2+v^2}{2}-\frac{1}{2\varpi^2}
     -\frac{x^2}{2\sqrt{\frac{\alpha^2}{64}\,r(u)^2\,x^4+1}}.
\]

The equations of motion take the form
\begin{equation}\label{equ: Ham system 2-new}
\left\{\begin{aligned}
\frac{dx}{d\mathfrak t}&=-x^3y,\\
\frac{dy}{d\mathfrak t}&=-\Big(\frac{\alpha^2}{64}r(u)^2x^4+1\Big)^{-\frac{3}{2}}x^4,\\
\frac{du}{d\mathfrak t}&=-\frac{\varpi K}{r(u)^2}\, v,\\
\frac{dv}{d\mathfrak t}&=\frac{\varpi K}{r(u)^2}\, u
-\Big(\frac{\alpha^2}{64}r(u)^2x^4+1\Big)^{-\frac{3}{2}}\frac{\alpha r(u)x^6}{8\sqrt{1+2\alpha}}.
\end{aligned}\right.
\end{equation}
By reversibility, if $(x,y,u,v)(\mathfrak t)$ solves \eqref{equ: Ham system 2-new},
then $(x,-y,u,-v)(-\mathfrak t)$ also solves \eqref{equ: Ham system 2-new}.

In polar coordinates $(u,v)=(\rho\cos\theta,\rho\sin\theta)$ the energy constraint
can be written as
\begin{equation}\label{equ: initial condition-new}
y^2-x^2+\rho^2
=g_x(\rho,\theta)
:=\frac{1-2\varpi^2}{\varpi^2}
 -x^2\left(1-\frac{1}{\sqrt{\frac{\alpha^2}{64}\,r(\rho\cos\theta)^2\,x^4+1}}\right).
\end{equation}
Moreover, \eqref{equ: Ham system 2-new} implies
\begin{equation}\label{equ: Ham system 3-new}
\left\{\begin{aligned}
\frac{d\rho}{d\mathfrak t}
&= -\Big(\frac{\alpha^2}{64}r(\rho\cos\theta)^2x^4+1\Big)^{-\frac{3}{2}}
\frac{\alpha\,r(\rho\cos\theta)\,x^6\sin\theta}{8\sqrt{1+2\alpha}},\\[0.3em]
\frac{d\theta}{d\mathfrak t}
&= \frac{\varpi K}{r(\rho\cos\theta)^2}
-\Big(\frac{\alpha^2}{64}r(\rho\cos\theta)^2x^4+1\Big)^{-\frac{3}{2}}
\frac{\alpha\,r(\rho\cos\theta)\,x^6\cos\theta}{8\sqrt{1+2\alpha}\,\rho}.
\end{aligned}\right.
\end{equation}

The invariant ``sphere at infinity'' corresponds to $x=0$. Restricting
\eqref{equ: initial condition-new} to $x=0$ gives
$$
\mathcal S_{+\infty}
=
\left\{x=0,\ y^2+\rho^2=\rho_\infty^2\right\},
\qquad
\rho_\infty^2:=\frac{1-2\varpi^2}{\varpi^2}>0.
$$
Its equator $y=0$ is the periodic orbit
$$
\zeta_{+\infty}
=
\{(x,y,u,v)=(0,0,\rho_\infty\cos\theta,\rho_\infty\sin\theta):
\ \theta\in\R/2\pi\Z\}.
$$

Since $d\theta/d\mathfrak t>0$ in a neighborhood of $\zeta_{+\infty}$,
we may use $\theta$ as a local transverse coordinate. For $\delta>0$ small,
let $\mathbb D_\delta^+:=\{(x,y)\in\mathbb D_\delta:x\ge 0\}$ be a standard half disk with radius $\delta>0$ and define
$$
W^\delta_\theta:=
\left\{(x,y,u,v)\in\mathfrak M: (x,y)\in\mathbb D_\delta^+, (u,v)=(\rho\cos\theta,\rho\sin\theta)\ \text{for some }\rho\ge 0\right\},
\quad \theta\in\R/2\pi\Z.
$$
Then $W_\theta^\delta$ is a smooth embedded surface transverse to the flow,
and $(x,y)$ gives a smooth chart on $W^\delta_\theta$.

\begin{prop}[{McGehee \cite[Theorem 1, Proposition 11]{McGehee1973}}]\label{prop_McGehee}
The periodic orbit $\zeta_{+\infty}$ admits
analytic local stable and unstable manifolds
$
W^{\rm s}_{\rm loc}(\zeta_{+\infty})$ and $
W^{\rm u}_{\rm loc}(\zeta_{+\infty}),$
contained in $\{z>1/\delta\}$ for $\delta>0$ sufficiently small.
Moreover, for every $\theta\in\R/2\pi\Z$,
the intersection
$W^{\rm s,u}_{\rm loc}(\zeta_{+\infty})\cap W^\delta_\theta$
is the graph of a real-analytic function
$y=\bar y^{\rm s,u}_\theta(x)$,
defined for $x\in[0,\delta_1]$, for some $0<\delta_1<\delta$,
with
$(\bar y^{\rm s}_\theta)'(0)=1$ and
$(\bar y^{\rm u}_\theta)'(0)=-1.$
\end{prop}

Fix $x_0\in(0,\delta_1]$.
Define the transversal disks
$$
\Sigma_{x_0}^\pm
:=\{x=x_0,\ \pm y>0\}\cap\mathfrak M,
$$
which are smooth open disks transverse to the flow of
\eqref{equ: Ham system 2-new}.
By Proposition~\ref{prop_McGehee},
the local stable and unstable manifolds of $\zeta_{+\infty}$
intersect $\Sigma_{x_0}^\pm$ along real-analytic circles
$C^{\rm s}_{x_0}\subset\Sigma_{x_0}^+$ and $
C^{\rm u}_{x_0}\subset\Sigma_{x_0}^-,$
bounding closed disks $D^{\rm s}_{x_0}\subset\Sigma_{x_0}^+$ and
$D^{\rm u}_{x_0}\subset\Sigma_{x_0}^-$, respectively, corresponding to trajectories that escape to $z=+\infty$ (or, equivalently, converge to $x=0^+$) as $\mathfrak t\to \pm \infty$.
Moreover, trajectories starting from the open annulus $A^{\rm s}_{x_0}:=\Sigma_{x_0}^+\setminus D^{\rm s}_{x_0}$
return to $A^{\rm u}_{x_0}:=\Sigma_{x_0}^-\setminus D^{\rm u}_{x_0}$.

In polar coordinates $(u,v)=(\rho\cos\theta,\rho\sin\theta)$, the disk $D^{\rm s}_{x_0}$ is parametrized by
$$
D^{\rm s}_{x_0}: \quad 0 \leq \rho\leq \rho^{\rm s}_{x_0}(\theta), \quad \theta\in \R / 2\pi \Z,$$
where $\rho^{\rm s}_{x_0} = \rho^{\rm s}_{x_0}(\theta)>0$ is a real-analytic function satisfying for $y=\bar y^{\rm s}_\theta(x_0)>0$
$$(\bar y^{\rm s}_\theta(x_0))^2-x_0^2+(\rho^{\rm s}_{x_0}(\theta))^2=g_{x_0}(\rho^{\rm s}_{x_0}(\theta),\theta), \quad \forall \theta \in \R / 2\pi \Z.$$  The annulus $A^{\rm s}_{x_0}$ is then parametrized by
$$
A^{\rm s}_{x_0}: \quad \rho^{\rm s}_{x_0}(\theta) < \rho \leq \rho^{\rm s}_{\rm max}(\theta), \quad \theta \in \R / 2\pi \Z,
$$
where $\rho^{\rm s}_{\rm max}=\rho^{\rm s}_{\rm max}(\theta)>\rho^{\rm s}_{x_0}(\theta)$ is a real-analytic function satisfying for $y=0$
$$-x_0^2+(\rho^{\rm s}_{\rm max}(\theta))^2=g_{x_0}(\rho^{\rm s}_{\rm max}(\theta),\theta), \quad \forall \theta \in \R / 2\pi \Z.$$

For any initial condition
$$
(x_0,y_0,u_0,v_0)\in A^{\rm s}_{x_0},
\qquad
(u_0,v_0)=(\rho_0\cos\theta_0,\rho_0\sin\theta_0),
$$
let $\zeta(\mathfrak t)$ be the corresponding solution of
\eqref{equ: Ham system 2-new}.
Denote by $T_{x_0}(\rho_0,\theta_0)>0$ the first hitting time such that
$\zeta\bigl(T_{x_0}(\rho_0,\theta_0)\bigr)\in \Sigma_{x_0}^-,$
and define the net angular variation
$$
\Theta_{x_0}(\rho_0,\theta_0)
:= \theta\bigl(T_{x_0}(\rho_0,\theta_0)\bigr)-\theta_0,
$$
where $\theta(\mathfrak t)=\arg(u(\mathfrak t)+iv(\mathfrak t))$
is chosen continuously along the trajectory.
Notice that $\Theta_{x_0}(\rho_0,\theta_0)=0$ for $\rho_0=\rho^{\rm s}_{\rm max}(\theta_0)$.

\begin{figure}[hbpt]
    \centering
    \includegraphics[width=0.4\linewidth]{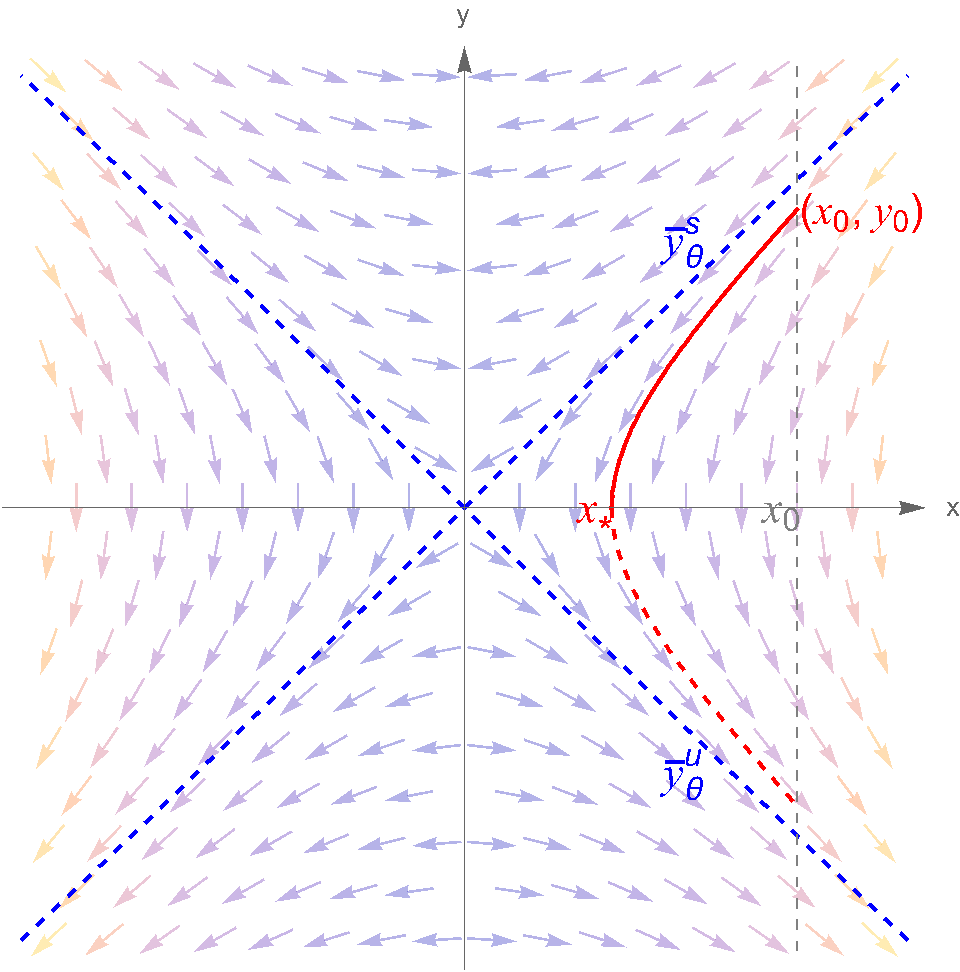}
    \caption{The Hamiltonian vector field of $H$ in $(x,y)$-coordinates.}
    \label{fig:limit-vf}
\end{figure}

\begin{prop}\label{prop: local twist}
Fix $x_0>0$ sufficiently small and $\theta_0\in\R/2\pi\Z$.
Then
$$
T_{x_0}(\rho_0,\theta_0)\to +\infty
\quad\text{and}\quad
\Theta_{x_0}(\rho_0,\theta_0)\to +\infty
\qquad
\text{as }\ \rho_0\to\left(\rho^{\rm s}_{x_0}(\theta_0)\right)^+.
$$
\end{prop}

\begin{proof}
Fix $\theta_0$ and consider initial conditions in $A^{\rm s}_{x_0}=\Sigma_{x_0}^+\setminus D^{\rm s}_{x_0}$
which in polar coordinates is given by $(\rho_0,\theta_0)$.
Let $x_*=x_*(\rho_0,\theta_0)\in(0,x_0)$ be the minimum value of $x(\mathfrak t)$ along the trajectory
before it hits $\Sigma_{x_0}^-$ (equivalently, the unique time when $y(\mathfrak t)=0$).
As $\rho_0\rightarrow (\rho^{\rm s}_{x_0}(\theta_0))^{+}$, the initial condition approaches the
stable boundary $C^{\rm s}_{x_0}$, hence the trajectory spends longer and longer in a neighborhood of $(x,y)=(0,0)$. In particular $x_*\to0^+$ and $T_{x_0}(\rho_0,\theta_0)\to +\infty$.

For the angular variation, \eqref{equ: Ham system 3-new} gives
$$
\frac{d\theta}{d\mathfrak t}
=\frac{\varpi K}{r(\rho\cos\theta)^2}+\frac{1}{\rho}O(x_0^6).
$$
Along the trajectories in McGehee's neighborhood one has $0<r_-\le r(\cdot)\le r_+$
and $\rho(\mathfrak t)\ge \rho_\infty/2$ for $x_0$ sufficiently small, hence
$d\theta/d\mathfrak t\ge c_0>0$ uniformly on $[0,T_{x_0}]$.
Consequently,
$$
\Theta_{x_0}(\rho_0,\theta_0)
=\int_0^{T_{x_0}(\rho_0,\theta_0)}\frac{d\theta}{d\mathfrak t}\,d\mathfrak t
\geq c_0\,T_{x_0}(\rho_0,\theta_0)\to+\infty,
$$
which concludes the proof.
\end{proof}

 Denote the $T_{x_0}$-map
$$
\Psi_{x_0}:A^{\rm s}_{x_0} \to A^{\rm u}_{x_0},
\qquad
(\rho_0,\theta_0)\mapsto (\rho(T_{x_0}),\theta(T_{x_0})),
$$
in polar coordinates $(u,v)=(\rho\cos\theta,\rho\sin\theta)$. This map preserves the area form $\omega=(\varpi/(u+1/\varpi)^2)du\wedge dv=dp_r\wedge dr$ induced by $\tilde \omega_0$. Identify $D_{x_0}^{\rm s,u}$ with its projection $\mathcal D_{x_0}^{\rm s,u}$ to the $(u,v)$-plane. The $\omega$-areas of $\mathcal D_{x_0}^{\rm s,u}$ coincides with the area of the projection of $\mathcal S_{\pm\infty}$ to the $(u,v)$-plane, given by $\mathcal D_{\pm\infty}:=\{(u,v):u^2+v^2\leq \rho^2_\infty\}$.
Let $\mathcal N(u,v)=(u,-v)$. Due to the reversibility, we have
$$
\mathcal N(\mathcal D_{x_0}^{\rm s})=\mathcal D_{x_0}^{\rm u}.
$$

We replace $x_0$ with $z_0=8/(\alpha\sqrt{1+2\alpha}\,x_0^2)$ and denote $\Sigma_{x_0}^\pm,D_{x_0}^{\rm s,u},A_{x_0}^{\rm s,u},C_{x_0}^{\rm s,u}$ by $\Sigma_{z_0}^\pm,D_{z_0}^{\rm s,u},A_{z_0}^{\rm s,u}$, $C_{z_0}^{\rm s,u}$, respectively.  By forward flowing $D_{z_0}^{\rm u}$ and backward flowing $D_{z_0}^{\rm s}$, we obtain disks $D_{0}^{\rm s,u}\subset \Sigma_0^\pm$ and circles $C^{\rm s,u}_{0}=\partial D_{0}^{\rm s,u}$. Let $U^{\rm s,u}_{0}\subset \Sigma^\pm_0\setminus D_{0}^{\rm s,u}$ be the open annuli so that $A_{z_0}^{\rm s}$ is the image of $U_{0}^{\rm s}$ along the forward flow and $U_{0}^{\rm u}$ is the image of $A_{z_0}^{\rm u}$ along the forward flow. In this way, we obtain a map
$$
\Phi_{0}:=\Phi^{\rm u}_{z_0,0}\circ \Phi_{z_0}\circ \Phi^{\rm s}_{0,z_0}:U^{\rm s}_{0}\to U^{\rm u}_{0}
\qquad\forall\ z\in[0,z_0],
$$
where $\Phi_{z_0}:A^{\rm s}_{z_0}\to A^{\rm u}_{z_0}$ is the map $\Psi_{x_0}$, $\Phi^{\rm s}_{0,z_0}:U_{0}^{\rm s}\to A_{z_0}^{\rm s}$ and $\Phi^{\rm u}_{z_0,0}:A_{z_0}^{\rm u}\to U_{0}^{\rm u}$ denote the first hitting maps along the flow. Notice that $\Sigma^{+}_0=\Sigma_0$ and $\Sigma^{-}_0=\Sigma_{1/2}$, which project to $\dot \Upsilon\subset \Upsilon$ in the $(p_r,r)$-plane.

Let $\pi_0^{\pm}:\Sigma_0^{\pm}\to \Upsilon$ be the projection of $\Sigma_0^\pm$ to the $(u,v)$-plane. Identify $C_{0}^{\rm s,u},D_{0}^{\rm s,u},U_{0}^{\rm s,u}$ with their projections $\mathcal C_{0}^{\rm s,u},\mathcal D_{0}^{\rm s,u}, \mathcal U_{0}^{\rm s,u}$ to $\Upsilon$, respectively. The map $\Phi_0$ further induces the following map
$$
g_{0}=\pi^-_{0}\circ \Phi_0\circ (\pi^+_{0})^{-1}:\mathcal U^{\rm s}_{0}\to \mathcal U_{0}^{\rm u},
$$
which preserves the area form $\omega=dp_r \wedge dr$.  The reversibility implies that
$$
g_0\circ \mathcal N\circ g_0=\mathcal N,\qquad
\mathcal N(\mathcal C_{0}^{\rm s})=\mathcal C_{0}^{\rm u},\qquad
\mathcal N(\mathcal D_{0}^{\rm s})=\mathcal D_{0}^{\rm u},\qquad
\mathcal N(\mathcal U_{0}^{\rm s})=\mathcal U_{0}^{\rm u}.
$$

The map $g_0$ becomes the restriction of
$\bar g:=\pi_0^- \circ \psi_{0,1/2}\circ (\pi^+_0)^{-1}:\Upsilon\setminus \mathcal D_0^{\rm s}\to \Upsilon\setminus \mathcal N(\mathcal D_0^{\rm s})$ determined by the first hitting map $\psi_{0,1/2}:\Sigma_0\setminus D_{0}^{\rm s}\to \Sigma_{1/2}\setminus D_{0}^{\rm u}$.
Since $\Phi^{\rm s}_{0,z_0}$ and $\Phi^{\rm u}_{z_0,0}$ are smooth diffeomorphisms,  Proposition \ref{prop: local twist} implies that $g_0$ has an infinite twist near the inner boundary component $\mathcal C_0^{\rm s}$. This means that the image under $g_0$ of small arcs in $\mathcal U_0^{\rm s}$ with an endpoint in $\mathcal C^{\rm s}_0$ spirals infinitely often around $\mathcal C^{\rm u}_0$. Although $g_0$ is not defined on the closed disk $\mathcal D^{\rm s}_{0}$, we may use the notation
$$
g_0(\mathcal D^{\rm s}_0) = \mathcal D^{\rm u}_0=\mathcal N(\mathcal D^{\rm s}_0).
$$ If $\mathcal D^{\rm u}_0 \subset \Upsilon \setminus \mathcal D^{\rm s}_0$, then we may consider the set $g_0^2(\mathcal D_0^{\rm s})$ even if the map $g_0^2$ is not defined on that set. In the same way, we may consider higher iterates $g^k(\mathcal D^{\rm s}_0)$ whenever $g_0^{j}(\mathcal D^{\rm s}_0)\subset \Upsilon \setminus \mathcal D^{\rm s}_0, j=1,\ldots k-1$. All such iterates have the same symplectic area given by the action of the orbits $\zeta_{\pm\infty}$ at infinity.

Now we focus on the proof of Theorem \ref{thm: main5}. Since $g_0$ preserves area, there exists a least $k\geq 1$ such that $g_0^k(\mathcal D^{\rm s}_0) \cap \mathcal D
^{\rm s}_0 \neq \emptyset$ and $g_0^j(\mathcal D^{\rm s}_0) \cap \mathcal D^{\rm s}_0 = \emptyset$ for every $j=1,\ldots, k-1$. In particular, $g_0^k(\mathcal C^{\rm s}_0) \cap \mathcal C
^{\rm s}_0 \neq \emptyset$ and $g_0^j(\mathcal C^{\rm s}_0) \cap \mathcal C^{\rm s}_0 = \emptyset$ for every $j=1,\ldots, k-1$. Such intersections corresponds to parabolic trajectories converging to $\zeta_{\pm\infty}$ for $t\to \pm \infty$. In fact, if $k$ is even, then the parabolic trajectory converges to the same orbit at infinity, and if $k$ is odd, then the parabolic trajectory converges to distinct orbits $\zeta_{+\infty}\subset \{z=+\infty\}$ and $\zeta_{-\infty}\subset \{z=-\infty\}$ at infinity forward and backward in time.

Let us distinguish two cases: (I) $g_0^k(\mathcal D^{\rm s}_0) = \mathcal D^{\rm s}_0$; and (II) $g_0^k(\mathcal D^{\rm s}_0) \neq \mathcal D^{\rm s}_0$. If (I) holds, it is immediate that there exists infinitely many parabolic trajectories. Also, all iterates of $g_0$ are well-defined on the holed-disk $\Upsilon \setminus \bigcup_j g_0^j(\mathcal D^{\rm s}_0)$, and the map $g_0^k$ has an infinite twist around each hole $g_0^j(\mathcal D^{\rm s}_0)$. A direct application of Franks' generalization of the Poincar\'e-Birkhoff theorem \cite[Theorem 2.1]{Franks92} implies that $g_0$ has infinitely many periodic points, which correspond to periodic orbits in $\mathfrak M$.

Now assume that (II) holds. Since $\mathcal C^{\rm s}_0 = \partial \mathcal D^{\rm s}_0$ is a real-analytic curve and $g_0$ is a real-analytic map, a small open arc $\gamma \subset g_0^k(\mathcal C^{\rm s}_0)$ in $\mathcal U^{\rm s}_0$ with an endpoint in $\mathcal C^{\rm s}_0$ can be iterated under $g_0^k$. By the infinite twist described in Proposition \ref{prop: local twist},  $g_0^k(\gamma)$ spirals and converges to $g_0^k(\mathcal D^{\rm s}_0)$ and thus intersects $\mathcal C^{\rm s}_0$ infinitely often. These intersections produce infinitely many parabolic trajectories associated with the map $g_0^{k}$. In particular, these parabolic trajectories converge to the same periodic orbit at infinity if $k$ is even, and converge to distinct orbits $\zeta_{\pm\infty}$ at infinity forward and backward in time if $k$ is odd.

Let us now seek for brake orbits and parabolic-brake orbits. Recall that $\mathcal H\,=V^{-1}((-\infty,-1])\subset \mathbb R^2$ is the Hill region of $\mathfrak M$. Its boundary $\partial \mathcal H$ is identified with the points in $\mathfrak M \cap \{p_r=p_z=0\}$. We split $\partial \mathcal H\cap \{z\neq 0\}$ into two parts $\mathcal B_+\cup \mathcal B_-$, where $\mathcal B_+\subset \{z>0\}$ and $\mathcal B_-\subset \{z<0 \}$. Since $\mathcal H$ is unbounded in the $z$, both $\mathcal B_+$ and $\mathcal B_-$ consist of two connected components, $\mathcal B_{\pm,1}$ and $\mathcal B_{\pm,2}$, where $\mathcal B_{\pm,1}$ is the graph of a function $\pm \hat z(r),r\in(r_-,r^*_{-})$, and $\mathcal B_{\pm,2}$ is the graph of a function $\pm \hat z(r),r\in(r^*_{+},r_+)$. Here, $r_\pm=(1\pm \fe)/(2\beta)$ and $r^*_\pm=(1\pm\sqrt{1-2\varpi^2})/2$. Notice that $\hat z(r)>0$ is increasing in $r\in(r_-,r^*_-)$ from $0$ to $+\infty$ and decreasing in $r\in(r^*_+,r_+)$ from $+\infty$ to $0$. The points in $\partial \mathcal H\cap \{z\neq 0\}$ correspond to $\mathfrak M \cap \{y=v=0\}$ in McGehee's coordinates $(x,y,u,v)$ near $z=\pm \infty$.

Forward flowing $\mathcal B_+$ and $\mathcal B_-$, we obtain the first hitting curves $C_-\subset \Sigma_0^-$ and $C_+\subset \Sigma_0^+$, respectively. The $z$-symmetry and the reversibility of the flow imply that both $C_+$ and $C_-$ are embedded curves in $\Sigma_0^\pm\setminus D_{0,\pm}^{\rm u}$ with the same projection $\mathcal C\subset \Upsilon\setminus \mathcal D_0^{\rm u}$ and satisfying $g_0\circ \mathcal N (\mathcal C)=\mathcal C$. Here, $D_{0,+}^{\rm u}=D_0^{\rm u}$ is given as before and $D_{0,-}^{\rm u}$ is the analog disk that associates to the unstable manifold of $\zeta_{-\infty}$.
Moreover, $\mathcal C = \mathcal C_1\cup \mathcal C_2$ has two components, where $\mathcal C_1$ corresponds to the left component $\mathcal B_{\pm,1}$ and $\mathcal C_2$ corresponds to the right component $\mathcal B_{\pm,2}$. As discussed in \cite[Theorem 2.3-(ii)]{HLOSY2023}, since the rotation number of the Euler orbit satisfies $\rho_e>2$, the curves $\mathcal C_1$ and $\mathcal C_2$ have points in $\{p_r>0\}$ and in $\{p_r<0\}$, respectively, near one of the endpoints, and spiral and converge to the circle $\mathcal C^{\rm u}_0$ at the other end. This last property follows from the proof of Proposition \ref{prop: local twist}. In particular, at least one such curve intersects $\dot \Upsilon \cap \{p_r=0\}$ and thus at least one $z$-symmetric brake orbit exists.

%

If (I) holds, then since $g_0^k$ is reversible with respect to $\mathcal N$, a theorem of J. Kang \cite{Kang2018} implies that $g_0^k$ admits infinitely many $\mathcal N$-symmetric periodic orbits in $\Upsilon$. This follows from the existence of a $\mathcal N$-symmetric fixed point in $\mathcal C\cap \{p_r=0\}$ and another periodic point due to Franks. Each $\mathcal N$-symmetric periodic orbit corresponds to a $z$-symmetric orbit in $\mathfrak M$ that hits $\{z=p_r=0\}$, see Proposition 7.2 in \cite{HLOSY2023}. This proves (ii).

If (II) holds, then recall from the previous arguments that both curves $\mathcal C_1$ and $\mathcal C_2$ spiral and converge to $\mathcal C_0^{\rm u}$ at one of the endpoints. This implies that $g_0^{j}(\mathcal C_i),j=0,\ldots,k-1$, spiral and converge to $g_0^j(\mathcal C_0^{\rm u})$. Hence, $g_0^{k-1}(\mathcal C_i)$ intersects $\mathcal N(\mathcal C_l)$ at infinitely many points. Each such an intersection point corresponds to a brake orbit in $\mathfrak M$, possibly connecting $\mathcal B_{+,i}$ and $\mathcal B_{+,l}$ or $\mathcal B_{+,i}$ and $\mathcal B_{-,l}$. In particular, if $i\neq l$, then the brake orbit is not $z$-symmetric. This finishes the proof of (i) and (iii). The proof of Theorem \ref{thm: main5} is complete.

\appendix
\section{Maslov-type index}\label{sec: maslov-type indices and rotation numbers}
In this section, we briefly introduce the Maslov-type index, mean index and the rotation number of the periodic orbits, see \cite{HLOSY2023,Long2002} for more details. Let $\mathrm{Sp}(2n)$ denote the set of $2n\times2n$ real symplectic matrix. For every $\omega\in\mathbf{U}$, the $\omega$-singular set is defined as
$\mathrm{Sp}(2n)_\omega^0:=\left\{M\in \mathrm{Sp}(2n): \det(M-\omega I_{2n})=0 \right\}\subset \mathrm{Sp}(2n)$, on which there exists a nowhere vanishing vector field $V(M):=\frac{d}{dt}Me^{tJ_{2n}}|_{t=0},\forall M\in \mathrm{Sp}(2n)_\omega^0$, which is everywhere transverse to $\mathrm{Sp}(2n)_\omega^0$ and determines a positive co-orientation of $\mathrm{Sp}(2n)_\omega^0$. Denote the $\omega$-regular set as $\mathrm{Sp}(2n)^*_{\omega}:=\mathrm{Sp}(2n)\setminus \mathrm{Sp}(2n)_\omega^0$. 
\begin{defi}\label{def:Maslov-type index}
Let $\gamma\in \mathcal{P}_{\tau}(2n):=\left\{\gamma \in C([0,\tau], \mathrm{Sp}(2n)): \gamma(0)=I_{2 n}\right\}$. For every $\omega\in \mathbf{U}$, we define the $\omega$-index and $\omega$-nullity of $\gamma$ as
$$
i_\omega(\gamma):=\begin{cases}(e^{-\epsilon J_{2n}}\gamma)\cdot \mathrm{Sp}(2n)_\omega^0-n, & \text { if } \omega=1 \\
(e^{-\epsilon J_{2n}}\gamma)\cdot \mathrm{Sp}(2n)_\omega^0, & \text { if } \omega\neq1,\end{cases} \quad \mathrm{and}\quad
\nu_\omega(\gamma):=\dim_{\mathbb{C}}\ker_{\mathbb{C}}(\gamma(\tau)-\omega I_{2n}),
$$
for every $\epsilon>0$ sufficiently small, where $(e^{-\epsilon J_{2n}}\gamma)\cdot \mathrm{Sp}(2n)_\omega^0$ denotes the algebraic intersection number between $e^{-\epsilon J_{2n}}\gamma$ and $\mathrm{Sp}(2n)_\omega^0$. Define the mean index and the rotation number of $\gamma$ as
$$
\hat{i}(\gamma):=\lim_{m\rightarrow+\infty}\frac{i_{1}(\gamma^m)}{m},\quad \mathrm{and}\quad \rho(\gamma):=\frac{\hat i(\gamma)}{2},
$$
where $\gamma^m:[0,k\tau]\rightarrow \mathrm{Sp}(2n)$ denotes the $m$-th iteration of $\gamma$ for every $m\in  \mathbb Z_+$, which is defined as
$$
\gamma^m(t):=\gamma(t-j\tau)\gamma(\tau)^j,\quad\forall t\in[j\tau,(j+1)\tau],\quad j=0,\cdots,m-1.
$$
\end{defi}
Note that if $\omega=1$, the $1$-index coincides with the Conley-Zehnder index. In $\mathrm{Sp}(2)$, we can define the $1$-index as follows. Let $\eta_v(t)\in C([0,\tau],\mathbb R)$ be the argument of $\gamma(t)v$ for any $v\in \mathbf{U}$. Let $\Delta(v):=(\eta_v(\tau) - \eta_v(0))/(2\pi)$ be the net variation of $\eta$ on $[0,\tau]$. Denote by $I_\gamma:= \{\Delta(v), v \in \mathbf{U}\} \subset \mathbb{R}$, which is a closed interval of length $<1/2$. For every $\epsilon>0$ sufficiently small, we define $i_1(\gamma):=2k+1$ if $I_\gamma-\epsilon\in (k,k+1)$ and define $i_1(\gamma): = 2k$ if $k\in I_\gamma-\epsilon$. The rotation number of $\gamma$ becomes
$$
\rho(\gamma):=\frac{\hat i(\gamma)}{2}= \lim_{k\to \infty}\frac{\eta_v(k\tau)}{2\pi k}.
$$

In \cite{Long2002}, Long provided an explicit formula for the mean index. For simplicity, we only consider the paths in $\mathrm{Sp}(2)$. Firstly, every $M\in \mathrm{Sp}(2)$ is symplectically similar ($\approx$) to one of the following matrices:
\begin{equation}\label{definition of R(theta)}
D(\lambda)=\left(\begin{array}{cc}
\lambda & 0 \\
0 & 1/\lambda
\end{array}\right),\ \
N(\pm 1, a)=\left(\begin{array}{cc}
\pm 1 & a \\
0 & \pm 1
\end{array}\right), \ \
R(\vartheta)=\left(\begin{array}{cc}
\cos \vartheta & -\sin \vartheta \\
\sin \vartheta & \cos \vartheta
\end{array}\right),\nonumber
\end{equation}
where $\lambda\in \mathbb{R}\setminus\{0\}$, $a\in\{\pm1,0\}$ and $\vartheta\in(0,\pi)\cup(\pi,2\pi)$. By Corollary $8.3.2$ in \cite{Long2002}, we have
\begin{equation} \label{mean index 2}
\hat{i}(\gamma)=\left\{\begin{array}{ll}
i_1(\gamma)-1+\vartheta/\pi, & \text {if}\ \ \gamma(\tau)\approx R(\vartheta), \vartheta \in(0, \pi) \cup(\pi, 2 \pi),  \\
i_1(\gamma)+1, & \text{if}\ \ \gamma(\tau)\approx N_1(1,a), a=1,0, \\ i_1(\gamma), & \text{otherwise.}
\end{array}\right.
\end{equation}

The $\omega$-index also relates to the Morse index of a certain self-adjoint operator. Consider a second order system $\ddot{x}=\mathcal D(t)x$, $t\in[0,\tau]$. Let $\mathcal{A}:=-\frac{d^2}{dt^2}+\mathcal D$ be a self-adjoint operator on $L^2([0,\tau],\mathbb{C}^n)$ with domain
$$
D(\omega,\tau):=\left\{W^{2,2}([0,\tau], \mathbb{C}^n): y(\tau)=\omega y(0), \dot{y}(\tau)=\omega \dot{y}(0)\right\}.
$$
Let $m^-_\omega(\mathcal{A})$ denote the Morse index of $\mathcal A$, which is the total multiplicity of the negative eigenvalues of $\mathcal A$. Let $\nu_{\omega}(A):=\dim\ker(A)$ denote the nullity of $\mathcal{A}$. From Theorem 7.3.4 in \cite{Long2002}, we have the following relations
\begin{equation*}
m^-_{\omega}(\mathcal{A})=i_{\omega}(\gamma), \quad \nu_{\omega}(\mathcal{A})=\nu_{\omega}(\gamma),\quad \forall \omega\in \mathbf U,
\end{equation*}
where $\gamma\in \mathcal P_\tau(2n)$ is the fundamental solution of the linear system
$\dot{\gamma}=J_{2n} \mathrm{diag}(I_n,-D(t))\gamma$. For more general boundary conditions, we refer to \cite{HWY} for a similar result.

Consider a contact-type sphere-like energy surface $M=H^{-1}(h)\subset \mathbb R^4$ of $H=H(y_1,y_2,x_1,x_2)$. Let $\mathcal{P}$ denote the set of all periodic orbits of the Hamiltonian equation $\dot x=X_H(x)$. Let $\zeta\in\mathcal{P}$ be a $T$-periodic orbit with the linearized flow $d\varphi_t:\mathbb{R}^4\rightarrow \mathbb{R}^4,\forall t\in \mathbb R$, which preserves the standard symplectic form $\hat \omega_0=dy_1\wedge dx_1+dy_2\wedge dx_2$. Consider a quaternion frame $\{X_0,X_1,X_2,X_3\}$ in $\mathbb{R}^4$, where
\begin{equation}\label{equ: global trivalization}
\begin{aligned}
X_0&=\frac{\nabla H}{|\nabla H|},\quad X_1=\frac{1}{|\nabla H|}\big(\partial_{x_2}H,-\partial_{x_1}H,\partial_{y_2}H,-\partial_{y_1}H\big),\\
X_2&=\frac{1}{|\nabla H|}\big(-\partial_{y_2}H,\partial_{y_1}H,\partial_{x_2}H,-\partial_{x_1}H\big),\quad X_3=\frac{X_H}{|\nabla H|}.
\end{aligned}
\end{equation}
Under this global trivialization, the linearized flow $d\varphi_t$ becomes a symplectic path $\Phi_\zeta:[0,T]\rightarrow \mathrm{Sp}(4)$ with $\Phi_\zeta(0)=I_{2n}$. Moreover, since $d\varphi_t$ has an invariant subspace $\{X_1,X_2\}$, we obtain a $2\times 2$-symplectic matrix path, named $\Phi_0:[0,T]\rightarrow \mathrm{Sp}(2)$. In the reordered basis $\{X_1,X_2,X_0,X_3\}$, the path $\Phi_\zeta(t)$ can be rephrased as
$$
\Phi_\zeta(t)=\begin{pmatrix}\Phi_0(t) & 0\\
0 & N(1,d(t))^T
\end{pmatrix},\quad
N(1,c):=\begin{pmatrix}
1 & c \\ 0 & 1
\end{pmatrix},\quad d\in C^\infty([0,T],\mathbb{R}).
$$
The rotation number of $\zeta$ can be computed as
$$
\rho(\zeta):=\rho(\Phi_\zeta)=\frac{\hat i(\Phi_\zeta)}{2}=\frac{\hat i(\Phi_0)}{2}=\lim_{k\rightarrow +\infty}\frac{\arg(\Phi_0(kT)v)}{2\pi k},\quad \forall v\in\mathbb{C}\setminus\{0\}.
$$

\section{Relative winding number and linking number}\label{sec: wind and link}
For every periodic points $x_1\neq x_2\in \dot \Upsilon$ in Theorem \ref{thm: main4}, an interesting question can be asked: what is the relationship between the mean relative winding number $w_\infty(x_1,x_2),x_1\in \mathcal O(\zeta_1),x_2\in\mathcal O(\zeta_2)$ and the linking number $\mathrm{lk}(\zeta_1,\zeta_2)$ (Remark \ref{rem: pairing and w_infty})? The following theorem provides a general relation between these two quantities.

\begin{thm}\label{thm: Winfty and lk}
Let $\zeta_1,\zeta_2 \subset \mathfrak{M}\setminus \zeta_e$ be two geometrically distinct simple periodic orbits. Let $x_j\in \zeta_j\cap \Sigma_0$ be an intersection point, which is a periodic point of the first return map $\psi$ with prime period $k_j:=\mathrm{lk}(\zeta_j,\zeta_e)$ for $j=1,2$. Then
\begin{equation}\label{equ: average mean winding number}
\begin{aligned}
\sum_{l_1=0}^{k_1-1}\sum_{l_2=0}^{k_2-1}w_\infty(\psi^{l_1}(x_1),\psi^{l_2}(x_2))=\mathrm{lk}(\zeta_1,\zeta_2).
\end{aligned}
\end{equation}
Moreover, if $\mathrm{gcd}(k_1,k_2)=1$, then
$w_\infty(\psi^{l_1}(x_1),\psi^{l_2}(x_2))=\mathrm{lk}(\zeta_1,\zeta_2)/(k_1k_2)$ is independent of $l_1,l_2$.
\end{thm}
\begin{proof}
It suffices to show that $\sum_{l_2=0}^{k_2-1}w_\infty(x_1,\psi^{l_2}(x_2))=\mathrm{lk}(\zeta_1,\zeta_2)/k_1$ is independent of $x_1\in \mathcal O(\zeta_1)$. We establish this identity as follows. Let $\mathfrak M_*:=\mathfrak M\setminus \zeta_e \simeq \dot \Upsilon\times\mathbb R/\mathbb Z$ and let $\mathfrak M'_*:=\dot \Upsilon\times(\mathbb R/(k_1k_2\mathbb Z))$ be the $(k_1k_2)$-covering space of $\mathfrak M_*$. Denote by $\zeta_{j,k}\in\mathfrak M_*,k=0,\cdots,k_j-1,j=1,2,$ the segment of trajectory from $\psi^k(x_j)$ to $\psi^{k+1}(x_j)$. Take a point $p\in \zeta_e$ and choose a path $\beta\in \Upsilon\setminus \mathcal O(\zeta_2)$ connecting $x_1$ to $p$. Choose $k_1$ paths $\beta_k\subset \dot \Upsilon\setminus (\dot \beta\cup \mathcal O(\zeta_2)),k=0,\ldots k_1-1$, with disjoint interiors and connecting $\psi^k(x_1)$ to $x_1$, where $\beta_0\equiv x_1$ and $\dot \beta = \beta \setminus \{x_1,p\}$. We thus obtain loops $\bar \zeta_{1,k}:=\bar \beta_{k}\ast\zeta_{1,k} \ast\beta_{k+1},k=0,\ldots k_1-1$. Notice that $\zeta_1$ is the sum of $\bar\zeta_{1,0},\ldots,\bar\zeta_{1,k_1-1}$ and each $\bar\zeta_{1,k}$ bounds a disk $D_k\subset \mathfrak M$. Hence,
$$
\mathrm{lk}(\zeta_1,\zeta_2)=\sum_{k=0}^{k_1-1}\mathrm{lk}(\bar \zeta_{1,k}, \zeta_{2})=\sum_{k=0}^{k_1-1}\zeta_2\cdot D_{k}=\sum_{k'=0}^{k_2-1}\sum_{k=0}^{k_1-1}\zeta_{2,k'}\cdot D_k.
$$

Fix a lift $\tilde \zeta_j\in \mathfrak M_*'$ of $\zeta_j$ for both $j=1,2$. Denote by $\tilde\zeta_{j,\bar k}\subset\Upsilon\times[\bar k,\bar k+1]$ the lift of $\zeta_{j,k}$ for every $\bar k=k+lk_j\in [0,k_1k_2-1]$ and define a loop $\hat \zeta_{1,\bar k}:=\bar \beta\ast \bar \beta_{k}\ast \tilde \zeta_{1,\bar k} \ast\beta_{k+1} \ast\beta\ast (\{p\}\times [\bar k,\bar k+1])$ for every $k=0,\cdots,k_1-1$. Then $\hat \zeta_{j,\bar k}$ bounds a disk $\hat D_{\bar k}$, whose projection is $D_k$. We can naturally extend the disks for every $\bar k\in\N$ such that $\hat D_{\bar k}=\hat D_{\bar k+k_1}$. By definition, we can explain the mean relative winding number $W_\infty$ as an average intersection number
$$\begin{aligned}
w_\infty(x_1,\psi^{l_2}(x_2))&=\frac{1}{k_1k_2}\sum_{\bar k=0}^{k_1k_2-1}\tilde \zeta_{2,\bar k+l_2}\cdot \hat D_{\bar k}=\frac{1}{k_1k_2}\sum_{\bar k=0}^{k_1k_2-1}\zeta_{2,k'}\cdot D_{k},
\end{aligned}$$
where $(k,k')=(\bar k,\bar k+l_2)\ \mathrm{mod}\ (k_1,k_2)$. Let $[n]_j:=n\ \mathrm{mod}\ k_j,j=1,2$. We compute
$$\begin{aligned}
&\quad\ \sum_{l_2=0}^{k_2-1}w_\infty(x_1,\psi^{l_2}(x_2))=\frac{1}{k_1k_2}\sum_{k=0}^{k_1-1}(\sum_{l_2=0}^{k_2-1}\sum_{j=0}^{k_2-1}\zeta_{2,[jk_1+k+l_2]_2})\cdot D_{k}\\
&=\frac{1}{k_1k_2}\sum_{k=0}^{k_1-1}(\sum_{l_2=0}^{k_2-1}\sum_{j=0}^{k_2-1}\zeta_{2,[jk_1+l_2]_2})\cdot D_{k}=\frac{1}{k_1}\sum_{k=0}^{k_1-1}\zeta_{2}\cdot D_{k}=\frac{\mathrm{lk}(\zeta_1,\zeta_2)}{k_1}.
\end{aligned}$$
Here, the third equality follows from the fact that $\sum_{l_2=0}^{k_2-1}\zeta_{2,[jk_1+l_2]_2}$ is independent of $j$. Similarly, we obtain
$$\sum_{l_1=0}^{k_1-1}w_\infty(\psi^{l_1}(x_1),x_2)=\frac{\mathrm{lk}(\zeta_1,\zeta_2)}{k_2}.$$
Hence, \eqref{equ: average mean winding number} holds. In particular, if $\mathrm{gcd}(k_1,k_2)=1$, we further show that
$$\begin{aligned}
w_\infty(x_1,x_2)&=\frac{1}{k_1k_2}\sum_{\bar k=0}^{k_1k_2-1}\zeta_{2,k'}\cdot D_k=\frac{1}{k_1k_2}\sum_{k=0}^{k_1-1}\sum_{k'=0}^{k_2-1}\zeta_{2,k'}\cdot D_k\\
&=\frac{1}{k_1k_2}\sum_{k=0}^{k_1-1}\mathrm{lk}(\bar\zeta_{1,k},\zeta_{2})=\frac{\mathrm{lk}(\zeta_1,\zeta_2)}{k_1k_2},\quad (k,k')=\bar k\ \mathrm{mod}\ (k_1,k_2).
\end{aligned}$$
Here, we have used the identity $\sum_{\bar k=0}^{k_1k_2-1}a_{k,k'}=\sum_{k=0}^{k_1-1}\sum_{k'=0}^{k_2-1}a_{k,k'},(k,k')=\bar k\ \mathrm{mod}\ (k_1,k_2)$ for every $k_1,k_2$ satisfying $\mathrm{gcd}(k_1,k_2)=1$. Therefore, $w_\infty(x_1,x_2)$ is independent of $x_j\in\mathcal O(\zeta_j), j=1,2$. The proof of Theorem \ref{thm: Winfty and lk} is complete.
\end{proof}

\section{Proof of Lemma \ref{lem: trace and determinant are positive}}\label{sec: supplementary proofs}
\begin{proof}[Proof of Lemma \ref{lem: trace and determinant are positive}-(i)]
Fix an arbitrary $\nu\in(0,1)$. The proof includes two cases.

\textbf{Case 1:} If $\mathrm{tr}_2\leq 0$, then we only need to prove $\mathrm{tr}_0+\mathrm{tr}_1e<0$. Since $\mathrm{tr}_0<0$, then it is sufficient to show $\mathrm{tr}_0+\mathrm{tr}_1<0$ for every $\beta\in[0,1]$. Since $\mathrm{tr}_0+\mathrm{tr}_1<4C_3C_4(4C_2^2C_3+C_3-2C_2)$, we compute
$$
F_0(\beta,\nu):=(1+\nu)^2(2+\nu)(3+\nu)^2(4+\nu)(4C_2^2C_3+C_3-2C_2)=f_0+f_1\beta+f_2\beta^2+f_3\beta^3,
$$
where
$$\begin{aligned}
f_0&=-3(3+4\nu+\nu^2)^2(8+6\nu+\nu^2),\\
f_1&=7(189+420\nu+320\nu^2+100\nu^3+11\nu^4),\\
f_2&=-196(14+14\nu+3\nu^2),\quad f_3=1372.
\end{aligned}$$
In particular, $F_0(0,\nu)=f_0<0$, $F_0(1,\nu)=-265-542\nu+665\nu^2+40\nu^3-157\nu^4-42\nu^5-3\nu^6<0$ for every $\nu\in(0,1)$. Since $\partial_\beta F_0(\beta,\nu)$ is a parabola of $\beta$, we find a local maximal point $\beta^+=(1/42)(28(1+\nu)+6\nu^2-\sqrt{217+308\nu+160\nu^2+36\nu^3+3\nu^4})$ that solves $\partial_\beta F_0(\beta^+,\nu)=0$. Then we compute
$$\begin{aligned}
F_0(\beta^+,\nu)&=\frac{1}{27}\big((217+308 \nu+160\nu^2+36\nu^3+3\nu^4)^{3/2}\\
&\qquad\quad -2(5+2\nu)(397+746\nu+529\nu^2+162\nu^3+18\nu^4)\big)\leq 0,
\end{aligned}$$
which is because
$$\begin{aligned}
&\quad\ (217+308 \nu+160\nu^2+36\nu^3+3\nu^4)^{3}-4(5+2\nu)^2(397+746\nu+529\nu^2+162\nu^3+18\nu^4)^2\\
&\leq 27(1+\nu)^2(3+\nu)^2\cdot(-22809)<0,\quad \forall \nu\in(0,1).
\end{aligned}$$
Finally, since $F_0(\cdot,\nu)$ can only reach to its maximum at $\beta=0,\beta^+,1$, therefore, the previous computation implies that $F_0(\beta,\nu)<0$, which implies that $\mathrm{tr}_1+\mathrm{tr}_2<0$ for every $\beta\in[0,1]$.

\textbf{Case 2:} If $\mathrm{tr}_2>0$, then $C_1+C_2>0$ necessarily holds due to $C_0,C_3,C_4<0$. We only need to prove $\mathrm{tr}_0+\mathrm{tr}_1e+(C_1+C_2)e^2<0$. Due to the convexity of this parabola, it is sufficient to prove
$\mathrm{tr}_0+\mathrm{tr}_1+(C_1+C_2)<0$ for every $\beta\in[0,1]$.

We first compute
$$\begin{aligned}
&\quad\ \mathrm{tr}_0+\mathrm{tr}_1+C_1+C_2\\
&=4 (C_2C_3-1)^2C_4+4C_2^2C_3+4C_0C_1^2+4C_3^2C_4
-C_4+C_1+C_2\\
&\leq 4C_2^2C_3+4C_0C_1^2+4C_3^2C_4-C_4+C_1+C_2\\
&=:4C_2^2C_3+4(C_0+1)C_1^2+F_2+F_3,
\end{aligned}$$
where $F_2=-4C_1^2+C_1+C_2-2(C_3+1)C_4$ and $
F_3=(4C_3^2+2C_3+1)C_4$.
We compute
$$\tilde F_2=\nu^2(1+\nu)(2+\nu)^2(3+\nu)(4+\nu)(5+\nu)F_2=f_{20}+f_{21}\beta+f_{22}\beta^2,$$
where
$$\begin{aligned}
f_{20}&=-6\nu^2(2+\nu)^2(60+107\nu+59\nu^2+13\nu^3+\nu^4),\\
f_{21}&=7\nu(1080+2606\nu+2263\nu^2+907\nu^3+170\nu^4+12\nu^5),\\
f_{22}&=-98(120+214\nu+120\nu^2+29\nu^3+3\nu^4).
\end{aligned}$$
Then $\tilde F_2$ admits a maximal point $\beta^+_2=-f_{21}/(2f_{22})$. Then we compute
$$
\tilde F_2(\beta^+_2,\nu)=\frac{\nu^2(2+\nu)^2g(\nu)}{8(120+214\nu+120\nu^2+29\nu^3+3\nu^4)}<0,
$$
where $g(\nu)=-54000-117000\nu-53255\nu^2+47502\nu^3+
56101\nu^4+21684\nu^5+3724\nu^6+240\nu^7<0$ on $(0,1)$. Therefore, $F_2$ is negative. Moreover, since
$$\tilde F_3=(2+\nu)^2(4+\nu)^2\frac{F_3}{C_4}
=3(8+6\nu+\nu^2)^2-42(8+6\nu+\nu^2)\beta+196\beta^2,$$
and $\tilde F_3(42(8+6\nu+\nu^2)/392,\nu)=3(8 + 6 v + v^2)^2/4>0$, we see that $F_3$ is also negative. Finally, due to $C_0+1<0$, we conclude that
$\mathrm{tr}_0+\mathrm{tr}_1+C_1+C_2<0$ for every $\beta\in[0,1]$.
\end{proof}

Before the proof of Lemma \ref{lem: trace and determinant are positive}-(ii), we need the following lemma.
\begin{lem}\label{lem: order Sj}
Define the curves $\{S_i,i=1,\cdots,5\}$ in $[0,1]\times [0,1]$ as follows.
$$\begin{aligned}
S_1:&=\{(\beta_1(\nu),\nu)\subset[0,1]\times[0,1]\}=\{7\beta=2\nu+\nu^2\}=\{C_1=0\}\subset\{\mathrm{det}_2=0\},\\
S_2:&=\{(\beta_2(\nu),\nu)\subset[0,1]\times[0,1]\}=\{7\beta=4\nu+\nu^2\}=\{C_1+C_3=0\}=\{\mathrm{det}_4=0\},\\
S_3:&=\{(\beta_3(\nu),\nu)\subset[0,1]\times[0,1]\}\\
&=\{C_0C_1C_2-2C_0C_1C_4-2C_0C_3C_4+C_2C_3C_4=0\}\subset\{\mathrm{det}_2=0\}\\
S_4:&=\{(\beta_4(\nu),\nu)\subset[0,1]\times[0,1]\}=\{7\beta=3+4\nu+\nu^2\}=\{C_2=0\}\subset\{\mathrm{det}_2=0\},\\
S_5:&=\{(\beta_5(\nu),\nu)\subset[0,1]\times[0,1]\}=\{\mathrm{det}_3=0\}=\big\{(C_0C_1+C_3(C_2-C_4))^2\\
&\quad\ \ =C_2C_3(C_1+C_3)(C_2-2C_4)+(C_0+C_2+C_4) (C_0C_1^2+C_3^2C_4)\big\}.
\end{aligned}$$
Then for every $(\beta,\nu)\in [0,1]\times(0,1)$, we have $\beta_{i}(\nu)<\beta_j(\nu)$ for every $1\leq i<j\leq 5$.
\end{lem}
\begin{proof}
It is sufficient to consider the order of curves $S_3$ and $S_5$.

Let $\mathrm{det}_{21}=(1/2)(1-\nu)\nu(1+\nu)^2(3+\nu)^2(4+\nu)(5+\nu)(C_0C_1C_2-2C_0C_1C_4-2C_0C_3C_4+C_2C_3C_4)$. We compute
$$\begin{aligned}
\mathrm{det}_{21}&=(1-\nu)\nu(1+\nu)^2(3+\nu)^2(4+\nu)(5+\nu)+21 (1+\nu)(3+\nu)(-5+32\nu+24\nu^2+8\nu^3+\nu^4)\beta \\
&\quad\ -49(54+92\nu+71\nu^2+24\nu^3+3\nu^4)\beta^2+343(-9+4\nu+\nu^2)\beta^3.
\end{aligned}$$
Since the discriminant of $\partial_\beta\mathrm{det}_{21}(\beta,\nu)$ is
$9604(1701 + 16632\nu + 29326 \nu^2 + 20264\nu^3 + 7333\nu^4 + 1440\nu^5 + 120\nu^6)>0$, we see that $\mathrm{det}_{21}$ admits two different critical points and one of them is negative. Therefore, $\mathrm{det}_{21}(0,\nu)>0$ implies that $\mathrm{det}_{21}(\cdot,\nu)$ is either increasing and then decreasing on $[0,1]$ or monotonically decreasing on $[0,1]$. In either case, $\mathrm{det}_{21}(\beta,\nu)=0$ admits only one branch in $[0,+\infty)\times [0,1]$, that is, $(\beta_3(\nu),\nu)$. Therefore, for every $\nu\in(0,1)$, we have $\mathrm{det}_{21}>0,\forall\beta\in [0,\beta_3(\nu))$ and $\mathrm{det}_{21}<0,\forall \beta\in(\beta_3(\nu),+\infty)$. Since
$$
\mathrm{det}_{21}(\beta_2(\nu),\nu)=24\nu^2(4+\nu)^2>0\quad \text{and}\quad \mathrm{det}_{21}(\beta_4(\nu),\nu)=-96(3+4\nu+\nu^2)^2<0,
$$
we have $\beta_2(\nu)<\beta_3(\nu)<\beta_4(\nu)$ for all $\nu\in(0,1)$.

To see $S_5$, we can rewrite $\mathrm{det}_3$ as
\begin{equation}\label{equ: det3}
\begin{aligned}
\mathrm{det}_3&=4(C_2C_3(C_2-C_4)(2C_1-C_3)-C_2^2C_3(C_1-C_3)+C_0C_2(C_3^2+2C_1^2)\\
&\quad\ +C_0(C_4-C_2)(C_1+C_3)^2)\\
&=4(C_1C_3(C_2-2C_0)(C_2-C_4)+C_2C_3(C_2-C_4)(C_1-C_3)-C_2^2C_3(C_1-C_3)\\
&\quad\ +C_0C_2(C_3^2+2C_1^2)+C_0(C_4-C_2)(C_1^2+C_3^2)).
\end{aligned}
\end{equation}
For every $\beta\in[0,1]$, one can check that $C_0<C_4<C_3<0,C_2-C_4>0, C_1-C_3>0$ and
$$
2C_1-C_3=\frac{7(8+\nu)\beta-\nu(2+\nu)(4+\nu)}{\nu(2+\nu)(4+\nu)},\quad
C_2-2C_0=\frac{7(7+\nu)\beta+(1-\nu^2)(3+\nu)}{(1-\nu^2)(3+\nu)}>0.
$$
In particular, $2C_1-C_3\geq0$ for every $\beta\geq \beta_*$, where $\beta_*:=\nu(2+\nu)(4+\nu)/(7(8+\nu))<\beta_1(\nu)$. Using the first expression of $\mathrm{det}_3$ in \eqref{equ: det3}, we have $\mathrm{det}_3>0$ due to the fact that $C_2\leq 0$ for every $\beta_1\leq\beta\leq\beta_4$. Using the second expression in \eqref{equ: det3}, we also have $\mathrm{det}_3>0$, since $C_1,C_2<0$ for every $0<\beta<\beta_1$. In conclusion, we have $\mathrm{det}_3>0$ for every $\beta\in[0,\beta_4(\nu)]$ and $\nu\in(0,1)$.

Let $\mathrm{det}_{31}:=(1/12)(1-\nu)\nu^2(1+\nu)^2(2+\nu)^2(3+\nu)^2(4+\nu)^2(5+\nu)\mathrm{det}_{3}$. We compute
$$
\partial_{\beta}\mathrm{det}_{31}=g_0+g_1\beta+g_2\beta^2+g_3\beta^3,
$$
where
$$\begin{aligned}
g_0&=14\nu(1+\nu)(2+\nu)^2(3+\nu)(4+\nu)(-15+20\nu+37\nu^2+16\nu^3+2\nu^4),\\
g_1&=-98(-480+780\nu+4979\nu^2+7512\nu^3+5675 \nu^4+2496\nu^5+656\nu^6+96\nu^7+6\nu^8),\\
g_2&=2058(1+\nu)(3+\nu)(64+108\nu+59\nu^2+16\nu^3+2\nu^4),\\
g_3&=-9604(8+3\nu+\nu^2)(12+5\nu+\nu^2).
\end{aligned}$$
We also compute
$$\begin{aligned}
\partial_{\beta}\mathrm{det}_{31}(\beta_*,\nu)&=\frac{-84(1-\nu)\nu(2+\nu)(4+\nu)G(\nu)}{(8 + \nu)^3}<0,\\
\partial_{\beta}\mathrm{det}_{31}(\beta_4,\nu)&=168(1+\nu)(3+\nu)(40+44\nu+11\nu^2)>0,
\end{aligned}$$
where
$G(\nu)=2560+10176\nu+15160\nu^2+11317\nu^3+4461\nu^4+880\nu^5+68\nu^6$. Since $g_3<0$, $\mathrm{det}_{31}(\cdot,\nu)$ only possesses at most one critical point on $\beta\in(\beta_4,1]$. Therefore, $\mathrm{det}_{31}(\beta,\nu)$ is either increasing on $\beta\in(\beta_4,1]$ or first increasing then decreasing on $\beta\in(\beta_4,1]$. In either case, there is only one branch in $S_5$, named $(\beta_5(\nu),\nu)$. We have $\beta_5(\nu)>\beta_4(\nu)$ in $(\beta,\nu)\in[0,1]\times (0,1)$. The proof is now completed.
\end{proof}
\begin{proof}[Proof of Lemma \ref{lem: trace and determinant are positive}-(ii)]
Due to Remark \ref{rem: 0 to ((n+nu)^2-1)/7}, it is sufficient to consider $\beta\in((2\nu+\nu^2)/7,1]$ for every $\nu\in(0,1)$. Let
$D_4^+(e):=e^3\mathrm{det}^+_4(e)=\mathrm{det}_1+\mathrm{det}_2e+\mathrm{det}_3e^2+\mathrm{det}_4e^3.$
For every $\beta\in[0,1]$, we have known that $\mathrm{det}_1>0$. Moreover, we have
$$
\mathrm{det}_2\left\{\begin{aligned}
&>0,\quad \forall\ \beta_1(\nu)<\beta<\beta_3(\nu)\\
&<0,\quad \forall\ \beta_3(\nu)<\beta<\beta_4(\nu)\\
&>0,\quad \forall\
\beta_4(\nu)<\beta\leq 1
\end{aligned}\right.
,\quad
\mathrm{det}_3\left\{\begin{aligned}
&>0,\quad \forall\  0<\beta<\min\{1,\beta_5(\nu)\}\\
&<0,\quad \forall\  \min\{1,\beta_5(\nu)\}<\beta\leq 1.
\end{aligned}\right.
$$
Since $C_0+C_2+C_4=(21\beta+3(1-\nu)(5+\nu))/((-1+\nu)(5+\nu))<0$ for every $\beta\in[0,1]$, we have
$$
\mathrm{det}_4\left\{\begin{aligned}
&>0,\quad \forall\ 0<\beta<\beta_2(\nu)\\
&<0,\quad \forall\ \beta_2(\nu)<\beta<1.
\end{aligned}\right.
$$
From these estimates, we discuss in three cases:
\begin{itemize}
\item[(a)] For every $\beta\in(\beta_1,\beta_2]$, we have $\mathrm{det}_2,\mathrm{det}_3,\mathrm{det}_4\geq 0$. Then $\mathrm{det}^+_4(e)>0$ for every $e\in(0,1)$;
\item[(b)] For every $\beta\in(\beta_2,\beta_3)\cup (\beta_4,1]$, we have $\mathrm{det}_2,-\mathrm{det}_4>0$. Then
$$D^+_4(e)\geq\min\{\mathrm{det}_1,\mathrm{det}^+_4(1)\}.$$
\item[(c)] For every $\beta\in(\beta_3,\beta_4)$, we have $-\mathrm{det}_2,\mathrm{det}_3,-\mathrm{det}_4>0$. Then
$$D^+_4(e)\geq \min\{ D^+_4(e_*),\mathrm{det}^+_4(1)\},$$
where $e_*$ is the local minimal point of $D^+_4(e)$ in $[0,1]$ so that $\partial_eD^+_4(e_*)=0$.
\end{itemize}

Now we aim to prove $\partial_eD^+_4(1)=\mathrm{det}_2+2\mathrm{det}_3+3\mathrm{det}_4>0$ for every $\beta\in[\beta_2,1]$. Let $s=\beta-(4\nu+\nu^2)/7$. Since $\partial_eD^+_4(1)$ is a polynomial of $\beta$ up to order six, we find an explicit splitting for this polynomial as follows:
$$\begin{aligned}
\partial_{e}D^+_4(1)&=\mathrm{det}_2+2\mathrm{det}_3+3\mathrm{det}_4\\
&=10+2\beta\cdot\frac{s^3h_0+s^2(1-s)h_1+\nu(4+\nu)((1-s)^3h_2+s(1-s)^4h_3)}{(1-\nu)\nu^2(1+\nu)^3(2+\nu)^2(3+\nu)^3(4+\nu)^2(5+\nu)},
\end{aligned}$$
where
$$\begin{aligned}
h_1&=7((49392 - 250704 \nu + 400000 \nu^2) + \nu^2 (802172 (1 - \nu) + 101545 \nu^2)\\
&\quad\ +5172 (1 - \nu) \nu^3  + 8 (1 - \nu)^2 \nu^3 (123830 + 119047 \nu + 46828 \nu^2)\\
&\quad\ +\nu^8 (205334 + 62000 \nu + 11020 \nu^2 + 1080 \nu^3 + 45 \nu^4))>0,\\
h_2&=7(8100 + 54324 \nu + 42957 \nu^2 + 14048 \nu^3 + 11596 \nu^4 + 15240 \nu^5 + 10230 \nu^6 + 3840 \nu^7\\
&\quad\  + 840 \nu^8 + 100 \nu^9 + 5 \nu^{10})>0,\\
h_3&=7(102420 - 24264 \nu - 14594 \nu^2 - 44264 \nu^3 - 8773 \nu^4 + 29620 \nu^5 + 26735 \nu^6 + 10960 \nu^7\\
&\quad\  + 2485 \nu^8 + 300 \nu^9 + 15 \nu^{10})>0.
\end{aligned}$$
are positive functions on $\nu\in [0,1]$, and $h_0=h_{00}+h_{01}\beta+h_{02}\beta^2$ is a parabola of $\beta$ with
$$\begin{aligned}
h_{00}&=(1/7)(33882912+249905488\nu+266191156\nu^2+ 139182320\nu^3 +59571566 \nu^4 + 31765912 \nu^5\\
&\quad\ - 164062 \nu^6 - 26219904 \nu^7- 28729104 \nu^8 - 17465140 \nu^9 - 7130937 \nu^{10} - 2080520 \nu^{11} \\
&\quad\ - 442195 \nu^{12} - 67900 \nu^{13} - 7225 \nu^{14} - 480 \nu^{15} - 15 \nu^{16}),\\
h_{01}&=7(6050520 - 12550496 \nu - 9311576 \nu^2 - 3952768 \nu^3 - 3411664 \nu^4 - 1960224 \nu^5 \\
&\quad\ + 1026824 \nu^6 + 2271312 \nu^7 + 1643997 \nu^8 + 711460 \nu^9  + 205765 \nu^{10} + 40680 \nu^{11} \\
&\quad\ + 5335 \nu^{12} + 420 \nu^{13} + 15 \nu^{14}),\\
h_{02}&=2(-2420208 + 665968\nu+ 263548\nu^2 + 82640 \nu^3 + 191650 \nu^4 + 79356 \nu^5- 109707 \nu^6\\
&\quad\  - 136560 \nu^7 - 70575 \nu^8 - 20900 \nu^9 -
 3685 \nu^{10} - 360 \nu^{11} - 15 \nu^{12})<0,\quad \forall \nu\in [0,1].
\end{aligned}$$
Then $\min_{\beta\in[\beta_2,1]} h_0(\beta,\nu)=\min\{h_0(\beta_2,\nu),h_0(1,\nu)\}$ and we compute
$$\begin{aligned}
h_0(\beta_2,\nu)&=14(345744 + 3043976 \nu + 1420050 \nu^2 + 315128 \nu^3 + 167863 \nu^4 + 6772 \nu^5 \\
&\quad\ - 190593 \nu^6 - 194288 \nu^7 - 96263 \nu^8 - 28100 \nu^9 - 4925 \nu^{10} - 480 \nu^{11} - 20 \nu^{12})>0,\\
h_0(1,\nu)&=(\nu/7)(106830976 + 148742944 \nu + 87892928 \nu^2 + 21199120 \nu^3 + 8211220 \nu^4 \\
 &\quad\ + 8835831 \nu^5- 1112976 \nu^6 - 9171321 \nu^7 - 8528800 \nu^8 - 4430792 \nu^9 \\
 &\quad\ - 1528640 \nu^{10} - 368240 \nu^{11} - 62020 \nu^{12} - 7015 \nu^{13} - 480 \nu^{14} - 15 \nu^{15})>0,
\end{aligned}$$
are both positive on $\nu\in [0,1]$. Therefore, $\partial_eD^+_4(1)>0$ has been proved for every $\beta_2\leq \beta\leq 1$. Then $D^+_4(e)>0$ for every $e\in (0,1)$ in case (b). 

Next, we aim to prove $D^+_4(e_*)>0$ for every $\beta\in[\beta_3,\min\{1,\beta_4\}]$, where
$$e_*\in\left\{\frac{-\mathrm{det}_3+\sqrt{\mathrm{det}_3^2-3\mathrm{det}_2\mathrm{det}_4}}{\mathrm{det}_4},\frac{-\mathrm{det}_3- \sqrt{\mathrm{det}_3^2-3\mathrm{det}_2\mathrm{det}_4}}{\mathrm{det}_4}\right\},$$
is the local minimal point of $D^+_4$ solving $\partial_eD^+_4(e)=\mathrm{det}_2+2\mathrm{det}_3e_*+3\mathrm{det}_4e_*^2=0$. For every $\beta\in[\beta_3,\min\{1,\beta_4\}]$, we know that $-\mathrm{det}_4>0,\partial_eD^+_4(1)>0$ and $D^+_4(0)=\mathrm{det}_1>0$ are positive. Hence, $e_*=(-\mathrm{det}_3+(\mathrm{det}_3^2-3\mathrm{det}_2\mathrm{det}_4)^{1/2})/\mathrm{det}_4$, and we compute
that
$$
D^+_4(e_*)=\frac{1}{27\mathrm{det}_4^2}\big(2\mathrm{det}_3^3- 9\mathrm{det}_2\mathrm{det}_3\mathrm{det}_4+ 27\mathrm{det}_1\mathrm{det}_4^2-2(\mathrm{det}_3^2-3\mathrm{det}_2\mathrm{det}_4)^{3/2}\big).
$$
Let $\mathrm{det}_{2c}:=\mathrm{det}_2/(16C_1C_2C_3)=C_0C_1C_2-2C_0C_1C_4-2C_0C_3C_4+C_2C_3C_4$. Denote by
$$\begin{aligned}
D^*_4:&=\frac{1}{27\mathrm{det}_4^2}\big((2\mathrm{det}_3^3- 9\mathrm{det}_2\mathrm{det}_3\mathrm{det}_4+ 27\mathrm{det}_1\mathrm{det}_4^2)^2-4(\mathrm{det}_3^2-3\mathrm{det}_2\mathrm{det}_4)^3\big)\\
&=4\mathrm{det}_1\mathrm{det}_3^3+4\mathrm{det}_2^3\mathrm{det}_4-\mathrm{det}_2^2\mathrm{det}_3^2-18\mathrm{det}_1\mathrm{det}_2\mathrm{det}_3\mathrm{det}_4+
27\mathrm{det}_1^2\mathrm{det}_4^2\\
&=(C_2^2\mathrm{det}_3^2)\cdot D_{41}+(-\mathrm{det}_2\mathrm{det}_3)\cdot D_{42}+(\frac{(C_1+C_3)\mathrm{det}_2^3}{16C_1C_3})\cdot D_{43}+27\mathrm{det}_1^2\mathrm{det}_4^2.
\end{aligned}$$
where
$$\begin{aligned}
D_{41}&=\frac{1}{C_2^2\mathrm{det}_3^2}\big(4\mathrm{det}_1\mathrm{det}_3^3-\mathrm{det}_2^2\mathrm{det}_3^2+76C_1C_2^2(\frac{C_3+C_4}{C_1+C_3})\mathrm{det}_2\mathrm{det}_3^2\mathrm{det}_4\big)\\
&=64C_1^2(4C_3^2(C_0C_4\mathrm{det}_3-\mathrm{det}_{2c}^2)+19C_2C_3(C_3+C_4)(C_0+C_2+C_4)\mathrm{det}_{2c}),\\
D_{42}&=\frac{-1}{\mathrm{det}_2\mathrm{det}_3}\big(\frac{\nu(C_1+C_3)}{16C_1C_3}\mathrm{det}_{2}^3\mathrm{det}_{3}+18\mathrm{det}_{1}\mathrm{det}_{2}\mathrm{det}_{3}\mathrm{det}_{4}-76C_1C_2^2(\frac{C_3+C_4}{C_1+C_3})\mathrm{det}_2\mathrm{det}_3^2\mathrm{det}_4\big)\\
&=76C_1C_2^2(C_3+C_4)(C_0+C_2+C_4)\mathrm{det}_{3}+\nu C_2(C_1+C_3)\mathrm{det}_{2}\mathrm{det}_{2c}+18\mathrm{det}_{1}\mathrm{det}_{4},\\
D_{43}&=\frac{16C_1C_3}{(C_1+C_3)\mathrm{det}_2^3}\big(4\mathrm{det}_2^3\mathrm{det}_4-\frac{\nu(C_1+C_3)}{16C_1C_3}\mathrm{det}_{2}^3\mathrm{det}_{3}\big)
=64C_1C_3(C_0+C_2+C_4)-\nu \mathrm{det}_{3}.
\end{aligned}$$
Recall that $-\mathrm{det}_2,\mathrm{det}_3,C_1,C_1+C_3>0$ for every $\beta\in[\beta_3,\min\{1,\beta_4\}]$. We obtain that $-\mathrm{det}_2\mathrm{det}_3>0$ and $(C_1+C_3)\mathrm{det}_2^3/(C_1C_3)>0$ for every $\beta\in[\beta_3,\min\{1,\beta_4\}]$.
Now we aim to prove $D_{41},D_{42},D_{43}>0$ for every $\beta\in[\beta_2,\min\{1,\beta_4\}]$. Let $b:=\beta-(4\nu+\nu^2)/7\in[0,3/7]$ and $s=1-t:=7b/3\in[0,1]$. 
Note that $\beta\in[\beta_2,\min\{1,\beta_4\}]$ implies $b\in [0,3/7]$. Since $D_{41},D_{42},D_{43}$ are all polynomials of $\beta$, we can manipulate them as
$$\begin{aligned}
D_{41}(b,\nu)&=128C_1^2C_2C_3\cdot \frac{I_0\beta^4+(I_1+(I_2+I_3t)s^4+I_4s^2t^2+(I_5+I_6s)t^4)(1-\beta)}{7^3(1-\nu)^2\nu^2(1+\nu)^3(2+\nu)^3(3+\nu)^3(4+\nu)^3(5+\nu)^3},\\
D_{42}(b,\nu)&=16C_1C_2^2\cdot \frac{(J_0+J_1 s+J_2st)s^3t+(J_3+s^5J_4)t^3+(J_5+J_6s+J_7s^3)st^4 + J_8s^3 t^5}{(1-\nu)^2\nu^2 (1+\nu)^4(2+\nu)^3(3+\nu)^4(4+\nu)^4(5+\nu)^3},\\
D_{43}(b,\nu)&=12\cdot \frac{K_0+K_1b(1-b)+K_2b^3+K_3b^2(1-b)^2}{(1-\nu)\nu(1+\nu)^2 (2+\nu)^2(3+\nu)^2(4+\nu)^2(5+\nu)},
\end{aligned}$$
where $I_i,J_i,K_i$ are positive polynomials on $[0,3/7]\times[0,1]$ listed below. Then we obtain that both $D^*_4>0$ and $D^+_4(e_*)>0$ happen for every $\beta\in[\beta_3,\min\{1,\beta_4\}]$. Finally, we conclude that $D^+_4(e)>0$ for every $e\in[0,1]$ in case (c). Then $D^+_4>0$ for every $\nu\in(0,1)$ and $(\beta,e)\in[0,1]\times [0,1]$.
\end{proof}

To support the previous proof, one can easily check that all the following polynomials $I_i=I_i(b,\nu),J_i=J_i(b,\nu),K_i=K_i(b,\nu)$ are positive for every $(b,\nu)\in [0,3/7]\times(0,1)$.
\begin{align*}
I_0&=343(318988323106 (1 - \nu)^6 \nu^6 + (1 - \nu)^5 \nu^5(709472721174 (1 - \nu)^2 + 32628342299 \nu^2) + (1 - \nu)^3 \nu^3\cdot\\
&\quad\ ((707818267936 (1 -\nu)^4 + 3641094530 \nu^4) (1 - \nu) \nu + (343012672392 (1 -\nu)^6 + 12267971328 \nu^6)) \\
&\quad\ + (1 - \nu)^2 \nu^2 (66265149648 (1 - \nu)^8 + 438255383 \nu^8) + (1 - \nu)\nu(15252878392 (1 - \nu)^6 + 603430459 \nu^6)  \\
&\quad\ + 128*10^8 (0.85^3 - \nu^3)^2(1 -\nu) \nu + 64*10^9 \nu^2 (0.85^2 -\nu^2)^2 + 1896652800 (1 - \nu)^5  + 754785954 \nu^5 \\
&\quad\ + 22371740 \nu^{13} + 25051317 \nu^{14} + 7031231 \nu^{15} + 1105850 \nu^{16} + 104668 \nu^{17} + 5590 \nu^{18} + 130 \nu^{19},\\
I_1&=18522000 (4 + \nu) ((15 b - 4\nu (4 + \nu))^2 + 1147 b^2 + 4\nu^2 (4 + \nu)^2 + b\cdot \nu (4 + \nu)).\\
I_2&=241645420800 (1 - \nu^3) + 1260506753856 \nu( 1- \nu^2) + 2050516055808 \nu^2( 1 - \nu) + 521733557184 \nu^3\\
&\quad\ + 1730182425720 \nu^3 (1 - \nu^2)^3 + 2921955506072 \nu^6 + 6633393718672 \nu^7 + 15753975073825 \nu^8 \\
&\quad\ + 14197563542359 \nu^9 +
 1642563861432 \nu^3 (1 - \nu)^2 + 3446328021553 \nu^{10} + 962775581704 \nu^{11} \\
&\quad\ + 960140343217 \nu^{10} (1 - \nu^2) +
 947144396383 \nu^{10}(1 - \nu^3) + 452003667786 \nu^{10} (1 - \nu^4) \\
&\quad\ + 127612732476 \nu^{10} (1 - \nu^5) + 14037277419 \nu^{10} (1 - \nu^6) +
 5487164043 \nu^{17} + 3457358326 \nu^{18} \\
&\quad\ + 1027169240 \nu^{19} +
 202448765 \nu^{20} + 28122083 \nu^{21} + 2744586 \nu^{22} + 180068 \nu^{23} +
 7150 \nu^{24} + 130 \nu^{25},\\
I_3&=124023312000 + 1216470286080 \nu+ 1204432368354\nu^2 + 99276282000 \nu^3 + 1544534463582 \nu^2 (1 - \nu^2)^2 \\
&\quad\ + 1620812805408\nu^3 (1 -\nu^2)^2 + 9654250620761\nu^6 + 32853763782805\nu^7 + 47722287341434\nu^8 \\
&\quad\ + 40589709344272\nu^9 + 13162363827505 \nu^{10} + 4656337501615 \nu^{11} +  2708100932398 \nu^{10} (1 - \nu^2) \\
&\quad\ + 3230941529560 \nu^{10} (1 - \nu^3) +
 1661051347221\nu^{10} (1 - \nu^4) + 500525897751 \nu^{10} (1 -\nu^5) \\
&\quad\ + 64788105330 \nu^{10} (1 - \nu^6) + 17794421376 \nu^{17} + 12790147447 \nu^{18} +
 3938461781 \nu^{19} + 790914794 \nu^{20} \\
&\quad\ + 111112640 \nu^{21} + 10918674 \nu^{22} + 719102 \nu^{23} + 28600\nu^{24} + 520 \nu^{25},\\
I_4&=\nu (194260323600 + 1041641902980 \nu + 4702581180036\nu^2 +
    9421403814825\nu^3 + 14227219295475\nu^4  \\
&\quad\ + 27898819192772\nu^5+ 55000286797582\nu^6 + 77607181907617 \nu^7 + 73193852329735\nu^8\\
&\quad\  + 29321729717068 \nu^9 + 13095199552939\nu^{10} + 3492436109089\nu^9 (1 -\nu^2) + 6288457544413\nu^9 (1 -\nu^3) \\
&\quad\ + 3639542450967\nu^9 (1 -\nu^4) + 1203800745639\nu^9 (1 -\nu^5) + 187451603343\nu^9 (1 -\nu^6) + 30916447305 \nu^{16}\\
&\quad\  + 28463618350\nu^{17} + 9260743244\nu^{18} + 1911700964\nu^{19} + 272977208\nu^{20} + 27087840\nu^{21} + 1793660 \nu^{22} \\
&\quad\ + 71500 \nu^{23} + 1300\nu^{24})\\
&+ 7 b (37140314400 + 390647305440\nu + 866912460912\nu^2 -  45329796384\nu^3 - 2807177085248\nu^4 \\
&\quad\ - 3813210047396 \nu^5 - 5315060210 \nu^6 + 5558428206018\nu^7 + 7590346635878 \nu^8 + 5250708041890\nu^9 \\
&\quad\ + 1760280008837 \nu^{10} - 237319765741 \nu^{11} - 603266516808 \nu^{12} - 335505829558 \nu^{13} - 97614680521 \nu^{14} \\
&\quad\  - 9434791343 \nu^{15} + 4705949148\nu^{16} + 2554200390\nu^{17} + 664929512\nu^{18} + 111178972\nu^{19} + 12476864\nu^{20}\\
&\quad\  + 914788\nu^{21} + 39780\nu^{22} + 780\nu^{23}),\\
I_5&=\nu^3 (4 + \nu)^3 (5738644800 + 18296792640 \nu + 24292250784 \nu^2 +
22370899520 \nu^3 + 29209611064 \nu^4 \\
&\quad\ + 40145872456 \nu^5 + 35996617304 \nu^6 + 7319255909 \nu^7 +
 566191800 \nu^8 + (5030541720 (1 -\nu^2)\\
&\quad\ + 3542767236 (1 - \nu^3) + 1188542244 (1 - \nu^4) + 140026419 (1 - \nu^5)) \nu^7 + 53365815 \nu^{13} + 30652353 \nu^{14}\\
&\quad\  + 7628317 \nu^{15} + 1141340 \nu^{16}  + 105578 \nu^{17} + 5590 \nu^{18} + 130 \nu^{19}), \\
I_6&=\nu^2 (4 + \nu)^2 (6215850900 + 111272090580\nu + 276773602533 \nu^2 + 366947626719 \nu^3 + 500138559512 \nu^4 \\
&\quad\ + 839451658266 \nu^5 + 1108468475510 \nu^6 + 945649109766\nu^7 + 264954806710 \nu^8 + 4581438074 \nu^9\\
&\quad\ + 107946065802 \nu^8 (1 - \nu^2) + (87928518714 + 34802068134 \nu)\nu^8 (1 - \nu^3) + 6617278746 \nu^9 (1 - \nu^4)\\
&\quad\ + 552146583 \nu^{14}  + 770745393 \nu^{15} + 254828074 \nu^{16} + 49722180 \nu^{17} + 6304918 \nu^{18} + 512922\nu^{19}\\ &\quad\ + 24440 \nu^{20} + 520\nu^{21}),\\
J_0&=13824(3+4\nu+\nu^2)^4(928+1576\nu+1536 \nu^2+803\nu^3+197\nu^4+18\nu^5),\\
J_1&=1728(3+4\nu+\nu^2)^3 (103872 + 654896\nu+ 1184224\nu^2 + 1039034 \nu^3 + 503002 \nu^4 + 136847 \nu^5 \\
&\quad\ + 19585 v^6 + 1146 v^7),\\
J_2&=2592 (3 + 4\nu+\nu^2)^2 (131304 + 3426438\nu+ 12294878\nu^2 + 19264450\nu^3 + 16666402\nu^4 + 8680667\nu^5 \\
&\quad\ + 2791869\nu^6 + 543001\nu^7 + 58569 \nu^8 + 2690\nu^9),\\
J_3&=30780\nu^2 (1 +\nu)^2 (4 +\nu)^5 (375 + 1415\nu + 1171\nu^2 + 361\nu^3 +  38\nu^4),\\
J_4&=81(41508288 + 1862770464\nu+ 11951373816 \nu^2 + 35330437075\nu^3 + 61279679097 \nu^4 \\
&\quad\ + 69169678653\nu^5 + 53583593209\nu^6 + 29269308867 \nu^7 + 11382733755\nu^8 + 3134375639\nu^9 \\
&\quad\ + 597182883 \nu^{10} + 74871798 \nu^{11} + 5556344\nu^{12} + 184896\nu^{13}),\\
J_5&=108 \nu(4 +\nu)^3 (1593000 + 16578225\nu + 54224115\nu^2 + 88491266\nu^3 + 84139780 \nu^4 + 49615937\nu^5 \\
&\quad\ + 18377575\nu^6 + 4152092\nu^7 + 521722\nu^8 + 27888\nu^9),\\
J_6&=81(4 +\nu)^2 (432000 + 46686600\nu + 325921440\nu^2 + 945321835\nu^3 + 1512906297\nu^4 + 1500039736\nu^5 \\
&\quad\ + 971849206 \nu^6 + 418953983 \nu^7 + 119063577\nu^8 + 21403542\nu^9 + 2202968\nu^{10} + 98816\nu^{11}),\\
J_7&=162(41333328 + 1867391352\nu + 11970628812 \nu^2 + 35356663039 \nu^3 +
 61297005537\nu^4 + 69175942653 \nu^5 \\
&\quad\ + 53584863181\nu^6 + 29269445055\nu^7 + 11382739803 \nu^8 + 3134375639 \nu^9 +597182883 \nu^{10} \\
&\quad\ + 74871798\nu^{11} + 5556344 \nu^{12} + 184896 \nu^{13}),\\
J_8&=81(42208128 + 1872338832\nu+ 11989210536 \nu^2 + 35382594163\nu^3 + 61314411033 \nu^4 + 69182267565\nu^5 \\
&\quad\ + 53586144169\nu^6 + 29269581891 \nu^7 + 11382745851 \nu^8 + 3134375639 \nu^9 +597182883\nu^{10} \\
&\quad\ + 74871798 \nu^{11} + 5556344 \nu^{12} + 184896 \nu^{13}).\\
K_0&=4\nu(4 + \nu)^2 (720 + 1845\nu+ 1532\nu^2 + 583\nu^3 + 80 \nu^4),\\
K_1&=56 (2880 + 9252\nu + 11189\nu^2 + 8264 \nu^3 + 3644\nu^4 + 932\nu^5 + 131\nu^6 + 8 \nu^7),\\
K_2&=56 (2064 + 12638\nu+ 20831\nu^2 + 15884 \nu^3 + 6854\nu^4 + 1766\nu^5 + 255\nu^6 + 16 \nu^7) \\
&\quad\ - 28 b (-288 + 14108\nu+ 27250 \nu^2 + 20000 \nu^3 + 7883 \nu^4 + 1864 \nu^5 + 255 \nu^6 + 16 \nu^7),\\
K_3&=7 (31776 + 82500 \nu+ 121005\nu^2 + 82744 \nu^3 + 31875 \nu^4 + 7456 \nu^5 +
1020\nu^6 + 64 \nu^7).
\end{align*}
In particular, the following polynomials are positive for every $\nu\in(0,1)$.
\begin{align*}
I_4(\frac{3}{7},\nu)&=111420943200 + 1366202239920\nu + 3642379285716 \nu^2 +
 4566591790884 \nu^3 + 999872559081 \nu^4 \\
&\quad\ + 2787589153287 \nu^5 +
 27882874012142 \nu^6 + 71675571415636 \nu^7 + 100378221815251 \nu^8 \\
&\quad\ + 88945976455405 \nu^9 + 31479222136333 \nu^{10} + 12383240255716 \nu^{11} +
 5302235659513 \nu^{10} (1 - \nu^2) \\
&\quad\ + 7294975033087 \nu^{10} (1 -\nu^3) +
 3932386492530 \nu^{10} (1 - \nu^4) + 1232105119668 \nu^{10} (1 - \nu^5) \\
&\quad\ +
 173333755899 \nu^{10} (1 - \nu^6) + 38579048475 \nu^{17} + 30458406886 \nu^{18} +
 9594280160 \nu^{19} \\
&\quad\ + 1949131556 \nu^{20} + 275721572 \nu^{21} + 27207180\nu^{22} +
 1796000 \nu^{23} + 71500 \nu^{24} + 1300 \nu^{25},\\
K_2(\frac{3}{7},\nu)&=4 (29760 + 134608 \nu + 209884 \nu^2 + 162376 \nu^3 + 72307 \nu^4 + 19132 \nu^5 + 2805 \nu^6 + 176 \nu^7).
\end{align*}

\hfill\newline
\noindent{\bf Acknowledgment.}
XH and ZQ are partially supported by the National Natural Science Foundation of China (Grant number: 12521001) and the Taishan Scholars Climbing Program of Shandong Province (Grant number: TSPD20240802). LL is partially supported by the National Natural Science Foundation of China (Grant number: 12401238) and the Natural Science Foundation of Shandong Province (Grant number: ZR2024QA188). YO is partially supported by the National Natural Science Foundation of China (Grant number: 12371192), the Young Taishan Scholars Program of Shandong Province (Grant number: tsqn202312055), and the Qilu Young Scholar Program of Shandong University. XH, LL, YO and ZQ thank the support of the School of Mathematics at Shandong University. PS is partially supported by the National Natural Science Foundation of China (Grant number: w2431007). LL and PS thank the support of the Shenzhen International Center for Mathematics - SUSTech.

\end{document}